\documentclass[11pt]{amsart}
\usepackage{amsfonts, amsmath, amssymb, amsthm, stmaryrd, color, enumerate, relsize}
\usepackage{bbm}
\usepackage{accents}
\usepackage{mathrsfs,array}
\usepackage{xy}
\usepackage{hyperref}
\usepackage{tikz-cd}
\usepackage{cleveref}
\input xy
\xyoption{all}
\DeclareMathOperator{\Spa}{Spa}
\DeclareMathOperator{\Spd}{Spd}
\DeclareMathOperator{\GL}{GL}

\DeclareMathOperator{\Sym}{Sym}
\DeclareMathOperator{\Tot}{Tot}
\DeclareMathOperator{\Tr}{Tr}
\DeclareMathOperator{\colim}{colim}
\DeclareMathOperator{\Hom}{Hom}
\DeclareMathOperator{\Ad}{Ad}
\DeclareMathOperator{\End}{End}
\DeclareMathOperator{\Spec}{Spec}
\DeclareMathOperator{\Proj}{Proj}
\DeclareMathOperator{\Bun}{Bun}
\DeclareMathOperator{\Gr}{Gr}
\DeclareMathOperator{\Perf}{Perf}
\DeclareMathOperator{\sPerfd}{sPerfd}
\DeclareMathOperator{\Sch}{Sch}
\DeclareMathOperator{\Grpd}{Grpd}
\DeclareMathOperator{\Set}{Set}
\DeclareMathOperator{\Aut}{Aut}
\DeclareMathOperator{\Mor}{Mor}
\DeclareMathOperator{\Frac}{Frac}
\DeclareMathOperator{\Kos}{Kos}
\DeclareMathOperator{\Rep}{Rep}
\DeclareMathOperator{\fppf}{fppf}

\newcommand{\Sp}{\mathrm{Gr}_{G,S/Div^1}}
\newcommand{\id}{\mathrm{id}}
\newcommand{\set}[1]{\left\{ #1 \right\}}

\newcommand{\powerseries}[1]{[\![ #1 ]\!]}

\newcommand{\et}{{\mathrm{\acute{e}t}}}

\renewcommand*{\diamond}{\diamondsuit}
\renewcommand*{\hat}{\widehat}
\renewcommand*{\tilde}{\widetilde}

\def\injects{\hookrightarrow}
\def\isomto{\stackrel{\sim}{\to}}
\def\Z{\mathbb{Z}}
\def\Q{\mathbb{Q}}
\def\F{\mathbb{F}}
\def\O{\mathcal{O}}
\def\E{\mathcal{E}}
\def\L{\mathcal{L}}
\def\P{\mathcal{P}}
\def\V{\mathcal{V}}
\def\g{\mathfrak{g}}
\def\c{\mathfrak{c}}
\def\iso{\xrightarrow{\sim}}
\def\Div{\mathrm{Div}}
\def\Higgs{\mathrm{Higgs}_G}
\def\Map{\underline{\mathrm{Map}}}
\def\Isom{\underline{\mathrm{Isom}}}

\newtheorem{theorem}{Theorem}
\numberwithin{theorem}{section}
\newtheorem{lemma}[theorem]{Lemma}
\newtheorem{corollary}[theorem]{Corollary}
\newtheorem{proposition}[theorem]{Proposition}

\newtheorem{definition}[theorem]{Definition}

\theoremstyle{definition}
\newtheorem{remark}[theorem]{Remark}
\newtheorem{example}[theorem]{Example}

\DeclareEmphSequence{\bfseries\itshape}

\date{\today}

\title{Higgs bundles on the Fargues-Fontaine curve}
\author{Ho Leung Fong}
\begin{document}

\begin{abstract} In this paper, we introduce a notion of Higgs bundles on the Fargues-Fontaine curve. We establish a version of the BNR correspondence, which relates Higgs bundles to line bundles on suitable curves. We then describe an action of a Picard stack on the moduli stack of Higgs bundles and show that, modulo this action, there is a natural injective map of \'etale-stacks from the product of $B_{dR}^+$-affine Springer fibers to the Hitchin fiber that induces an equivalence of categories on every geometric point. Finally, we discuss connections with number-theoretic objects.
\end{abstract}

\maketitle

\tableofcontents

\section{Introduction}
Higgs bundles on algebraic varieties play an important role in several areas of mathematics.
To motivate this paper, let us give two motivations from the Langlands program for studying Higgs bundles. We will be slightly imprecise to streamline the discussion. 
\begin{enumerate}
    \item \textbf{Langlands duality}\\
    Let $X$ be a smooth projective curve over $\mathbb C$.
    It has been conjectured that (cf. \cite[Conjecture 2.5]{Donagi_Pantev}) there is an equivalence of categories\footnote{This is just an approximation. It is unknown precisely what the two categories should be.}
    \[\mathrm{QCoh}(\Higgs) \simeq \mathrm{QCoh}(\mathrm{Higgs}_{\check{G}}),\]
    where $G$ is a connected reductive group over $\mathbb C$, $\check{G}$ is its dual group, and $\Higgs$ is the moduli stack of $G$-Higgs bundles on $X$.
    This is highly interesting because it is an instance of Langlands duality: it involves both $G$ and $\check G$.
    Even for $G=\GL_n$, this is a non-trivial assertion, because this equivalence is expected to be compatible with certain Fourier-Mukai equivalence, and in particular is not satisfied by the identity functor.
    To be less vague, there is a Hitchin fibration $\Higgs\to \mathcal A_G$.
    By the BNR theorem \cite[Proposition 3.6]{BNR}, the fibres of this map over a suitable locus on $\mathcal A_G$ are given by the Jacobian of the spectral curves (up to some stacky parts).
    In particular, there is a Fourier-Mukai equivalence \cite{Fourier-Mukai} on the derived category of quasi-coherent sheaves on the fibres.
    The conjectured equivalence above should restrict to this Fourier-Mukai equivalence. 
    
    \item \textbf{Proof of the fundamental lemma}\\
    One of the central tools in the study of the Langlands conjectures is the trace formula. Most of its applications rely on a deep and highly nontrivial input known as the fundamental lemma, which relates local orbital integrals attached to two different reductive groups. By \cite{Waldspurger}, to prove the fundamental lemma for all non-Archimedean local fields, it is enough to prove it for non-Archimedean local fields with equal characteristic.
    In \cite{Ngo_2010}, Ngô achieved this using the geometry of Higgs bundles.

    The guiding principle is a global-to-local philosophy. Locally, orbital integrals can be interpreted in terms of point counts on affine Springer fibers, certain closed subsheaves of the affine Grassmannians. Globally, the Hitchin fiber in the moduli stack $\Higgs$ serves as a global avatar of affine Springer fibers. 
    The reason is that there are certain Picard stacks acting on $\Higgs$ and the affine Springer fibers.
    Ngo proves in \cite[Proposition 4.15.1]{Ngo_2010} that modulo the action of these Picard stacks, the Hitchin fiber factorises as a product of affine Springer fibers over all places. This product formula makes precise the global-local relationship and allows one to deduce identities between orbital integrals from global geometric properties of the Hitchin fibration.

    It is worth emphasizing that Ngô works with a variant of the usual notion of Higgs bundles: instead of the canonical line bundle $\Omega_X^1$, one allows an arbitrary line bundle on $X$. This additional flexibility plays an essential role in the geometric arguments.
\end{enumerate}

In this paper, we introduce a notion of Higgs bundles on the fundamental curve of $p$-adic Hodge theory, namely the Fargues-Fontaine curve. Introduced by Laurent Fargues and Jean-Marc Fontaine, this curve provides a geometric framework for studying $p$-adic Hodge theory and the Local Langlands Program. In particular, the study of $G$-bundles on the Fargues-Fontaine curve and $\ell$-adic sheaves on their moduli stacks plays a central role in the geometric approach to local Langlands developed in \cite{FS_geometrisation}.

From this perspective, the local Langlands program can be viewed as a form of geometric Langlands theory on the Fargues-Fontaine curve. It is therefore natural to extend geometric constructions from the classical setting of smooth projective curves to this $p$-adic context. In this paper, we carry out such an extension for the notion of Higgs bundles.

In contrast to the classical setting, the Fargues-Fontaine curve is an adic space\footnote{There is also a schematic version of the Fargues-Fontaine curve, but that is also not an algebraic variety.}, rather than an algebraic variety. As a consequence, certain basic tools from algebraic geometry are not available. In particular, the notion of Kähler differentials does not exist for the Fargues-Fontaine curve, which prevents one from defining Higgs bundles in the usual way. We thus adapt Ngo's definition of Higgs bundles, i.e. we fix a line bundle $\L$ on the Fargues-Fontaine curve $X_S$ and define a ($\L$-twisted $G$-)Higgs bundle on $X_S$ to be a pair $(\E,\phi)$, where
\begin{itemize}
    \item $\E$ is a $G$-bundle on $X_S$,
    \item $\phi\in H^0(X_S, \Ad(\E)\otimes_{\O_{X_S}}\L)$.
\end{itemize}
We let $\Higgs$ denote the moduli stack of $\L$-twisted $G$-Higgs bundles.
We also introduce in \Cref{defn: asf 1} a local analogue of $\Higgs$, which we call $B^+_{dR}$-affine Springer fiber. 

We prove three abelianisation theorems in this paper.
As we noted in motivation 1, to formulate motivation 1 for $\GL_n$, we need to know that the Hitchin fibers are the Jacobians of certain curves (up to some stacky parts), so that we have Fourier-Mukai equivalence on them. 
In \Cref{BNR}, we prove a version of the BNR correspondence, which relates $\GL_n$-Higgs bundles to line bundles on suitable curves.\footnote{These will be the spectral curves. In the same way that the Fargues-Fontaine curve is not a curve in the usual sense, the spectral curve is also not a curve in the usual sense. However, the spectral curve will be a finite covering of the Fargues-Fontaine curve.}
This can be viewed as progress towards formulating an analogue of motivation 1 for the Fargues-Fontaine curve.
In \Cref{Higgs^reg is a gerbe}, we prove that $$\Higgs^{reg}\cong \mathcal P,$$ where $\Higgs^{reg}$ is a certain sub-prestack of $\Higgs$ and $\mathcal P$ is a certain Picard prestack which classifies $J_\L$-bundles on the Fargues-Fontaine curve, for some commutative group scheme $J_\L$ over the Hitchin base.
Similarly, in \Cref{Gr^reg equiv to P^loc}, we prove that $$\Gr_{G,a,v}^{reg}\cong \mathcal P_{a_v}^{loc},$$ where $\Gr_{G,a,v}^{reg}$ is a certain sub-v-sheaf of the $B^+_{dR}$-affine Springer fiber $\Gr_{G,a,v}$ and $\mathcal P_{a_v}^{loc}$ is the $B^+_{dR}$-affine Grassmannian of a commutative group scheme $J_{a_v}$.

In \Cref{product-formula}, we exhibit the close relationship between the Hitchin fibers $\mathrm{Higgs}_a$ and the affine Springer fibers $\Gr_{G,a,v}$ by showing that for generically regular semisimple $a$, there is a natural injective morphism of \'etale-stacks 
$$\prod_{\text{closed }v\in |X_C^{sch}|} [\Gr_{G,a,v}/\mathcal P_{a,v}^{loc}]_\et \to [\mathrm{Higgs}_a/\mathcal P_a]_\et,$$ which is an equivalence on every geometric point. This can be viewed as progress towards motivation 2 for arbitrary non-Archimedean local fields.
Finally, we will discuss connections with number-theoretic objects.

\subsection*{Convention}
If $\mathcal C$ is a category, then by a prestack on $\mathcal C$, we mean a 2-functor $\mathcal C^{op}\to \mathrm{Grpd}$, where $\Grpd$ is the (2,1)-category of (small) groupoids.
For a scheme or adic space $X$, we denote by $|X|$ its underlying topological space, not just the closed points. Unless otherwise stated, by a stack, we mean a stack in groupoids. By a vector bundle on $X$, we mean a locally finite free $\O_X$-module on $X$.
Unless otherwise stated, our $G$-torsors will be \emph{right} $G$-torsors and all other (group) actions will be left actions.

If $\mathcal C$ is a site, $X\in\mathcal C$ and $G$ is a group sheaf on $\mathcal C$, then a $G$-torsor/$G$-bundle on $X$ is a sheaf $\E$ on $\mathcal C_{/X}$ equipped with a right $G|_X$-action on it for which there is a cover $\set{U_i\to X}$ and  $G|_{U_i}$-equivariant isomorphisms $\E|_{U_i}\cong G|_{U_i}$, where $G|_{U_i}$ acts on itself via right multiplication.
If $\mathcal F$ is a sheaf on $\mathcal C$ with a left $G$-action, then we define $\E\times^G \mathcal F$ to be the sheafification of $U\mapsto (\E(U)\times \mathcal F(U))/G(U)$ where $G(U)$ acts as $g\cdot(x,y)=(xg^{-1},gy)$.
We will denote the trivial $G$-bundle by $\E_0$ or $G$.

We let $\Perf, \Perf_{\overline\F_q}, \Perf_S$ be the category of perfectoid spaces over $\F_p$, $\overline{\F_q}$, $S$ respectively.
We let $X_S$ denote the Fargues-Fontaine curve.
If $S$ is affinoid perfectoid, then we let $X_S^{sch}$ denote the schematic Fargues-Fontaine curve.\footnote{In \cite{FS_geometrisation}, this is denoted as $X_S^{alg}$.}
If $\mathcal V$ is a vector bundle on the Fargues-Fontaine curve $X_S$, we let $\mathcal H^0(\mathcal V)$ be the functor on $\Perf_S$ given by $T\mapsto H^0(X_T,\mathcal V)$.
We let $\mathbb D_T$ denote $\Spec B^+_{Div^1}(T)$ and $\mathbb D_T^*$ denote $\Spec B_{Div^1}(T)$. Note that this depends on the map $T\to Div^1_X$, not just on $T$.

If $A$ is a Huber ring, then $\Spa A :=\Spa(A,A^\circ)$.

If $X$ is an affine scheme, then $\O(X)$ is the ring of global sections of its structure sheaf $\O_X$.

Unless otherwise stated, by `ring', we mean a commutative ring.

\subsection*{Acknowledgements}
I would like to thank my supervisor Tobias Berger for providing feedback of drafts of this paper, James Newton and Robert Kurinczuk for pointing out various errors in a previous version of this paper, and Ken Lee for providing the proof of \Cref{lem: exact criterion}.
This work was completed while in receipt of the EPSRC Doctoral Training Partnership (DTP) Studentship with grant number EP/W524360/1.

\section{Background}
In this section, we provide some background useful for this paper. While most of these materials are standard, there are some new results which do not seem to exist in the literature, including \Cref{inertia gives action}, \Cref{inertia gives action over a base}, \Cref{lem: exact criterion}, \Cref{lem: exact 1}, \Cref{exactness for BL}, \Cref{lem: exactness on B_dR}.

\subsection{Recollections on stacks}
Let $\mathcal C$ be a site. We shall mainly apply this to the fppf site of schemes and the \'etale site of sousperfectoid spaces as in the appendix.

If $\mathcal F$ is a sheaf on $\mathcal C$ and $X\in \mathcal C$, then we abbreviate $\mathcal F|_{\mathcal C_{/X}}$ as $\mathcal F|_X$.

Let $G$ be a group sheaf on $\mathcal C$.
\begin{definition}
    For $X\in \mathcal C$, a \emph{$G$-torsor/$G$-bundle} on $X$ is a sheaf $\E$ on $\mathcal C_{/X}$ equipped with a right $G|_X$-action on it for which there is a cover $\set{U_i\to X}$ and  $G|_{U_i}$-equivariant isomorphisms $\E|_{U_i}\cong G|_{U_i}$, where $G|_{U_i}$ acts on itself via right multiplication.
\end{definition}

We recommend the readers to read \cite[Appendix to Lecture XIX]{Berkeley} for some more discussions on $G$-torsors in the settings of schemes and adic spaces.

\begin{definition}\label{defn: quotient stack}
    Let $\mathcal F$ be a sheaf on $\mathcal C$ acted on by $G$. We define $[\mathcal F/G]$ to be the stack
    \begin{align*}
        \mathcal C^{op} &\to \mathrm{Grpd}\\
        S &\mapsto (\E \text{ a $G$-torsor on $S$}, f:\E\to \mathcal F|_S \textrm{ a $G$-equivariant map}).
    \end{align*}
\end{definition}
Since by our convention $\E$ and $\mathcal F$ receive right and left $G$-actions respectively, by `$G$-equivariant', we mean $f(eg)=g^{-1}f(e)$.

Let us recall a well-known lemma.
\begin{lemma}\label{lem:quotient stack}
    $[\mathcal F/G]$ is the stackification of the prestack $[\mathcal F/_{pre}\; G]: \mathcal C^{op} \to \mathrm{Grpd}$ sending $S$ to the quotient groupoid\footnote{The groupoid whose objects are elements of $X(S)$ with morphisms given by $\Mor(x,y)=\set{g\in G(S):gx=y}$.} $\mathcal F(S)/G(S)$.
\end{lemma}

\begin{proof}
    By \cite[04TQ, 04TR]{stacks-project}, $[\mathcal F/G]$ is indeed a stack. 
    We have a map of prestacks $[\mathcal F/_{pre}\; G] \to [\mathcal F/G]$ given on $S$-valued points by
    \begin{align*}
        \mathcal F(S)/G(S) &\to [\mathcal F/G](S)\\
        x \in \mathcal F(S) &\mapsto (\text{trivial $G$-bundle on }S, \; g\mapsto g^{-1}x)\\
        g \in G(S) &\mapsto g\cdot  
    \end{align*}
    It is easy to verify using \cite[Lemma 02ZN]{stacks-project} that this map exhibits $[\mathcal F/G]$ as the stackification of $[\mathcal F/_{pre}\; G]$.
\end{proof}

Let us also recall a well-known description of the sections of associated bundles via $G$-equivariant maps.
\begin{lemma}\label{alt_defn_for_quotient_stack}
    Let $\E$ be a $G$-torsor on $X\in \mathcal C$. We have a bijection between
    \begin{enumerate}
        \item $G$-equivariant maps $\E \to \mathcal F|_X$ and
        \item elements of $(\E\times^G \mathcal F|_X)(X).$
    \end{enumerate}
    In particular, we can identify $[\mathcal F/G]$ with the stack $S \mapsto (\E \text{ a $G$-torsor on $S$}, s \text{ an element of }(\E\times^G \mathcal F|_X)(X)).$
\end{lemma}

\begin{proof}
    By replacing $\mathcal C$ by $\mathcal C|_X$, we can assume $X$ is the final object of $\mathcal C$.
    Let $f:\E\to \mathcal F$ be a $G$-equivariant map.
    Then we get a $G$-equivariant map $(\id,f):\E\to \E\times \mathcal F$, and hence a map
    \[X=\E/G \to (\E\times \mathcal F)/G= \E\times^G \mathcal F,\]
    \footnote{The canonical map $\E/G\to X$ is an isomorphism locally, and hence is an isomorphism by \cite[04TQ]{stacks-project}.}i.e. an element of $(\E\times^G \mathcal F)(X).$

    Conversely, let $s\in (\E\times^G \mathcal F)(X).$
    By abuse of notation, we let $s$ also denote the composite $\E \to X \to \E\times^G\mathcal F$, where the second map corresponds to the original $s$.
    We get a $G$-equivariant map
    \[\E \xrightarrow{(id,s)} \E\times (\E\times^G\mathcal F) \xrightarrow{r}\mathcal F\]
    where $r$ is the sheafification of the map (as $U\in \mathcal C$ vary)
    \begin{align*}
        \E(U) \times (\E(U)\times \mathcal F(U))/G(U) &\to \mathcal F(U)\\
        (e,[e,g]&\mapsto g).
    \end{align*}
    When defining $r$, we have used the fact that $G(U)$ acts on $\E(U)$ freely. This follows from the isomorphism $G\times \E \xrightarrow{(pr_2,\text{action})}\E\times \E$, which holds because by \cite[04TQ]{stacks-project}, we can check this \'etale locally on $X$ and hence assume $\E$ is the trivial $G$-torsor, for which the isomorphism is clear.

    It is straightforward to check that these two maps are inverses of each other.
\end{proof}

We will also use the following lemma, whose proof is evident, without explicitly mentioning it.
\begin{lemma}\label{fiber product}
    Suppose we have 
    \[\begin{tikzcd}
        & X \\
        Y & Z
        \arrow["f", from=1-2, to=2-2]
        \arrow["g", from=2-1, to=2-2]
    \end{tikzcd}\]
    where $X,Y,Z$ are groupoids and $f,g$ are functors. Assume that for each $x\in X$ and map $f(x)\xrightarrow{\beta} z$, there is a map $x\xrightarrow{\alpha} x'$ such that $f(\alpha)=\beta$. (I.e. $f$ satisfies the `path lifting property'.) Then the 2-fiber product\footnote{The groupoid whose objects are $(x,y,f(x)\iso g(y))$ and morphisms $(x,y,f(x)\iso g(y))\to (x',y',f(x')\iso g(y'))$ are pairs $(x\xrightarrow{\alpha}x', y\xrightarrow{\alpha'}y')$ such that the obvious diagram commutes.} and the 1-fiber product\footnote{The groupoid whose objects are $(x,y)$ such that $f(x)=g(y)$ and morphisms $(x,y)\to (x',y')$ are pairs $(x\xrightarrow{\alpha}x', y\xrightarrow{\alpha'}y')$ such that $f(\alpha)=g(\alpha')$.} of the diagram are canonically equivalent.
\end{lemma}

\subsection{Picard stack}\label{sec: picard stacks}
Recall that a \emph{Picard groupoid} is by definition a symmetric monoidal groupoid $G$ for which every element $x\in G$ admits an inverse $x^{-1}\in G$ such that $x\otimes x^{-1}\cong 1$, where $1$ is the monoidal unit. For example, if $X$ is a scheme, then the groupoid of line bundles on $X$ equipped with the tensor product structure is a Picard groupoid. One can think of Picard groupoids as categorifications of abelian groups.

Let us recall the following definition in \cite[Definition 7.1.2]{tensor_cats}\footnote{There, $X$ is called a left module category over $G$.}:
\begin{definition}\label{action defn}
    Let $X$ be a groupoid and $G$ be a Picard groupoid. An \emph{action} of $G$ on $X$ is a functor $\otimes: G\times X \to X$ together with two natural isomorphisms $m_{g,h,x}:(g\otimes h)\otimes x \iso g\otimes(h\otimes x)$ and $l_x:1\otimes x \iso x$ such that the obvious diagrams
\[\begin{tikzcd}
	& {((g\otimes h)\otimes k)\otimes x} \\
	{(g\otimes (h\otimes k))\otimes x} && {(g\otimes h)\otimes (k\otimes x)} \\
	{g\otimes ((h\otimes k)\otimes x)} && {g\otimes (h\otimes (k\otimes x))}
	\arrow[from=1-2, to=2-1]
	\arrow[from=1-2, to=2-3]
	\arrow[from=2-1, to=3-1]
	\arrow[from=2-3, to=3-3]
	\arrow[from=3-1, to=3-3]
\end{tikzcd}\]
and
\[\begin{tikzcd}
	{(g\otimes 1)\otimes x} && {g\otimes (1\otimes x)} \\
	& {g\otimes x}
	\arrow[from=1-1, to=1-3]
	\arrow[from=1-1, to=2-2]
	\arrow[from=1-3, to=2-2]
\end{tikzcd}\]
commute for all $g,h,k\in G$ and $x\in X$.
\end{definition}

In this case, we can form the $2$-categorical quotient of $[X/G]$. It is a $2$-category which can be explicitly described as follows (cf. \cite[end of section 4]{Ngo_2006}):
\begin{itemize}
    \item The objects are objects of $X$.
    \item The $1$-morphisms from $x_1$ to $x_2$ are pairs $(g,\alpha)$, where $g\in G$ and $\alpha\in Mor_X(gx_1,x_2)$. 
    \item The $2$-morphisms from $(g,\alpha)$ to $(g',\alpha')$ are $\phi\in Mor_G(g,g')$ such that 
\[\begin{tikzcd}
	{gx_1} & {x_2} \\
	{g'x_1}
	\arrow["\alpha", from=1-1, to=1-2]
	\arrow["{\phi\otimes id_{x_1}}"', from=1-1, to=2-1]
	\arrow["{\alpha'}"', from=2-1, to=1-2]
\end{tikzcd}\]
commutes.
\end{itemize}

It turns out that in the case of interest in this paper, every $2$-quotient is equivalent to a $1$-groupoid.
\begin{lemma}\label{2-quot equiv to 1-cat}\cite[Lemme 4.7]{Ngo_2006}
    Let $X$ be a groupoid and $G$ be a Picard groupoid acting on $X$. Then the $2$-quotient $[X/G]$ is equivalent to a $1$-groupoid if and only if the natural map induced by the action
    \begin{equation*}
        \Aut_G(1_G) \to \Aut_X(x)
    \end{equation*}
    is an injection for every $x\in X$.
\end{lemma}

\begin{proof}
    It is evident from the explicit description above that the $2$-quotient $[X/G]$ is equivalent to a $1$-groupoid if and only if for every $g\in G, x\in X$, the natural map 
    \begin{equation*}
        \Aut_G(g) \to \Aut_X(gx)
    \end{equation*}
    is an injection.
    This is equivalent to the condition in the statement of the lemma, because we have the following commutative diagram
\[\begin{tikzcd}
	{ \Aut_G(g)} & {\Aut_X(gx)} \\
	{\Aut_G(1_G)}
	\arrow[from=1-1, to=1-2]
	\arrow["{g\otimes-}", from=2-1, to=1-1]
	\arrow[from=2-1, to=1-2]
\end{tikzcd}\]
where the vertical map is an isomorphism.
\end{proof}

Now let $\mathcal C$ be a site. Let $*$ be the terminal object in the topos.
Note that for any stack $\mathcal P$ on $\mathcal C$, an element in $\mathcal P(*):=\Mor(*,\mathcal P)$ is the same as a compatible collection of objects in $\mathcal P(U)$ as $U\in \mathcal C$ varies.
\begin{definition}
    A \emph{Picard prestack} on $\mathcal C$ is a prestack $\mathcal P$ on $\mathcal C$ together with 
    \begin{itemize}
        \item a morphism $\otimes: \mathcal P \times \mathcal P \to \mathcal P$,
        \item a 2-morphism\footnote{Here, both $\otimes \circ (\otimes \times \id)$ and $\otimes \circ (\id \times \otimes)$ are morphisms $\mathcal P\times \mathcal P \times \mathcal P \to \mathcal P$. A map between these morphisms is thus a 2-morphism in the category of stacks.} $\otimes \circ (\otimes \times \id) \Rightarrow \otimes \circ (\id \times \otimes)$,
        \item a 2-morphism\footnote{Here, $\mathrm{flip}$ is the morphism $\mathcal P\times \mathcal P \to \mathcal P\times \mathcal P$ given by $(x,y)\mapsto (y,x)$.} $\otimes \Rightarrow \otimes\circ \mathrm{flip}$,
        \item an element $1\in \mathcal P(*)$,
        \item a 2-morphism $1\otimes \id \Rightarrow \id$,
        \item a 2-morphism $\id \otimes 1 \Rightarrow \id$,
    \end{itemize}
    such that for every $U\in \mathcal C$, these data make $\mathcal P(U)$ a Picard groupoid.
\end{definition}

\begin{remark}
    One should think of a Picard prestack as a `commutative group prestack', i.e. a prestack $\mathcal P$ with a group operation $\otimes$ such that $x\otimes y = y\otimes x$. However, it is unnatural to assume equality holds. One should instead have an isomorphism $x\otimes y \isomto y\otimes x$ for every $x,y$. Moreover, these isomorphisms should be compatible. The definition above formalises this idea.
\end{remark}

\begin{definition}
    Let $X$ be a prestack and $\mathcal P$ be a Picard prestack. An \emph{action} of $\mathcal P$ on $X$ is 
    \begin{itemize}
        \item a morphism $\otimes:\mathcal P\times X \to X$,
        \item a 2-morphism\footnote{Both $\otimes \circ (\otimes \times \id_X)$ and $\otimes \circ (\id_{\mathcal P} \times \otimes)$ are morphisms $\mathcal P\times\mathcal P\times X \to X$.} $\otimes \circ (\otimes \times \id_X) \Rightarrow \otimes \circ (\id_{\mathcal P} \times \otimes)$
        \item a 2-morphism $1_{\mathcal P}\otimes \id_X \Rightarrow \id_X$
    \end{itemize}
    such that for every $U\in \mathcal C$, these data define an action of $\mathcal P(U)$ on $X(U)$ in the sense of \Cref{action defn}.
\end{definition}

Let $J$ be a commutative group sheaf on $\mathcal C$.
Since $J$ is commutative, we can identify left $J$-torsors and right $J$-torsors in the naive way.\footnote{I.e. if $Q$ is a right $J$-torsor, then we view it as a left $J$-torsor with the same $J$-action, without taking inverse.}
In particular, we can form the contracted product $Q_1\times^{J} Q_2$ of two $J$-torsors $Q_1$ and $Q_2$. The contracted product $Q_1\times^{J} Q_2$ is still a $J$-torsor by the commutativity of $J$. Moreover, if $Q$ is a $J$-torsor, then $Q'$ defined to be the same sheaf $Q$ with $J$-action given by 
\begin{align*}
    Q\times J &\to Q\times J \to Q\\
    (q,j)&\mapsto (q,j^{-1})
\end{align*}
satisfies $Q\times^{J}Q' \cong J_a$.
Thus $[*/J]$ is a Picard stack.

Recall that for a prestack $X$ on a site $\mathcal C$, its inertia prestack $\mathcal I_X$ is the prestack whose $U$-valued points are the groupoid of pairs $(x,\theta)$, where $x\in X(U), \theta\in\mathrm{Aut}(x)$. 
If $x\in X(U)$, then we let $\underline{\Aut}(x)$ denote the presheaf on $\mathcal C_{/U}$ given by $V\mapsto \Aut(x|_V)$.
\begin{proposition}\label{inertia gives action}
    Let $\mathcal C$ be a site, $X$ be a stack over $B$, and $J$ be a commutative group sheaf on $\mathcal C$.
    Suppose we have a morphism $$f:J\times X \to \mathcal I_X$$ over $X$.
    Equivalently, for every $U\in\mathcal C$ and $x\in X(U)$, we have a morphism 
    \begin{equation}\label{J to Aut}
        J|_U\to\underline{\Aut}(x)
    \end{equation}
    such that for every $U\in\mathcal C$ and every morphism $x\isomto y$ in $X(U)$, the following commutes
    \[\begin{tikzcd}
        & {J|_U} & \\
        {\underline{\Aut}(x)} && {\underline{\Aut}(y)}
        \arrow[from=1-2, to=2-1]
        \arrow[from=1-2, to=2-3]
        \arrow[from=2-1, to=2-3]
    \end{tikzcd}\]
    where the horizontal map is given by conjugation.
    If \eqref{J to Aut} is a group homomorphism for every $U\in\mathcal C$ and $x\in X(U)$,
    then there is a canonical action 
    \begin{equation}\label{eq:action-general}
        [*/J]\times X \to X.
    \end{equation}
\end{proposition}

\begin{remark}Informally, the action is given by twisting objects and morphisms in $X$ by $J$-bundles.
\end{remark}

\begin{remark}
    The assumption that $J$ is commutative ensures that $[*/J]$ is a Picard stack, so its action on another stack makes sense.
\end{remark}

\begin{proof}
    For $U\in\mathcal C$ and $x\in X(U)$, we denote the image of $j$ under $J|_U\to\underline{\Aut}(x)$ as $j_x$.

    We define a morphism of prestacks $[*/_{pre}\;J]\times X \to X$, which on $U$-valued points (for $U\in \mathcal C$) is given by $$(*,x)\mapsto x$$ on objects and is given by $$((*,x)\xrightarrow{(j,\theta)} (*,y))\mapsto (x\xrightarrow{\theta\circ j_x = j_y\circ\theta}y)$$
    on morphisms. Taking the stackification gives the desired action on stacks.
\end{proof}

\begin{remark}
    Let us describe the action explicitly on objects. To avoid confusing with index, let us write $j(x)$ for what we denote as $j_x$ above.
    Let $Q$ be a $J$-bundle on $U$ and $x\in X(U)$.
    Fix a covering $\set{U_a\to U}$ of $U$ and trivialisations $\iota_a:Q|_{U_a}\isomto J|_{U_a}$.
    Write $U_{ab}=U_a\times_U U_b$.
    Let $$\iota_b|_{U_{ab}}\circ \iota_a|_{U_{ab}}^{-1}=:j_{ab}\in H^0(U_{ab}, Q).$$
    Then $$(x|_{U_a}, j_{ab}(x|_{U_{ab}})),$$
    where $j_{ab}(x|_{U_{ab}})$ is viewed as an isomorphism $(x|_{U_a})|_{U_{ab}} \isomto (x|_{U_b})|_{U_{ab}}$,
    is a descent datum on $X$, so it glues to an object on $X(U)$.
    It is straightforward to check that this is the image of $(j,x)$ under the map \eqref{eq:action-general}.
\end{remark}

We will need the following slight generalisation of \Cref{inertia gives action}.
\begin{proposition}\label{inertia gives action over a base}
    Let $\mathcal C$ be a site, $B$ be a sheaf, $X$ be a stack on $\mathcal C$, and $h:J\to B$ be a commutative group sheaf over $B$.
    Suppose we have a morphism $$f:J\times_B X \to \mathcal I_X$$ over $X$.
    In particular, for all $U\in\mathcal C$, $x\in X(U)$ having image $b$ in $B(U)$, we get a map
    \begin{equation*}
        \set{j\in J(U):h(j)=b} \to \Aut(x)\\
    \end{equation*}
    Suppose that this is a group homomorphism.
    Then there is a canonical action 
    \begin{equation*}
        [B/J]\times_B X \to X
    \end{equation*}
    of stacks on $\mathcal C$ over $B$.
\end{proposition}

\begin{proof}
    This follows from \Cref{inertia gives action} by replacing $\mathcal C$ by $\mathcal C_{/B}$ and identifying stacks over $B$ with stacks on $\mathcal C_{/B}$.
\end{proof}

\subsection{An exactness criterion}
We will often prove results about $G$-bundles by first proving them for vector bundles, and then generalise the results using the Tannakian formalism, namely by interpreting $G$-bundles on $X$ as \textit{exact} tensor functors $\mathrm{Rep}(G) \to \set{\text{vector bundles on }X}$. The following result is useful for proving exactness of functors in various situations.
\begin{lemma}\label{lem: exact criterion}
    Let $\set{X_i\xrightarrow{f_i} X}_{i\in I}$ be morphisms of locally ringed spaces such that $\cup f_i(X_i) = X$. Then a sequence of vector bundles on $X$
    \begin{equation}\label{E_i exact}
        0 \to \E \xrightarrow{a} \E' \xrightarrow{b} \E'' \to 0 
    \end{equation}
    with $b\circ a =0$
    is exact if and only if
    \begin{equation}\label{pullback of E_i exact}
        0 \to f_i^*\E \to f_i^*\E' \to f_i^*\E'' \to 0 
    \end{equation}
    is exact for every $i$.
\end{lemma}

The proof is given to us by Ken Lee in the context of part \eqref{item:seq} of \Cref{lem: exact 1}, namely in the case where $X$ and $X_i$ are the Fargues-Fontaine curves. We sincerely thank him for it. The generalisation to locally ringed spaces is proved in the same way.

To simplify notation, in the proof, whenever $Y$ is a locally ringed space and $y\in Y$, then we let $\kappa(y)=\O_{Y,y}/m_y$ denote the residue field of $Y$ at $y$. If $\mathcal F$ is a sheaf of $Y$, then $\mathcal F(y):=\mathcal F_y \otimes_{\O_{Y,y}} \kappa(y)$.

\begin{proof}
    The 'only if' direction is clear by considering stalks. Let us prove the `if' direction.
    By \cite[Lemma 00O0]{stacks-project}, to check that \eqref{E_i exact} is exact, it suffices to show that for every $x\in X$,  
    \begin{equation}\label{new seq}
        0\to \E(x) \to\E'(x) \to\E''(x) \to 0
    \end{equation}
    is exact.
    By assumption, there exists $i\in I$ and $x'\in X_i$ such that $f_i(x')=x$.
    We can check the exactness of \eqref{new seq} after base changing along $\kappa(x) \to \kappa(x')$.
    Note that $\E(x)\otimes_{\kappa(x)} \kappa(x') = (f_i^*\E)(x')$, so we would like to show the exactness of
    \[0\to (f_i^*\E)(x') \to (f_i^*\E')(x') \to (f_i^*\E'')(x') \to 0.\]
    This follows from the exactness of \eqref{pullback of E_i exact}.
\end{proof}

\subsection{Fargues-Fontaine curve}
Let $E$ be a non-Archimedean local field with residue field $\F_q$. Let $\pi$ be a uniformizer of $E$.
\begin{definition}
    For an affinoid perfectoid $S=\Spa(R,R^+)\in\Perf_{\F_q}$, define $$\mathcal Y_{E,S}:=\Spa W_{\O_E}(R^+) \setminus V([\varpi])$$ and $$Y_{E,S}:=\Spa W_{\O_E}(R^+) \setminus V(\pi[\varpi]),$$ where $\varpi$ is a pseudouniformizer of $R$ and $W_{\O_E}(R^+)$ is equipped with the $(\pi,[\varpi])$-adic topology.
    We define the \emph{(relative) Fargues-Fontaine curve} to be $$X_{E,S}:=Y_{E,S}/\phi_S^\Z,$$ where\footnote{The quotient makes sense as an adic space because the action of $\phi$ on on $Y_S$ is free and totally discontinuous by \cite[Proposition II.1.16]{FS_geometrisation}.} $\phi_S$ is the automorphism on $Y_{E,S}$ induced by the $q$-Frobenius\footnote{The morphism $x\mapsto x^q$.} on $R^+$.
    For general $S\in\Perf_{\F_q}$, we define $X_{E,S}$ to be the adic space obtained by gluing.\footnote{This gluing makes sense because if $S\to T$ is an open immersion of affinoid perfectoid spaces, then $X_S\to X_T$ is an open immersion by \cite[Proposition II.1.3]{FS_geometrisation}.}
\end{definition}
When there is no ambiguity, we will write $\mathcal Y_S$ for $\mathcal Y_{E,S}$, $Y_S$ for $Y_{E,S}$, and $X_S$ for $X_{E,S}$.

Here, $W_{\O_E}(R^+)$ is the unique $\pi$-adically complete, flat $\O_E$-algebra such that $W_{\O_E}(R^+)/\pi = R^+$ over $\F_q$.
Its existence and uniqueness can be proved using the vanishing of cotangent complex on perfect rings and the relationship between cotangent complex and lift of algebras, as summarised in \cite[Theorem 5.11]{Perfectoid_spaces}.
We also have $W_{\O_E}(R^+) = W(R^+)\hat{\otimes}_{W(\F_q)} \O_E$, where the completion is the $\pi$-adic completion.
Explicitly,
\begin{align*}
    W_{\O_E}(R^+) &= W(R^+)\otimes_{W(\mathbb F_q)}\O_E &\text{if } \mathrm{char} E=0,\\
    W_{\O_E}(R^+) &= R^+[[t]] &\text{if } E=\F_q((t)).
\end{align*}
In any case, every element in $W_{\O_E}(R^+)$ can be written uniquely as $\sum_{i=0}^\infty [a_i]\pi^i$, where $a_i\in R^+$ and $[-]$ is the Teichmuller lift, and conversely every such element defines an element of $W_{\O_E}(R^+)$.

While the definition of $Y_S$ and $X_S$ may look a bit complicated, it turns out that the diamonds associated with them have simple and intuitive descriptions.
\begin{proposition}\cite[Proposition II.1.17]{FS_geometrisation}
    For every $S\in\Perf_{\F_q}$, $Y_S^\diamond = \Spd E\times S$ and $X_S^\diamond = \Spd E\times S/\phi_S^\Z$.
\end{proposition}

\begin{definition}
    For $\lambda=\frac{d}{r}$ with $d,r\in \mathbb Z$ and $r>0$, we have the vector bundle $\O_{Y_S}^r$  with a $\phi_S$-linear automorphism
    $\begin{pmatrix}
0 & 1 & 0 & \dots &  \\
 & 0 & 1 & 0 &  \\
 &  & \ddots & \ddots & \ddots \\
0 &  &  & 0 & 1 \\
\pi^{-d} & 0 &  &  & 0 
\end{pmatrix}\phi_S,$
where the matrix is of size $r\times r$.
This descends to a vector bundle on $X_S$, which we denote by $\O(\lambda)$.
\end{definition}
Note that the power of $\pi$ in the definition is $-d$, not $d$.

When $S$ is an affinoid perfectoid space over $\mathbb F_q$, there is a schematic version of the Fargues-Fontaine curve.
\begin{definition}
    $X_S^{sch}:=\Proj\left(\bigoplus_{m\ge 0}H^0(X_S,\O_{X_S}(m))\right).$
\end{definition}
We have the following GAGA result:
\begin{proposition}\label{thm:GAGA}\cite[after Remark II.2.8]{FS_geometrisation}
    Suppose $S$ is affinoid perfectoid.
    Then we have a natural map of locally ringed spaces 
    \[f:X_S\to X_S^{sch}\]
    and pullback along this map induces an equivalence of categories 
    \begin{equation*}
        f^*:\{\text{vector bundles on }X_S^{sch}\}\xrightarrow{\simeq}  \{\text{vector bundles on }X_S\}.
    \end{equation*}
    In addition, for every vector bundle $\E$ on $X_S^{sch}$, we have
    \[H^i(X_S^{sch},\E)=H^i(X_S,f^*(\E))\]
    for all $i\ge 0$.
\end{proposition}

\begin{remark}
    By \cite[proof of Proposition II.2.7]{FS_geometrisation}, the equivalence above is an equivalence of exact categories, i.e. a sequence $0\to \E_1\to \E_2 \to \E_3 \to 0$ of vector bundles on $X_S^{sch}$ is exact if and only if $0\to f^*\E_1\to f^*\E_2 \to f^*\E_3 \to 0$ is exact. 
    It follows from the Tannakian interpretation of $G$-bundles that we have an equivalence of categories \begin{equation*}
        f^*:\{\text{$G$-bundles on }X_S^{sch}\}\xrightarrow{\simeq}  \{\text{$G$-bundles on }X_S\}
    \end{equation*}
    for any linear algebraic group $G$ over $E$.
\end{remark}

\subsection{Moduli stack of \texorpdfstring{$G$}{G}-bundles}
\begin{definition}
    For a linear algebraic group $G$ over $E$,
    we define $\Bun_G$ to be the prestack on\footnote{We switch to $\Perf_{\overline\F_q}$ because this is the setup where $\Bun_G$ is studied in \cite{FS_geometrisation}.} $\Perf_{\overline\F_q}$ taking $S\in \Perf_{\overline\F_q}$ to the groupoid of $G$-bundles on $X_S$.
\end{definition}
This is a v-stack \cite[before Definition III.1.2]{FS_geometrisation}. This is proved as follows: One first proves that $S\mapsto \{\text{vector bundles on }X_S\}$ is a v-stack \cite[Proposition II.2.1]{FS_geometrisation}. One can view a $G$-bundle on $X_S$ as an exact tensor functor $\mathrm{Rep}_E(G)\to  \{\text{vector bundles on }X_S\}$. Using these two facts and part \eqref{item:seq} of the following lemma, one can deduce that $\Bun_G$ is a v-stack.

\begin{lemma}\label{lem: exact 1}
    Let $S'\to S$ be a v-cover of perfectoid spaces over $\Spa \F_q$.
    Let $c:X_{S'} \to X_S$ be the associated map on the Fargues-Fontaine curves.
    \begin{enumerate}
        \item \label{item:top space} The map on the underlying topological spaces $|c|:|X_{S'}| \to |X_S|$ is surjective.
        \item \label{item:seq} A sequence
        \begin{equation*}
            0\to \E_1 \xrightarrow{f} \E_2 \xrightarrow{g} \E_3 \to 0
        \end{equation*}
        of vector bundles on $X_S$ with $g\circ f=0$ is exact if and only if its pullback
        \begin{equation*}
            0\to c^*\E_1 \to c^*\E_2 \to c^*\E_3 \to 0
        \end{equation*}
        is exact.
    \end{enumerate}
\end{lemma}

\begin{proof}
    For part \eqref{item:top space}, note that as in \cite[Section II.1.2]{FS_geometrisation}, $$|X_S|=|X_S^\diamond| = |S/\phi^\Z \times \Spd E| \cong |S\times (\Spd E)/\phi^\Z|$$ because the absolute Frobenius $\phi\times\phi$ of $S\times \Spd E$ acts trivially on the topological space.
    Since $S'\to S$ is a v-cover, it is a surjection of v-sheaves. It follows that its base change $S'\times (\Spd E)/\phi^\Z \to S\times (\Spd E)/\phi^\Z$ is also a surjection of v-sheaves. The map on the underlying topological spaces is thus surjective.

    Part \eqref{item:seq} follows from part \eqref{item:top space} and \Cref{lem: exact criterion}.
\end{proof}

\subsection{\texorpdfstring{$B^+_{dR}$-}{}Affine Grassmannian}
\begin{lemma}[Beauville-Laszlo lemma]\cite{Beauville-Laszlo}\label{BL lemma}
    Let $A$ be a ring and $f\in R$ be a non-zero-divisor. Let $\hat A=\lim_n A/f^n$. There is an equivalence of categories between $\set{\text{finite projective $A$-modules}}$ and the category of triples $(F,G,\varphi)$, where
    \begin{itemize}
        \item $F$ is a finite projective $A[\frac{1}{f}]$-module
        \item $G$ is a finite projective $\hat A$-module
        \item $\phi:F\otimes_A \hat A \isomto G[\frac{1}{f}]$ is an $\hat A[\frac{1}{f}]$-module isomorphism.
    \end{itemize}
\end{lemma}

\begin{lemma}\label{exactness for BL}
    Let $A$ be a commutative ring and $f\in R$ be a non-zero-divisor.
    A sequence
        \begin{equation*}
            0\to M_1 \to M_2 \to M_3 \to 0
        \end{equation*}
        of finite projective $A$-modules is exact if and only if its pullbacks
        \begin{equation*}
            0\to M_1\otimes_A A\left[\frac{1}{f}\right] \to M_2\otimes_A A\left[\frac{1}{f}\right] \to M_3\otimes_A A\left[\frac{1}{f}\right] \to 0
        \end{equation*}
        and
        \begin{equation*}
            0\to M_1\otimes_A \hat A \to M_2\otimes_A \hat A \to M_3\otimes_A \hat A \to 0
        \end{equation*}
        are exact.
\end{lemma}

\begin{proof}
    By \Cref{lem: exact criterion}, it is enough to show that the natural map $\Spec A\left[\frac{1}{f}\right] \sqcup \Spec \hat A \to \Spec A$ is surjective (on the underlying topological spaces). This follows from \cite[Lemme 1]{Beauville-Laszlo} by taking $M$ to be the residue field of $A$ at every prime ideal of $A$.
\end{proof}

The following follows immediately from \Cref{BL lemma,exactness for BL}.
\begin{corollary}
    Let $A$ be an $E$-algebra and $f\in R$ be a non-zero-divisor. Let $\hat A=\lim_n A/f^n$. Let $G$ be a linear algebraic group over $E$.
     There is an equivalence of categories between $\set{\text{$G$-bundles on }\Spec A}$ and the category of triples $(\mathcal F,\mathcal G,\varphi)$, where
    \begin{itemize}
        \item $\mathcal F$ is a $G$-bundle on $\Spec A[\frac{1}{f}]$
        \item $\mathcal G$ is a $G$-bundle on $\Spec \hat A$
        \item $\phi:\mathcal F|_{\Spec\hat A[\frac{1}{f}]} \isomto \mathcal G|_{\Spec\hat A[\frac{1}{f}]}$ is an isomorphism.
    \end{itemize}
\end{corollary}

Let us recall the definitions of $B_{dR}^+(R)$ and $B_{dR}(R)$.
Let $R$ be a perfectoid $E$-algebra with a pseudo-uniformizer $\varpi$.
Then we have Fontaine's map $$\theta:W_{\O_E}(R^{\circ\flat})\twoheadrightarrow R^\circ.$$ Its kernel is generated by an element $\xi$ of the form $\pi+[\varpi^\flat]\alpha$ for some $\alpha\in W_{\O_E}(R^{\circ\flat})$ and pseudouniformizer $\varpi^\flat$ of $R^\flat$.
We define
\[B_{dR}^+(R):=\varprojlim_i W_{\O_E}(R^{\circ\flat})\left[\frac{1}{[\varpi^\flat]}\right]/(\xi^i)\]
and 
\[B_{dR}(R):=B_{dR}^+(R)\left[\frac{1}{\xi}\right].\]

One can show that $B_{dR}^+(R)$ is $\xi$ torsion-free, $\xi$-adically complete, and $B_{dR}^+(R)/(\xi) \cong R$.

Let $G$ be a connected reductive group over $E$.
\begin{definition}
    The \emph{$B_{dR}^+$-affine Grassmannian} $\Gr_G$ is the v-sheaf on $\Perf$ given by sending an affinoid perfectoid space $S=\Spa(R,R^+)$ to the equivalence classes of triples $\set{(S^\sharp, \E,\iota)}$, where
    \begin{itemize}
        \item $S^\sharp=\Spa(R^\sharp,R^{\sharp+})$ is an untilt of $S$ over $\Spa E$
        \item $\E$ is a $G$-bundle on $\Spec B_{dR}^+(R^\sharp)$
        \item $\iota:\E|_{\Spec B_{dR}(R^\sharp)} \isomto \E_0|_{\Spec B_{dR}(R^\sharp)}$ is an isomorphism to the trivial $G$-bundle.
    \end{itemize}
\end{definition}

One way to prove that this is a v-sheaf is to first show that $$S\mapsto \set{\text{vector bundles on }\Spec B^+_{dR}(R^\sharp)}$$ is a v-stack \cite[Corollary 17.1.9]{Berkeley} and then use the Tannakian formalism together with \Cref{lem: exactness on B_dR} below.

One can show using the Beauville-Laszlo lemma that this is isomorphic to the v-sheaf sending an affinoid perfectoid space $S=\Spa(R,R^+)$ to the equivalence classes of triples $\set{(S^\sharp, \E,\iota)}$, where
\begin{itemize}
    \item $S^\sharp=\Spa(R^\sharp,R^{\sharp+})$ is an untilt of $S$ over $\Spa E$
    \item $\E$ is a $G$-bundle on $X_S^{sch}$
    \item $\iota:\E|_{X_S^{sch}\setminus \Spec R^\sharp} \isomto \E_0|_{X_S^{sch}\setminus \Spec R^\sharp}$ is an isomorphism to the trivial $G$-bundle.
\end{itemize}

In particular, we have a morphism $\Gr_G \to \Bun_G$, called the \emph{Beauville-Laszlo morphism}. This is surjective as a morphism of pro-\'etale stacks \cite[Proposition III.3.1]{FS_geometrisation}.

Let us state a handy lemma.
\begin{lemma}\label{exact_criterion}
    Let $R$ be a commutative ring and $\xi\in R$. Let
    \begin{equation}\label{sequence}
        0\to A\xrightarrow{f}B\xrightarrow{g}C
    \end{equation}
    be morphisms of $\xi$-complete $R$-modules. Assume that $g\circ f=0$ and $\xi$ is neither a zero divisor in $B$ nor $C$.
    If
    \begin{equation}\label{sequence2}
        0\to A/\xi\to B/\xi\to C/\xi
    \end{equation}
    is exact, then so is \eqref{sequence}.
    If $B/\xi\to C/\xi$ is surjective, then $B\to C$ is surjective.
\end{lemma}

\begin{proof}
    Assume \eqref{sequence2} holds.
    It is straightforward to check by induction that for every $n\ge 1$,
    \begin{equation*}
        0\to A/\xi^n\to B/\xi^n\to C/\xi^n
    \end{equation*}
    is exact.
    It is also clear that 
    \begin{equation*}
        0\to \varprojlim A/\xi^n\to\varprojlim B/\xi^n\to \varprojlim C/\xi^n
    \end{equation*}
    is exact.
    The exactness of \eqref{sequence} now follows by the $\xi$-completeness of $A,B,C$.

    Assume $B/\xi\to C/\xi$ is surjective.
    Let $c\in C$.
    Then there exists $b_0\in B$ such that $c=g(b_0)+\xi c_1$ for some $c_1\in C$.
    There exists $b_1\in B$ such that $c_1=g(b_1)+\xi c_2$ for some $c_2\in C$, and hence $c= g(b_0+\xi b_1) + \xi^2 c_2$.
    We continue this process.
    Then $b_0+\xi b_1 + \xi^2 b_2+\dots$ defines an element in $B$ whose image in $C$ is $c$ by $\xi$-completeness of $B$ and $C$.
    Thus $B\to C$ is surjective. 
\end{proof}

\begin{lemma}\label{lem: exactness on B_dR}
    Let $h:S'=\Spa(R',R'^+)\to S=\Spa(R,R^+)$ be a v-cover of affinoid perfectoid spaces over $\Spa E$.
    \begin{enumerate}
        \item \label{part 1} A sequence
            \begin{equation}\label{P_i exact}
                0\to P_1 \xrightarrow{f} P_2 \xrightarrow{g} P_3 \to 0
            \end{equation}
            of finite projective $R$-modules with $g\circ f=0$ is exact if and only if its pullback
            \begin{equation}\label{pullback of P_i exact}
                0\to P_1\otimes_R R' \to P_2\otimes_R R' \to P_3\otimes_R R' \to 0
            \end{equation}
            is exact.
        \item \label{part 2}A sequence
            \begin{equation}\label{M_i exact}
                0\to M_1 \xrightarrow{f} M_2 \xrightarrow{g} M_3 \to 0
            \end{equation}
            of finite projective $B^+_{dR}(R)$-modules with $g\circ f=0$ is exact if and only if its pullback
            \begin{equation}\label{pullback of M_i exact}
                0\to M_1\otimes_{B^+_{dR}(R)}B^+_{dR}(R') \to M_2\otimes_{B^+_{dR}(R)}B^+_{dR}(R') \to M_3\otimes_{B^+_{dR}(R)}B^+_{dR}(R') \to 0
            \end{equation}
            is exact.
    \end{enumerate}
\end{lemma}

This does not follow immediately from \Cref{lem: exact criterion}, because neither $\Spec(R')\to \Spec R$ nor $\Spec B^+_{dR}(R') \to \Spec B^+_{dR}(R)$ seem to be surjective.

\begin{proof}
    For part \eqref{part 1}, the `only if' direction is clear.
    Let us now prove the `if' direction.
    Note that $\tilde{P_i}:= P_i \otimes_R \O_S$ are vector bundles on $S$.
    By \cite[Theorem 1.3.4]{Kedlaya-AWS}, $H^1(S,\tilde{P_i})=0$.
    Thus, to prove the exactness of \eqref{P_i exact}, it suffices to show that the sequence of vector bundles on $S$
    \[0 \to \tilde{P_1} \to \tilde{P_2} \to \tilde{P_3} \to 0\]
    is exact. Since $S'\to S$ is a v-cover, $|S'|\twoheadrightarrow |S|$. By \Cref{lem: exact criterion}, it suffices to show that 
    \[0 \to h^*\tilde{P_1} \to h^*\tilde{P_2} \to h^*\tilde{P_3} \to 0\]
    is exact.
    Note that $h^*\tilde{P_i}= \tilde{P_i\otimes_R R'}:= (P_i\otimes_R R')\otimes_{R'} \O_{S'}$, so the desired result follows from the exactness of \eqref{pullback of P_i exact}.

    For part \eqref{part 2}, the `only if' direction is clear.
    Let us now prove the `if' direction.
    Note that $M_i$ is $\xi$-adically complete: There is a $B^+_{dR}(R)$-module $N_i$ such that $M_i\oplus N_i \cong B^+_{dR}(R)^{n_i}$ for some $n_i$. Since $B^+_{dR}(R)$ is $\xi$-adically complete, $B^+_{dR}(R)^{n_i} \isomto \lim_{m} B^+_{dR}(R)^{n_i}/\xi^m$ is an isomorphism, which implies that $M_i\isomto \lim_{m\ge 1} M_i/\xi^m$.
    By \Cref{exact_criterion}, to show the exactness of \eqref{M_i exact}, it suffices to show that 
    \[0\to M_1/\xi \to M_2/\xi \to M_3/\xi \to 0\]
    is exact.
    Since \eqref{pullback of M_i exact} consists of finite projective $B^+_{dR}(R')$-modules, its base change along $B^+_{dR}(R') \to B^+_{dR}(R')/\xi$ is still exact, i.e. 
    \[0 \to M_1/\xi \otimes_{R} R'\to M_2/\xi \otimes_{R} R'\to M_3/\xi \otimes_{R} R'\to 0\]
    is exact.
    Now the result follows from part \eqref{part 1}.
\end{proof}

\section{Moduli stack of Higgs bundles}
In this section, we define the moduli stack of Higgs bundles and establish some basic properties of it.

Let $E$ be a non-Archimedean local field with residue field $\F_q$, where $q$ is a power of $p$. Let $G$ be a connected reductive group over $E$. Let $\g(E)$ be its Lie algebra. Fix $S\in \Perf_{\F_q}$ and a line bundle $\mathcal L$ on the Fargues-Fontaine curve $X_S$.
Inspired by \cite[definition 4.2.1]{Ngo_2010}, we make the following definition:
\begin{definition}
    For $T\in \Perf_S$, a $\L$-twisted $G$-Higgs bundle on $X_T$ is a pair $(\E,\phi)$, where
    \begin{itemize}
        \item $\E$ is a $G$-bundle on $X_T$,
        \item $\phi\in H^0(X_T, \Ad(\E)\otimes_{\O_{X_T}}\L_T)$.
    \end{itemize}
\end{definition}
Here, $\L_T$ is the pullback of $\L$ to $X_T$ and $\Ad(\E)$ is the adjoint bundle.\footnote{If one view a $G$-bundle as an exact tensor functor $\mathrm{Rep}(G)\to (\text{vector bundles on }X_T)$, then this is the image of $\g(E)$ with the adjoint action. If one view a $G$-bundle as a sheaf on $X_T$, then $\Ad(\E)= \E \times^G (\g(E)\otimes_E \O_{X_T}).$}
We will usually just call them $G$-Higgs bundles or Higgs bundles.
We will let $\Ad(\E)_\L$ denote $\Ad(\E)\otimes_{\O_{X_T}}\L_T.$

\begin{definition}
    Let $\mathrm{Higgs}_{G,\mathcal L}$ be the prestack on $\Perf_S$ taking $T\in \Perf_S$ to the groupoid of Higgs bundles on $X_T$, where an isomorphism $(\E_1,\phi_1) \iso (\E_2,\phi_2)$ is an isomorphism $\E_1\iso \E_2$ mapping $\phi_1$ to $\phi_2$.
\end{definition}

When there is no ambiguity, we will write $\mathrm{Higgs}_{G,\mathcal L}$ as $\Higgs$. 

\begin{example}
    \begin{itemize}
        \item For $G=\mathbb G_m$, we have $\Higgs=\Bun_{\mathbb G_m}\times \mathcal H^0(\L)$.
        \item For $G=\GL_n$, there is a more concrete description of $G$-Higgs bundles. We have an equivalence of groupoids
        \begin{align*}
            \set{\text{rank $n$ vector bundles on $X_S$}} &\simeq \set{\GL_n\text{-bundles on }X_S}\\
            \mathcal V &\mapsto \Isom_{\O_{X_S}}(\O_{X_S}^n, \mathcal V) \\
            \E \times^{\GL_n}\O_X^n    &\mapsfrom \E.
        \end{align*}
        If $\V$ corresponds to $\E$, then $\Ad(\E)_\L = \underline{\Hom}(\V,\V\otimes \L)$, where the right hand side is the sheaf on $X_S$ sending $U$ to $\Hom_{\O_U}(\V|_U,(\V\otimes \L)|_U)$. Thus, we can identify a $\GL_n$-Higgs bundles with a pair $(V,\phi:V\to V\otimes_{\O_{X_S}} \L)$, where $V$ is a rank $n$ vector bundle on $X_S$ and $\phi$ is $\O_{X_S}$-linear.
    \end{itemize}
\end{example}

\begin{lemma}
    The prestack $\Higgs$ is a v-stack.
\end{lemma}

\begin{proof}
    This follows easily from the fact that $\Bun_G$ is a v-stack and for any $G$-bundle $\E$, the functor $T\mapsto H^0(X_T, \Ad(\E)\otimes_{\O_{X_T}}\L_T)$, being the Banach-Colmez space associated with $\Ad(\E)\otimes_{\O_{X_S}}\L$, is a v-sheaf by \cite[Proposition II.2.1]{FS_geometrisation}.
\end{proof}

\begin{proposition}\label{Higgs_to_Bun}
    The natural map $\Higgs\to \Bun_G$ is representable in locally spatial diamonds and is partially proper.
\end{proposition}

\begin{proof}
    Let us denote this map by $f$.
    The proof of \cite[Proposition III.1.3]{FS_geometrisation} shows that there is a perfectoid space $T$ such that $T\twoheadrightarrow \Bun_G$ as pro-\'etale stacks.
    This map corresponds to a $G$-bundle $\E$ on $X_T$.
    The base change of $f$ along this map is 
    $$\Higgs\times_{\Bun_G} T =\mathcal H^0(\Ad(\E)\otimes \L_T) \to T.$$
    As $\mathcal H^0(\Ad(\E)\otimes \L_T)$ is a locally spatial diamond, this map is representable in locally spatial diamonds.
    Moreover, by \cite[Proposition II.2.6]{FS_geometrisation}, this map is quasi-separated.
    It follows from \cite[Proposition 10.11]{EC_diamonds} that $f$ is quasi-separated.
    By \cite[Proposition 13.4(iv)]{EC_diamonds}, $f$ is representable in locally spatial diamonds.

    It is clear that $f$ is $0$-truncated and we have just shown that $f$ is quasi-separated. We can deduce that $f$ is partially proper by applying the following fact to the adjoint bundle: if $\Spa(R,R^+)$ is an affinoid perfectoid space over $S$ and $\mathcal V$ is a vector bundle on $X_{(R,R^+)}$, then we have an isomorphism $H^0(X_{(R,R^+)},\mathcal V) \isomto H^0(X_{(R,R^\circ)},\mathcal V)$. This fact is immediate from the partial properness of the map $\mathcal H^0(\mathcal V) \to S$, proved in \cite[Proposition II.2.16]{FS_geometrisation}.
\end{proof}

\begin{corollary}
    $\Higgs$ is an Artin v-stack if\footnote{We make the assumption that $S\in\Perf_{\overline F_q}$ because in \cite{FS_geometrisation}, $\Bun_G$ is defined on $\Perf_{\overline F_q}$.} $S\in\Perf_{\overline F_q}$.
\end{corollary}

\begin{proof}
    This follows from \Cref{Higgs_to_Bun} and \cite[Proposition IV.1.8(iii)]{FS_geometrisation}.
\end{proof}

\begin{proposition}\cite[Proposition 2.20]{relative-local-Langlands}
    If $S\in\Perf_{\overline F_q}$ and $\L=\O_{X_S}$, then $\Higgs \times_{\Bun_G} \Bun_G^1 \cong [\underline{\mathfrak g(E)}/\underline{G(E)} ].$
\end{proposition}

\begin{proof}
    This is true because $\Higgs$ is the same as the stack $\Bun^X_G$ in \cite{relative-local-Langlands}, where $X=\g$ with the adjoint action of $G$.
\end{proof}

\section{\texorpdfstring{$\GL_n$}{GL(n)} case}
In this section, we take $G=\GL_n$.
We will usually view $\Higgs$ as the stack that sends $T$ to the groupoid of pairs $(V,\phi:V\to V\otimes_{\O_{X_T}} \L_T)$, where $V$ is a vector bundle on $X_T$ and $\phi$ is $\O_{X_T}$-linear.

The aim of this section is to study the Hitchin fibration $\Higgs \to \mathcal A_G$ (cf. \Cref{Hitchin fibration for GL_n}), which essentially maps a pair $(V,\phi)$ to the characteristic polynomial of $\phi$. We prove an analogue of the BNR theorem, which relates fibres of the Hitchin fibration to line bundles on the spectral curves. 
This is an interesting result, because it is an instance of abelianisation: we reduce the study of (non-abelian) $\GL_n$ objects to (abelian) $\GL_1$ objects.
We also prove a GAGA type result for the spectral curve, which gives an equivalence between the category of vector bundles on the spectral curve and that on a schematic version of the spectral curve.

\subsection{Hitchin fibration}
\begin{definition}\label{A_in_GL_n_case}\label{Hitchin fibration for GL_n}
    We define the \emph{Hitchin base} $\mathcal A_{G,\L}:=\prod_{i=1}^n \mathcal H^0(\L^{\otimes i})$, which is a locally spatial diamond.
    When there is no ambiguity, we will write $\mathcal A_{G,\L}$ as $\mathcal A_G$.

    We define the {Hitchin fibration} 
    \[\Higgs \to \mathcal A_G\]
    to be the map which on $T$-valued points sends $(V,\phi:V\to V\otimes \L_T)$ to $((-1)^i \Tr \Lambda^i \phi)_i$.

    We call the image of $(V,\phi)\in \Higgs(T)$ under this map its \emph{characteristic polynomial}. We will also call $Z^n-\Tr(\phi) Z^{n-1}+ \Tr (\Lambda^2 \phi) Z^{n-2} +\dots + (-1)^n\Tr(\Lambda^n \phi) \in\prod_{i=1}^n H^0(X_T,\L^{\otimes i})Z^i$ the characteristic polynomial of $(V,\phi)$.
\end{definition}
Here, $\Tr \Lambda^i \phi$ is defined as follows: We apply $\Lambda^i$ to $\phi$ to get a map
\[\Lambda^i\phi : \Lambda^i V \to \Lambda^i(V\otimes \L)\cong \Lambda^i V\otimes \L^{\otimes i},\]
which corresponds to an element in $H^0(X_T, (\Lambda^i V)^\vee \otimes \Lambda^i V \otimes \L_T^{\otimes i})$. We can apply the evaluation map to get an element in $H^0(X_T,\L_T^{\otimes i})$, which we call $\Tr \Lambda^i \phi$.

\subsection{Spectral curve}
There is a way to describe the fibre of the Hitchin fibration in terms of line bundles on the spectral curve, which we will define now. This is desirable because it allows us to relate $\GL_n$ objects with $\GL_1$ objects.

Now, fix $S\in \Perf_{\F_q}$, $\L$ a line bundle on $X_S$, and $(a_i)\in\prod_{i=1}^n \L^{\otimes i}(X_S)$.
To construct the spectral curve, we shall construct adic spaces $\Spa(B_U,B_U^+)$ over $U$ for open subsets $U\subset X_S$ and then glue these spaces.

Let $U\subset X_S$ be an affinoid sousperfectoid space such that $\L|_U$ is isomorphic to $\O_U$.
Let $A=\O_U(U), A^+=\O_U^+(U)$.
If $\iota:\L|_U \isomto \O_U$ is a trivialisation, then $$B_{U,\iota}:=A[Z]/(Z^n+\iota(a_1)Z^{n-1}+\dots+\iota(a_n))$$
equipped with the induced topology\footnote{If $A$ is a topological ring, then any finitely generated $A$-module has a canonical topology, given by picking a surjection $A^n\twoheadrightarrow M$ and equipping $M$ with the quotient topology. This is independent of the choice of the surjection. See \cite[Lemma 8.3.34]{foundations-almost-ring}.} is a sheafy Huber ring by \cite[Corollary 4.7]{HK-sousperfectoid}.
If $\iota':\L|_U \isomto \O_U$ is another trivialisation. Then $\iota' = a \iota$ for some $a\in A^\times$, so 
\begin{align*}
    B_{U,\iota}&\to B_{U,\iota'}\\
    Z&\mapsto a^{-1}Z
\end{align*}
is an isomorphism.
Define\footnote{Of course $B_U \cong B_{U,\iota}$ for every $\iota$, but $B_U$ is more canonical in the sense that it is independent of the choice of $\iota$.} $$B_U:= \colim_{\iota:\L|_U \isomto \O_U} B_{U,\iota}$$
and $$B_U^+ = \text{integral closure of $A^+$ in $B_U$}.$$
\begin{lemma}
    $B_U^+$ is a ring of integral elements of $B_U$, i.e. $B^+_U$ is an open and integrally closed subring of $(B_U)^\circ$.
\end{lemma}

\begin{proof}
    We identify $B_U$ with some $B_{U,\iota}$ by fixing a trivialisation $\iota$.
    The map $A\hookrightarrow B_U$ is an adic morphism by \cite[Lemma 5.1.2]{Berkeley}, so it sends powerbounded elements to powerbounded elements.
    Thus, $A^+\subset (B_U)^\circ$.
    Thus
    \begin{align*}
        B_U^+ &\subset \text{integral closure of $(B_U)^\circ$ in $B_U$}\\
        &= (B_U)^\circ.
    \end{align*}
    It remains to show that $B_U^+$ is open in $B_U$.
    Let $\varpi$ be a pseudouniformizer of $A$.
    Then $\varpi^N Z$ is integral over $A^+$ for some $N\ge 1$, so $B_U^+ \supset A^+[\varpi^N Z]$.
    It is direct from the definition of the topology on $B_U$ that $A^+[\varpi^N Z]$ is open in $B_U$, so $B_U^+$ is also open.
\end{proof}

Hence, we have an adic space $\Spa(B_U,B_U^+)$ with a canonical map $$\pi_U: \Spa(B_U,B_U^+) \to \Spa(A,A^+)=U.$$
This is compatible with passing to open subspace of $U$:
\begin{lemma}\label{lem: open}
    Let $U'\subset U$ be an affinoid sousperfectoid open subspace of $U$.
    Then the canonical $U$-morphism $\Spa(B_{U'},B_{U'}^+) \to \pi_U^{-1}(U')$ is an isomorphism.
    In other words, we have a Cartesian diagram
    \[\begin{tikzcd}
        {\Spa(B_{U'},B_{U'}^+)} & {\Spa(B_{U},B_{U}^+)} \\
        {U'} & U.
        \arrow[from=1-1, to=1-2]
        \arrow["{\pi_{U'}}"', from=1-1, to=2-1]
        \arrow["{\pi_U}", from=1-2, to=2-2]
        \arrow[from=2-1, to=2-2]
    \end{tikzcd}\]
\end{lemma}

\begin{proof}
    We shall verify that they have the same functor of points.
    Fix a trivialisation $\iota:\L_U\isomto \O_U$ to make the identifications $B_U\cong B_{U,\iota}$ and $B_{U'}\cong B_{U',\iota|_{U'}}$.
    Let $U'=\Spa(A',A'^+)$. Let $(R,R^+)$ be a complete Huber pair over $(A,A^+)$.
    Recall that $B_{U'}^+$ is the integral closure of $A'^+$ in $B_{U'}=\frac{A'[Z]}{(Z^n+\iota(a_1)Z^{n-1}+\dots + \iota(a_n))}$.
    Thus
    \begin{align*}
        &(\Spa(B_{U'},B_{U'}^+))(R,R^+) \\
        &=\Hom_{(A,A^+)}\left((\frac{A'[Z]}{(Z^n+\iota(a_1)Z^{n-1}+\dots + \iota(a_n))}, B_{U'}^+ ) , (R,R^+)    \right) \\
        &=\Hom_{(A,A^+)}((A',A'^+), (R,R^+))\times \set{r\in R: r^n+\iota(a_1)r^{n-1}+\dots + \iota(a_n)=0}
    \end{align*}
    This is $\set{r\in R: r^n+\iota(a_1)r^{n-1}+\dots + \iota(a_n)=0}$ if the structure map $\Spa(R,R^+) \to \Spa(A,A^+)$ factors through $U'$ and empty otherwise.

    Similarly, if the structure map $\Spa(R,R^+) \to \Spa(A,A^+)$ factors through $U'$, then 
    \begin{align*}
        &\pi_U^{-1}(U')(R,R^+)\\
        &= \Hom_{U}(\Spa(R,R^+), \Spa(B_U,B_U^+))\\
        &= \Hom_{(A,A^+)}\left((\frac{A[Z]}{(Z^n+\iota(a_1)Z^{n-1}+\dots + \iota(a_n))},B_U^+ ) , (R,R^+)    \right) \\
        &= \set{r\in R: r^n+\iota(a_1)r^{n-1}+\dots + \iota(a_n)=0}.
    \end{align*}
    If the structure map $\Spa(R,R^+) \to \Spa(A,A^+)$ does not factor through $U'$, then $\pi_U^{-1}(U')(R,R^+) = \varnothing$.
    This finishes the proof.
\end{proof}

\begin{definition}
    Define the \emph{spectral curve} $C_{a,S}$ to be the adic space obtained by gluing $\Spa(B_U,B_U^+)$ as $U$ runs through all the affinoid sousperfectoid open subsets of $X_S$ such that $\L|_U\cong \O_U$.
\end{definition}
This gluing is valid by \Cref{lem: open}. By gluing $\pi_U$, we get a morphism $\pi:C_{a,S} \to X_S$.

Let $U=\Spa(A,A^+)$ be an affinoid sousperfectoid open subset of $X_S$ such that $\L|_U\cong \O_U$.
Then $(\pi^*\L)|_{\pi^{-1}(U)}$ is the vector bundle associated with the $B_U$-module\footnote{We rename $Z$ by $Z_{\iota}$ so as to distinguish the different copies of $Z$.} 
$$\L(U)\otimes B_U = \L(U)\otimes_A \colim_{\iota} A[Z_{\iota}]/(Z_{\iota}^n+\iota(a_1)Z_{\iota}^{n-1}+\dots+\iota(a_n)).$$
The element $$\iota^{-1}(1)\otimes Z_{\iota}$$ is independent of the choice of $\iota$, because if $\iota'=a\iota$ is another trivialisation of $\L|_U$, then $\iota^{-1}(1)\otimes Z_{\iota} = \iota^{-1}(1)\otimes a^{-1}Z_{\iota'} = \iota'^{-1}(1)\otimes Z_{\iota'}.$
Thus, we can glue them as $U$ varies to obtain an element \[\tau\in H^0(C_{a,S},\pi^*\L).\]
We call this the \emph{tautological section} of $\pi^*\L$.

Informally, locally we have $U=\Spa(A,A^+)$,
 $\pi^{-1}(U)\cong \Spa(A[Z]/(Z^n+\iota(a_1)Z^{n-1}+\dots+\iota(a_n)), \text{integral closure of }A^+)$,
  $\tau=Z$.

\begin{lemma}\label{lem: poly to vb}
    Let $A$ be a commutative ring, $M$ be a finite free $A$-module, $\phi\in\End_A(M)$ with characteristic polynomial $f(Z)=(Z-\lambda_1)\dots (Z-\lambda_n)$, where $\lambda_i\in A$ and $\lambda_i-\lambda_j\in A^\times$ for all $i\ne j$.
    If we regard $M$ as the $A[Z]/(f)$-module with $Z$ acting as $\phi$, then $M$ is a finite projective $A[Z]/(f)$-module of rank $1$.
\end{lemma}

\begin{proof}
    As $\lambda_i-\lambda_j\in A^\times$ for all $i\ne j$,
    \[A[Z]/(f) \cong A[Z]/(Z-\lambda_1) \times \dots \times A[Z]/(Z-\lambda_n)\cong A^n.\]
    Let $e_i$ be the idempotent in $A[Z]/(f)$ corresponding to the $i$-th factor, i.e. $e_i=\prod_{j\ne i} \frac{Z-\lambda_j}{\lambda_i-\lambda_j}$. Note that $1=e_1+\dots+e_n$ and $e_i e_j = \delta_{ij}$.
    Then 
    \begin{equation}\label{eq:M_sum}
        M=e_1M\oplus \dots\oplus e_n M
    \end{equation}
    as $A[Z]/(f)=A^n$-module.
    It suffices to show that each $e_iM$ is a finite projective $A$-module of rank $1$.

    Each $e_iM$ is a direct summand of the free $A$-module $M$, so each $e_iM$ is a finite projective $A$-module.
    By  \eqref{eq:M_sum}, it suffices to show $M_i\neq 0$ for all $i$.
    Let $K$ be a field over $A$. 
    If we view $\phi$ as a matrix over $K$, then it has characteristic polynomial $f$ with eigenspace decomposition 
    \[M\otimes_A K = ((e_1M)\otimes_A K)\oplus \dots \oplus ((e_n M)\otimes_A K).\]
    Note that $(e_i M)\otimes_A K$ is the $\lambda_i$-eigenspace, so it is non-zero.
    It follows that each $e_i M$ is non-zero and the claim follows.
\end{proof}

\begin{lemma}\label{lem:push vb}
    Let $A\to B$ be a ring homomorphism such that $B$ is free of rank $n$ as an $A$-module. Let $M$ be a finite projective $B$-module of rank $1$. Then $M$ is a finite projective $A$-module of rank $n$.
\end{lemma}

\begin{proof}
    Let $h:\Spec B \to \Spec A$ be the corresponding map on affine schemes. This is flat and of finite presentation, so it is open.
    Let $\tilde M$ be the quasi-coherent sheaf on $\Spec B$ corresponding to $M$. We can pick an open cover $\{U_i\}$ of $\Spec B$ such that $\tilde M|_{U_i} \cong \O_{U_i}$ for all $i$.
    Then $\{h(U_i)\}$ is an open cover of $\Spec A$ and $(h_*\tilde M)|_{h(U_i)} \cong (h_*\O_{U_i})|_{h(U_i)} \cong \O_{h(U_i)}^n$.
\end{proof}

We have the analogue of the BNR theorem \cite[Proposition 3.6]{BNR} in our setting:
\begin{proposition}\label{prop:CRT}
    For $1\le i\le n$, let $a_i\in \L^{\otimes i}(X_S)$. Let $f(Z)=Z^n+a_1Z^{n-1}+\dots +a_n\in \oplus_{i=1}^n \L^{\otimes i}(X_S)Z^i$.
    Assume that for all $x\in X_S$, there is an open neighborhood $U$ of $x$ such that $\L|_U$ is trivial and $f(Z)$ factorizes as $(Z-\lambda_1)\dots (Z-\lambda_n)$, where $\lambda_i\in \L(U)$ and $\lambda_i-\lambda_j\in \L(U)^\times$ for\footnote{Here is the definition of $\L|_U(U)^\times$: Pick an isomorphism $\iota:\L|_U\xrightarrow{\sim} \O_U$; $\L(U)^\times$ is the subset of $\L(U)$ which corresponds to $\O(U)^\times$ under the induced isomorphism $\L(U)\cong \O_U(U)$. This is independent of the choice of $\iota$.} all $i\neq j$.
    Then we have an equivalence of categories
    \begin{align*}
        \{\text{line bundles on }C_{a,S}\} &\to \{\text{rank }n \text{ Higgs bundles on }X_S \text{ with characteristic polynomial }f\}\\
        \mathcal F &\mapsto (\pi_*\mathcal F, \pi_*(\tau\otimes\mathcal F)),
    \end{align*}
    where $\pi:C_{a,S}\to X_S$ is the canonical map and $\pi_*(\tau\otimes\mathcal F)$ is defined as follows: $\tau$  corresponds to a map $\O\to \pi^*\L$; we can tensor it with $\mathcal F$ to get $\mathcal F\to (\pi^*\L)\otimes \mathcal F$; then we apply $\pi_*$ to get the desired map $\pi_*\mathcal F\to \pi_*((\pi^*\L)\otimes \mathcal F)=\L \otimes \pi_*\mathcal F$.
\end{proposition}

\begin{proof}
    For $x\in X_S$, let us call an open neighborhood $x\in U\subset X_S$ \emph{good} if $U$ is affinoid sousperfectoid and $\L|_U\cong \O_U$. Note that on such a $U$, we can identify $f$ with a polynomial with coefficients in $\O_U(U)$.

    Let $\mathcal F$ be a line bundle on $C_{a,S}$.
    Let us first show that $\pi_*\mathcal F$ is a rank $n$ vector bundle on $X_S$.
    Let $x\in X$ and $U=\Spa(A,A^+)\subset X$ be a good neighborhood.
    We know that $\pi^{-1}(U)\cong\Spa(A[Z]/(f), \tilde{A^+})$, where $\tilde{A^+}$ is the integral closure of $A^+$ in $A[Z]/(f)$.
    Then $\mathcal F|_{\pi^{-1}(U)}$ corresponds to a finite projective $A[Z]/(f)$-module of rank $1$, which, by \Cref{lem:push vb}, is a finite projective $A$-module of rank $n$. Thus, $(\pi_*\mathcal F)|_U$ is a rank $n$ vector bundle. It follows that the same holds for $\pi_*\mathcal F$.

    To see that $\varphi:=\pi_*(\tau\otimes \mathcal F)$ has characteristic polynomial $f$, i.e. $a_i=(-1)^i \Tr\Lambda^i\phi$ for all $1\le i \le n$, note that the formation of $\Tr\Lambda^i\varphi$ is compatible with restriction to open subsets. We can therefore check them locally, so we are reduced to showing the following:
    if $M$ is a finite projective $A[Z]/(f)$-module of rank $1$, then $Z\in \End_A(M)$, viewed as an $A$-linear endomorphism of the finite projective rank $n$ $A$-module $M$, has characteristic polynomial $f$. We can localise at $A$ to assume that $M$ is free of rank $n$ as an $A$-module. Then this is standard.

    Conversely, let $(\mathcal V,\phi)$ be a rank $n$ Higgs bundles on $X_S$ with characteristic polynomial $f$.
    Let $x\in X_S$ and $U=\Spa(A,A^+)\subset X_S$ be a good neighborhood such that $\mathcal V|_U\cong \O_U^n$ and $f(Z)$ factorises as in the statement of the proposition.
    We can view the finite free $A$-module $\mathcal V(U)$ as a $B_U= \colim_{\iota} A[Z_{\iota}]/(Z_{\iota}^n+\iota(a_1)Z_{\iota}^{n-1}+\dots+\iota(a_n))$-module with $Z_\iota$ acting as the endomorphism
    \begin{equation*}
        \mathcal V(U) \xrightarrow{\phi} (\mathcal V\otimes \L)(U) \xrightarrow{1\otimes\iota} \mathcal V(U).
    \end{equation*}
    This is well-defined, because if $\iota'=a\iota$ is another trivialisation of $\iota$, then $a^{-1}Z_{\iota'}$ ($=Z_{\iota}$ in $B_U$) acts as $a^{-1}$ times $\mathcal V(U) \xrightarrow{\phi} (\mathcal V\otimes \L)(U) \xrightarrow{1\otimes\iota'} \mathcal V(U)$, which is the same as the way that $Z_{\iota}$ acts.
    By \Cref{lem: poly to vb}, $\mathcal V(U)$ is a finite projective $B_U$-module of rank $1$.
    Thus, it gives rise to a line bundle $W$ on $\pi^{-1}(U)$.
    We can glue $W$ as $U$ varies to get a line bundle on $C_{a,S}$.

    It is not hard to verify that these two constructions are inverses of each other.
\end{proof}

\begin{remark}
    In the classical case, where $X$ is a smooth projective curve with genus at least $2$ over an algebraically closed field $k$ and $\L=\Omega_{X/k}$, then one can show that the assumption in the proposition is never satisfied by considering the degree of the discriminant divisor. However, the assumption in the proposition is possible in our setup. If $\L=\O_{X_S}$, then the assumption is roughly saying that the characteristic polynomial has distinct roots, which seems reasonable.
\end{remark}

As an immediate corollary of \Cref{prop:CRT}, we get the following relation between the Hitchin fibers and line bundles on the spectral curve.

\begin{theorem}\label{BNR}
    In the situation of \Cref{prop:CRT}, we have a canonical isomorphism of v-stacks
    \[\mathrm{Higgs}_{n,a} \cong \mathrm{Pic}_{C_a}.\]
    Here, $\mathrm{Higgs}_{n,a}$ is the v-stack on $\Perf_S$ sending $T$ to the groupoid of rank $n$ Higgs bundles on $X_S$ with characteristic polynomial $f$, i.e. the Hitchin fiber $\mathrm{Higgs}_{\GL_n}\times_{\mathcal A_{\GL_n}}S$; $\mathrm{Pic}_{C_a}$ is the v-stack on $\Perf_S$ sending $T$ to the groupoid of line bundles on $C_{a,T}$.
\end{theorem}

\begin{remark}
    In the setting of \Cref{prop:CRT}, the spectral cover $C_{a,S}$ is a degree $n$ covering space of $X_S$, i.e. for all $x\in X_S$, there is an open neighborhood $U$ of $x$ such that $\pi^{-1}(U)\cong \coprod_{i=1}^n U$ over $U$.
    Moreover, for all $T\in\Perf_S$, $$C_{a,T}=C_{a,S}\times _{X_S} X_T.$$ Indeed, there is a natural map $C_{a,T} \to C_{a,S}\times _{X_S} X_T$. Locally, $X_S$ is given by $\Spa(A,A^+)$, $X_T$ is given by $\Spa(B,B^+)$, $C_{a,S}$ is given by $\Spa(A[Z]/(f),\tilde A^+)$, and both $C_{a,T}$ and $C_{a,S}\times _{X_S} X_T$ are locally given by $\Spa(B[Z]/(f),\tilde B^+)$, where $\;\tilde{}\;$ denotes integral closure. The claim follows. This motivates the setup in the next subsection.
\end{remark}

\subsection{Covering space of the Fargues-Fontaine curve}
Fix $S\in \Perf_{\mathbb F_q}, n\ge 1$.
Let $\pi:C_S\to X_S$ be a degree $n$ covering of $X_S$, i.e. $C_S$ is an adic space and for all $x\in X_S$, there is an open neighborhood $U$ of $x$ such that $\pi^{-1}(U)\cong \coprod_{i=1}^n U$ over $U$.

Let us start with an easy lemma.
\begin{lemma}\label{lem:pvb}
    If $\E$ is a vector bundle of rank $m$ on $C_S$, then $\pi_*\E$ is a vector bundle of rank $mn$ on $X_S.$ Also, $\pi_*:Ab(C_S)\to Ab(X_S)$ is an exact functor, where $Ab(-)$ stands for the category of abelian sheaves. In particular, for all $\mathcal F\in Ab(C_S)$,
    \begin{equation}\label{eq:cohom}
        H^i(C_S,\mathcal F)=H^i(X_S,\pi_*\mathcal F)
    \end{equation} 
    for all $i\ge 0$.
\end{lemma}

\begin{proof}
    The first statement follows immediately from the fact that $\pi$ is a covering of degree $n$.
    For the second statement, note that for all $\mathcal F\in Ab(C_S)$ and $x\in X_S$, we have $(\pi_*\mathcal F)_x=\oplus_{y\in \pi^{-1}(x)}\mathcal F_y$ because $\pi$ is a covering.
\end{proof}

Let $\O_{C_S}(1):=\pi^*\O_{X_S}(1)$. Let us prove an ampleness result for $\O_{C_S}(1)$.
\begin{proposition}\label{ample}
    Assume $S$ is affinoid perfectoid. Let $\mathcal E$ be a vector bundle on $C_S$. Then there exists $N\ge 1$ such that for all $i\ge N$, $\E(i):=\E\otimes_{\O_{C_S}}\O_{C_S}(1)^{\otimes i}$ is globally generated\footnote{An $\O_{C_S}$-module $\mathcal F$ is called globally generated if there is an $m\ge 1$ and a surjection $\O_{C_S}^m \twoheadrightarrow \mathcal F$.} and $H^1(C_S,\E(i))=0$. Moreover, $H^j(C_S,\E)=0$ for all $j\ge 2$.
\end{proposition}

\begin{proof}
    By \Cref{lem:pvb}, $\pi_*\E$ is a vector bundle on $X_S$. 
    By equation \eqref{eq:cohom} and \cite[beginning of section II.2]{FS_geometrisation}, $H^j(C_S,\E)=0$ for all $j\ge 2$.
    By \cite[Theorem II.2.6]{FS_geometrisation}, there is an $N\ge 1$ such that for all $i\ge N$, $(\pi_*\E)(i):=(\pi_*\E)\otimes_{\O_{X_S}}\O_{X_S}(1)^{\otimes i}$ is globally generated and $H^1(X_S,(\pi_*\E)(i))=0$.
    Note that 
    \[(\pi_*\E)(i)=\pi_*(\E\otimes_{\O_{C_S}}\pi^*\O_{X_S}(i))=\pi_*(\E(i))\]
    by the projection formula.
    By equation \eqref{eq:cohom}, $H^1(C_S,\E(i))=0$ for all $i\ge N$.
    Finally, as $(\pi_*\E)(i)$ is globally generated, there is an $m\ge 1$ and a surjection
    \[\O_{X_S}^m \twoheadrightarrow (\pi_*\E)(i)=\pi_*(E(i)).\]
    Applying $\pi^*$ gives a surjection 
    \[g:\O_{C_S}^m \twoheadrightarrow \pi^*\pi_*(E(i)).\]
    Now, observe that the natural map 
    \[h:\pi^*\pi_*(E(i))\to E(i)\]
    is surjective, as can be checked on stalks.
    Hence we get a surjection $h\circ g:\O_{X_S}^m \twoheadrightarrow E(i)$, showing that $E(i)$ is globally generated.
\end{proof}

We shall also need a general topological lemma. Let us write qc as a shorthand for `quasi-compact'. Recall that a topological space is called quasi-separated if the intersection of any two qc open subsets is qc. We refer the reader to \cite[section 2]{EC_diamonds} for the definitions of spectral and locally spectral spaces. Recall that every adic space is a locally spectral space. Also, a locally spectral space is spectral if and only if it is quasi-compact and quasi-separated (qcqs).
\begin{lemma}\label{lem:qcqs}
    Let $Z$ be a locally spectral space. 
    Assume that $Z$ admits a finite open cover $\{U_i\}_{i=1}^m$ such that $U_i$ and $U_i\cap U_j$ are qc for all $i,j$.
    Then $Z$ is qcqs.
\end{lemma}

\begin{proof}
    Quasi-compactness of $Z$ is clear.
    The proof for quasi-separatedness is almost identical to that in \cite{quasiseparated}.
    Let $V$ be a qc open subset of $Z$. Let us show that $V\cap U_1$ is qc.
    Since $Z$ is locally spectral and $V$ is qc, we can write
    \[V=\bigcup_{i=1}^m \bigcup_{j=1}^M V_{ij}\]
    for some qc open $V_{ij}\subset U_i$ and $M\in \mathbb N$.
    Then 
    \begin{align*}
        V\cap U_1 &=\bigcup_{i=1}^m \bigcup_{j=1}^M (V_{ij}\cap U_1)\\
        &= \bigcup_{i=1}^m \bigcup_{j=1}^M (V_{ij}\cap U_1\cap U_j).
    \end{align*}
    By assumption, each $U_1\cap U_j$ is qc. 
    Since $V_{ij}$ and $U_1\cap U_j$ are both qc open subsets of $U_j$, we know that $(V_{ij}\cap U_1\cap U_j)$ is also qc by the quasi-separatedness of $U_j$.
    Hence $V\cap U_1$ is also qc.

    Now, suppose $W$ is also a qc open subset of $Z$.
    Then $W\cap U_1$ is qc by the above.
    By quasi-separatedness of $U_1$, $V\cap W\cap U_1$ is also qc. By symmetry, the same holds with $U_1$ replaced by any of the $U_i$. 
    Hence, $V\cap W=\cup_{i=1}^m(V\cap W\cap U_i)$, is also qc.
\end{proof}

We have the following GAGA result:
\begin{theorem}
    Suppose $S$ is affinoid perfectoid.
    Let $$C_S^{sch}:=\Proj\left(\bigoplus_{m\ge 0}H^0(C_S,\O_{C_S}(m))\right).$$
    Then we have a natural map of locally ringed spaces 
    \[f:C_S\to C_S^{sch}\]
    and pullback along this map induces an equivalence of categories 
    \begin{equation*}
        f^*:\{\text{vector bundles on }C_S^{sch}\}\xrightarrow{\simeq}  \{\text{vector bundles on }C_S\}.
    \end{equation*}
    In addition, for every vector bundle $\E$ on $C_S^{sch}$, we have
    \[H^i(C_S^{sch},\E)=H^i(C_S,f^*(\E))\]
    for all $i\ge 0$.
\end{theorem}

\begin{proof}
    For affinoid perfectoid $S$, $X_S$ is qcqs. Thus, there is a finite open cover $\{U_i\}_{i=1}^m$ such that $U_i\cap U_j$ is qc for all $i,j$.
    Since $\pi:C_S\to X_S$ is a covering of degree $n$, $\pi$ is quasi-compact, i.e. the preimage of any qc open subset is qc.
    It follows that $\{\pi^{-1}(U_i)\}_{i=1}^m$ is an open cover of $C_S$ satisfying the condition of \Cref{lem:qcqs}. Thus, $C_S$ is qcqs, and hence is a spectral space.
    The theorem now follows from \Cref{ample} and the general GAGA theorem in \cite[Proposition II.2.7]{FS_geometrisation}.
\end{proof}

\begin{remark}
    It is conceivable that for $S=\Spa K$, where $K$ is an algebraically closed perfectoid field over $\mathbb F_q$, one can establish an analogue of \cite[Proposition II.2.9]{FS_geometrisation} about the properties of $C_{S}^{sch}$. Once this is established, one should be able to get an analogue of \cite[Proposition II.2.10]{FS_geometrisation} on the classification of line bundles on $C_{S}^{sch}$ and hence a concrete description of the Hitchin fiber via \Cref{prop:CRT}.
\end{remark}

\section{Hitchin fibration and action of Picard stack}\label{sec: more}
The aim of this section is to study the Hitchin fibration for general $G$ and study the action of a certain Picard stack on $\Higgs$ (cf. \Cref{sec: picard stacks}.)
Starting from this section, we will mostly use the schematic version of the Fargues-Fontaine curve instead of the adic space version.

Let $E$ be a non-Archimedean local field with residue field $\mathbb F_q$, where $q$ is a power of $p$. 
Let $G$ be a split connected reductive group over $E$, $B$ be a Borel subgroup over $E$, $T$ be a split maximal torus in $B$. 
Assume $char E\nmid |W|$, where $W$ is the Weyl group of $G$.
We let $\g$ be the Lie algebra of $G$, thought of as a scheme over $E$.\footnote{I.e. $\g=\Spec Sym_E (Lie(G)^\vee)$, where $Lie(G)=\ker(G(E[\epsilon]/(\epsilon^2))\to G(E))$.}
Let $\mathfrak t$ be the Lie algebra of $T$.

Fix an affinoid perfectoid space $S$ over $\mathbb F_q$ and a line bundle $\mathcal L$ on $X_S^{sch}$.
As a rule of thumb, in this section, every stack is either a stack on $\Perf_S$ with the v-topology or a stack on $\Sch_{X_S^{sch}}$ with the fppf topology.

\subsection{Hitchin fibration}
Let $\c:= \g \sslash G$ be the GIT quotient of $\g$ by the adjoint action of $G$, i.e.
\[\c:=\g\sslash G:= \Spec \O(\g)^G\]
where $\O(\g)$ is the ring of global sections of the structure sheaf of $\g$.
Let us recall the meaning of $\O(\g)^G$: 
if we view elements of $\O(\g)$ as morphisms $\g\to \mathbb A^1$, then $\O(\g)^G$ are the morphisms such that the two compositions
\[\xymatrix{G\times \g\ar[r]<1.5pt>^-{\pi_2}\ar[r]<-1.5pt>_-{\text{action}} & \g}\to \mathbb A^1\]
agree. In other words, 
\begin{equation}\label{eq:invariant}
    \O(\g)^G = eq\left( \xymatrix{\O(\g)\ar[r]<1.5pt>^-{1\otimes -}\ar[r]<-1.5pt>_-{\text{action}^*} & \O(G)\otimes_E \O(\g)} \right).
\end{equation}

Let us recall the Chevalley restriction theorem.
\begin{proposition}\label{Chevalley-res}\cite[Theorem 2.1]{Ngo_2010}
    The inclusion $\mathfrak t \to \g$ induces an isomorphism $$\O(\g)^G \iso \O(\mathfrak t)^W.$$ Moreover, $\O(\g)^G$ is isomorphic to the polynomial algebra $E[T_1,\dots,T_r]$ over $E$, where $r$ is the rank of $G$, and we can pick the isomorphism such that $T_1,\dots,T_r$ correspond to homogeneous elements of $\O(\g)=Sym_E(Lie(G)^\vee)$. The degrees of these elements are independent of their choices.
\end{proposition}

The proposition implies that $\c\cong \mathbb A^r_E$ over $E$. 
The inclusion $\O(\g)^G\subset \O(\g)$ induces a map $$\g\to\c. $$ 
This map is invariant under the adjoint action of $G$ on $\g$ by equation \eqref{eq:invariant}. 
We base change both sides along $X_S^{sch}\to \Spec E$ to get
$$\g_{X_S^{sch}}\to\c_{X_S^{sch}}. $$ 
We can view both sides as vector bundles on $X_S^{sch}$ and tensor both sides with $\mathcal L$ to get a $G$-equivariant map
\begin{equation}\label{eq1}
    \g_\L\to \mathfrak c_\L
\end{equation}
where $\mathfrak g_\L:= \mathfrak g_{X_S^{sch}} \otimes_{\O_{X_S^{sch}}}\L$ with the adjoint $G$-action and $\mathfrak c_\L:= \mathfrak c_{X_S^{sch}} \otimes_{\O_{X_S^{sch}}}\L$ with the trivial $G$-action.

\begin{definition}
    We call $\mathcal A_{G,\L}:= \mathcal H^0(\mathfrak c_\L)$ the Hitchin base.
\end{definition}
Note that this is a Banach-Colmez space. In particular, it is a locally spatial diamond. When there is no ambiguity, we will write $\mathcal A_{G,\L}$ as $\mathcal A_G$.

\begin{remark}\label{concrete A_G}
    We can be more concrete about the definition of $\mathcal A_G$.
    Pick homogeneous $f_1,\dots,f_r\in \O(\g)^G$ such that the map $E[T_1,\dots,T_r] \to \O(\g)^G$, $T_i\to f_i$ is an isomorphism.
    Let $\deg(f_i)=d_i$.
    Then we get an isomorphism of schemes
    \begin{align*}
        \c \iso \,&\mathbb A_E^r\\
        f_i\mapsfrom \,&T_i.
    \end{align*}
    This induces a $\mathbb G_m$-equivariant isomorphism of sheaves on $X_S^{sch}$
    \[\mathfrak c_{X_S^{sch}}\cong \O_{X_S^{sch}}^r,\]
    where $\mathbb G_m$ acts on the $i$-th component of $\O_{X_S^{sch}}^r$ via $t\mapsto t^{d_i}$.
    Twisting both sides by the $\mathbb G_m$-torsor $\tilde \L=\underline{Isom}_{\O^{sch}_{X_S}}(\O^{sch}_{X_S},\L)$, we obtain
    \[\mathfrak c_\L \cong \L^{\otimes d_1}\oplus\dots\oplus \L^{\otimes d_r}\]
    because on the $i$-th component, we have an isomorphism\footnote{Note that the action of $\mathbb G_m$ on $\O^{sch}_{X_S}$ is different on each component.} $\O^{sch}_{X_S}\times^{\mathbb G_m}\tilde\L \iso \L^{\otimes d_i}$ given by sheafifiying $(z,f)\mapsto f(z)\otimes f(1)\otimes\dots\otimes f(1)$.
    It follows that 
    \[\mathcal A_G\cong \prod_{i=1}^r \mathcal H^0(\L^{\otimes d_i}).\]
    For example, for $G=\GL_n$, if we pick $f_i=(-1)^i \Tr(\Lambda^i(\;))$ for $1\le i\le n$, then we get back \Cref{A_in_GL_n_case}. 
\end{remark}

\begin{remark}
    From the analogue of \Cref{twist_of_line_bundle} for the fppf site on schemes, we can deduce that if $\E$ is a $G$-bundle on $X_S^{sch}$, then
\begin{align*}
    \E\times^G \mathfrak g_\L &\cong \E\times^G (\tilde\L \times^{\mathbb G_m} \mathfrak g_{X_S^{sch}}) \\
    &= \tilde\L \times^{\mathbb G_m} (\E\times^G \mathfrak g_{X_S^{sch}}) \\
    &= \tilde\L \times^{\mathbb G_m} \Ad(\E) \\
    &\cong\L\otimes_{\O_{X_S^{sch}}} \Ad(\E)
\end{align*}
where we use the analogue of \Cref{twist_of_line_bundle} for the first and last isomorphisms.
\end{remark}

\begin{definition}\label{defn: Hitchin fibration}
    We define the \emph{Hitchin fibration} to be the map
    \begin{equation}\label{Hitchin_fibration}
        \Higgs\to \mathcal A_G
    \end{equation}
    which on $T$-valued points (for affinoid perfectoid $T$) is given by sending $(\E,\phi)\in \Higgs(T)$ to its image under 
    \[\Ad(\E)_\L(X_T^{sch})=(\E\times^G \g_\L)(X_T^{sch})\to (\E\times^G \mathfrak c_\L)(X_T^{sch})=
    \mathfrak c_\L(X_T^{sch}),\]
    where the middle map is induced from equation \eqref{eq1}. We call the image of $(\E,\phi)$ under \eqref{Hitchin_fibration} its \emph{characteristic polynomial}.
\end{definition}

\begin{remark}\label{rmk: map FF}
        Note that by \Cref{alt_defn_for_quotient_stack}, we have $$\Higgs= \Map(FF,[\mathfrak g_\L/G]).$$
        Here, $\Map(FF,[\mathfrak g_\L/G])$ denotes \footnote{Note that by 2-Yoneda lemma, this is also the prestack sending $T$ to the groupoid of morphisms $X_T^{sch}\to[\g_\L/G]$ of stacks on $\Sch_{X_S^{sch},fppf}$. This justifies the notation $\Map(FF,-)$.} the prestack on $\Perf_{S}$ sending $T$ to the groupoid $[\mathfrak g_\L/G](X_T^{sch})$, where $G$ acts on $\mathfrak g_\L$ via the adjoint action.
        Similarly, we have $$\mathcal A_G:= \Map(FF,\mathfrak c_\L).$$
        The Hitchin fibration is obtained by applying $\Map(FF,\;)$ to
        \begin{equation}
            [\g_\L/G]\to \mathfrak c_\L.
        \end{equation}
\end{remark}

\begin{example}
    Assume $G=\GL_n$. Let us show that we recover \Cref{Hitchin fibration for GL_n}.
    To identify $\mathcal A_G$ with $\prod_{i=1}^{n}\mathcal H^0(\L^{\otimes i})$ as in \Cref{concrete A_G}, we pick $f_i=(-1)^i \Tr(\Lambda^i(\;))$ for $1\le i\le n$ as a basis of $\O(\g)^G$. 

    Let $\mathcal V$ be a rank $n$ vector bundle on $X_S^{sch}$. Its corresponding $\GL_n$-bundle is $\E:=\Isom_{\O_{X_S^{sch}}}(\O_{X_S^{sch}}^n, \mathcal V)$.
    Let $\phi\in H^0(X_S^{sch},\Ad(\E)_\L) = \Hom_{\O_{X_S^{sch}}}(\V,\V\otimes \L).$
    We want to show that the image of $(\E,\phi)$ (via \Cref{defn: Hitchin fibration}) and the image of $(\V,\phi)$ (via \Cref{Hitchin fibration for GL_n}) in $\c_\L$ agree.
    It suffices to check this on an open cover of $X_S^{sch}$.
    Let $U=\Spec B$ be an open affine subset of $X_S^{sch}$ such that $\L$, $\E$, $\V$ are all trivial.
    Then $\phi|_U$ corresponds to an element $M\in \g(U)=\mathfrak{gl}_{n\times n}(B)$.
    The restriction of the map in \Cref{defn: Hitchin fibration} to $U$ is given by
    \begin{align*}
        \g(A) &\to \mathfrak c(A) = (\Spec E[f_1,\dots,f_n])(A) \cong A^n\\
        M &\mapsto (f_i(M))_{1\le i \le n} = ((-1)^i \Tr(\Lambda^i M))_{1\le i \le n}.
    \end{align*}
    This agrees with \Cref{Hitchin fibration for GL_n}.
\end{example}

\subsection{$\g^{reg}$}
\begin{lemma}
    The functor given by
    \begin{align*}
        &I:\set{\text{schemes over $E$}}^{op}\to \mathrm{Set}\\
        &U\mapsto \set{(g,v)\in G(U)\times \g(U): \Ad_g(v)=v}
    \end{align*}
    is representable by an affine scheme of finite type over $E$.
\end{lemma}

\begin{proof}
    Fix an embedding $G\hookrightarrow \GL_n$ of group schemes.
    Let \begin{align*}
        &I':\set{\text{schemes over $E$}}^{op}\to \mathrm{Set}\\
        &U\mapsto \set{(g,v)\in \GL_n(U)\times \mathfrak{gl}_n(U): \Ad_g(v)=v}.
    \end{align*}
    Then $I=I'\times_{\GL_n\times\mathfrak{gl}_n}(G\times \g)$, so it suffices to prove the lemma for $I$ replaced by $I'$.
    Write
    $ \begin{pmatrix}
X_{11} & \cdots & X_{1n} \\
\vdots &  & \vdots \\
X_{n1} & \cdots & X_{nn} \\
\end{pmatrix}  
\begin{pmatrix}
    Y_1 \\
    \vdots \\
    Y_n \\
    \end{pmatrix}  
    \begin{pmatrix}
        X_{11} & \cdots & X_{1n} \\
        \vdots &  & \vdots \\
        X_{n1} & \cdots & X_{nn} \\
        \end{pmatrix}^{-1}
        -
        \begin{pmatrix}
            Y_1 \\
            \vdots \\
            Y_n \\
            \end{pmatrix} 
        =
        \begin{pmatrix}
            f_1 \\
            \vdots \\
            f_n \\
            \end{pmatrix} 
$.
Then $I'$ is represented by 
$$\Spec E[X_{11},\dots,X_{nn},Y_1,\cdots,Y_n][1/det(X_{ij})]/(f_1,\dots,f_n).$$
\end{proof}

We shall identify $I$ with the scheme that represents it. 
Note that there is a projection map $I \to \g$, $(g,v)\mapsto v$. Also, $I$ is a group scheme over $\g$.

\begin{lemma}\cite{semi-continuity}
    Let $H$ be a group scheme locally of finite type over a scheme $T$. Then the function $f:T\to \Z$, $t\mapsto dim(H_t)$ is upper semi-continuous, i.e. for all $n\in\Z$, $\set{t\in T: \dim(H_t)\le n}$ is open.
\end{lemma}

\begin{proof}
    This is the proof in \cite{semi-continuity}. We include it here as \cite{semi-continuity} is subject to change.

    Let $s:H\to T$ be the structure map.
    Let $T\xrightarrow{e} H$ be the unit map.
    By \cite[lemma 045X]{stacks-project}, $f$ is the composite of the maps $e$ and 
    \begin{align}
        H&\to \Z \label{semi}\\
        h&\mapsto \dim_h(H_{s(h)}).\nonumber
    \end{align}
    As $e$ is continuous and \eqref{semi} is upper semi-continuous by \cite[lemma 02FZ]{stacks-project}, their composite is also upper semi-continuous.
\end{proof}

Applying this lemma to $I\to \g$, we know that $\{v\in \g:\dim I_v\le r\}$ is open in $\g$.
\begin{definition}
    Define $\g^{reg}$ to be the open subscheme of $\g$ whose underlying subset is $\{v\in \g:\dim I_v\le r\}$.
\end{definition}
In fact, for every $v\in \g$, we have $\dim I_v= r$.

Note that $G\times \mathbb G_m$ acts on $I$, where $G$ acts via $h\cdot (g,v) := (hgh^{-1},\Ad_h(v))$ and $\mathbb G_m$ acts via $t\cdot (g,v) := (g,tv)$. 

Let us recall the construction of the Kostant section, as in \cite[section 2.3]{Chen_Zhu}.
For each simple root $\alpha_i$, pick $f_i\in \g_{-\alpha_i}(E)\setminus\set{0}$.
Let $f=\sum f_i$.
Complete $f$ into an $\mathfrak{sl}_2$-triple $e,f,h$.
Let $\g^e$ be the centraliser of $e$ in $\g$. 
\begin{proposition}[Kostant]\cite[Theorem 2.1]{Ngo_2006}
    $f+\g^e \subset\g^{reg}$ and the restriction of $\g \to \c=g//G$ to $f+\g^e$ gives an isomorphism
    \begin{equation}\label{eq: res}
        f+\g^e \iso \c.
    \end{equation}
\end{proposition}
The \emph{Kostant section} $\Kos:\c \to \g$ is the composition of the inverse of \eqref{eq: res} with the inclusion $f+\g^e \hookrightarrow \g$.
Then $\Kos$ is a section of the natural map $\g \to \c$.
Note that the Kostant section actually lands in $\g^{reg}$.

We define $\g^{reg}_\L:=\tilde \L\times^{\mathbb G_m} (\g^{reg}\times_{\Spec E}X_S^{sch})$ and $\tilde \L=\Isom(\O_X,\L)$. A priori, this is a fppf sheaf on $\Sch_{X_S^{sch}}$, but this is representable by a scheme by fppf descent of quasi-affine schemes \cite[0247]{stacks-project}, so we will also regard it as a scheme over $X_S^{sch}$.

\begin{lemma}\label{v are conjugate}
    Let $U$ be a scheme over $X_S^{sch}$. Assume $v_1,v_2\in \g^{reg}_\L(U)$ have the same image in $\c_\L(U)$. Then there is a fppf cover $W\to U$ such that $v_1|_W, v_2|_W$ are $G(W)$-conjugate.
\end{lemma}

\begin{proof}
    Let $c$ be the common image of $v_1,v_2$ in $\c_\L(U)$.
    Fix a Kostant section $\Kos:\c \to \g$.
    The morphism $G\times \c \to \g^{reg}$, given on $Z$-valued points (for $Z\in \Sch_E$) by $(g,c)\mapsto \Ad_g(\Kos(c))$, is a smooth and surjective morphism of schemes over $\c$, as stated in \cite[proof of Proposition 3.4]{Ngo_2006}.
    Note that this map is $\mathbb G_m$-equivariant.
    Thus, $G\times \c_\L \to \g^{reg}_\L$ is a smooth and surjective morphism of schemes over $\c_\L.$
    It follows that 
    \begin{equation}\label{G c to g reg}
        G\times \c_\L \to \g^{reg}_\L
    \end{equation}
    is surjective as morphism of fppf sheaves on $\Sch_{X_S^{sch}}$.

    Thus, there is a fppf cover $W\to U$ and $(g_1,c_1)\in (G\times \c_\L)(W)$ such that $\Ad_{g_1}(\Kos(c_1)) = v_1$. As \eqref{G c to g reg} is a morphism over $\c_\L$, $c_1=c|_W$. Hence $\Ad_{g_1}(\Kos(c)) = v_1$.
    Similarly, up to replacing $W$ by a fppf cover, there is $g_2\in G(W)$ such that $\Ad_{g_2}(\Kos(c)) = v_2.$
    Thus, $v_1|_W, v_2|_W$ are $G(W)$-conjugate.
\end{proof}

\subsection{Regular centraliser}\label{sec:centraliser}

Let $I|_{\g^{reg}}:= I \times_{\g} \g^{reg}$.
It turns out that the group scheme $I|_{\g^{reg}} \to \g^{reg}$ descends to a group scheme over $\c$:
\begin{lemma}\cite[Lemme 2.1.1]{Ngo_2010}
    There is a unique smooth commutative group scheme $J$ over $\c$ equipped with a group scheme isomorphism over $\g^{reg}$
    \begin{equation}\label{eq: isom J, I}
        J\times_{\c}\g^{reg} \iso I|_{\g^{reg}}.
    \end{equation}
    that is compatible with the $G\times\mathbb G_m$-action.
    This extends uniquely to a homomorphism over $\g$ 
    \begin{equation}\label{J to I}
        J\times_{\c}\g \to I.
    \end{equation}
    that is compatible with the $G\times\mathbb G_m$-action.
\end{lemma} 

\begin{remark}\cite[after Lemme 2.1.1]{Ngo_2010}\label{rmk: alt defn of J}
    It we pick a Kostant section $Kos:\c \to \g$, then $J\cong I\times_{\g,Kos}\c$.
\end{remark}

As the projection map $I\to \g$ is $G\times\mathbb G_m$-equivariant, we can base change both sides to $X_S^{sch}$ and twist the map by the $\mathbb G_m$-torsor $\tilde{\L}=\underline{Isom}(\O,\L)$ to get a $G$-equivariant map
$$I_{\L}\to \g_{\L}$$
of fppf sheaves on $\Sch_{X_S^{sch}}$, where $I_\L:=I_{X_S^{sch}}\otimes_{\O_{X_S^{sch}}}\L$.
Similarly, we can twist $J\to \mathfrak c$ to get
$$J_{\L}\to \mathfrak c_{\L}.$$
Note that for all $U\in\Sch_{X_S^{sch}}$, we have a canonical isomorphism
\begin{equation}\label{eq:I_L}
     I_\L(U)\cong\{(g,v)\in G(U)\times \g_\L(U):\Ad_g(v)=v\}.
\end{equation}
To see this, note that the right hand side defines a sheaf and there is an obvious map from $I_\L$ to it, which is an isomorphism locally (and hence globally).
Similarly, if we fix a Kostant section $\Kos$, then for all $U\in \Sch_{X_S^{sch}}$, we have a canonical isomorphism
\begin{equation}\label{eq:J_L}
     J_\L(U)\cong\{(g,c)\in G(U)\times \c_\L(U):\Ad_g(\Kos(c))=\Kos(c)\}.
\end{equation}

\subsection{Action of Picard prestack on Higgs bundles}\label{sec:action of Picard}
Our aim in this subsection is to define the action of a certain Picard prestack $\mathcal P$ on $\Higgs$.

We can twist \eqref{J to I} by $\L$ to get a group homomorphism over $\g_\L$
\begin{equation}\label{input data}
    J_\L\times_{\mathfrak c_\L}\g_\L\to I_\L
\end{equation}
that is compatible with the $G$-action.
This induces a map
\begin{equation}\label{eq: quot J to quot I}
    [(J_\L\times_{\mathfrak c_\L}\g_\L)/G]\cong J_\L\times_{\mathfrak c_\L}[\g_\L/G] \to [I_\L/G].
\end{equation}

Note that $[I_\L/G]=\mathcal I_{[\g_\L/G]}$, the inertia stack of $[\g_\L/G]$.
Indeed, it is clear from the definition and equation \eqref{eq:I_L} that the inertia prestack of the quotient prestack\footnote{We mean the prestack $S\mapsto [\g_\L(U)/G(U)]$.} $[\g_\L/_{pre}\;G]$ is $[I_\L/_{pre}\;G]$.
Recall that on prestacks, the operations of taking stackification and taking inertia prestack commutes \cite[Lemma 06NS]{stacks-project}. It follows that $[I_\L/G]$ is the inertia stack of $[\g_\L/G]$.

Thus, we can write \eqref{eq: quot J to quot I} as
\begin{equation}\label{eq: final}
    J_\L\times_{\mathfrak c_\L}[\g_\L/G] \to \mathcal I_{[\g_\L/G]}
\end{equation}
which is a morphism over $[\g_\L/G]$ satisfying the condition of \Cref{inertia gives action over a base} because \eqref{input data} is a group homomorphism over $\g_\L$.
By \Cref{inertia gives action over a base}, we get an action \begin{equation}\label{action-primitive}
    [\mathfrak c_\L/J_\L]\times_{\mathfrak c_\L}[\g_\L/G] \to [\g_\L/G]
\end{equation}
over $\mathfrak c_\L$.

\begin{example}
    Take $G=\mathbb G_m$.
    Let $\Tot(\L)\to X_S^{sch}$ be the geometric vector bundles corresponding to $\L$.
    Then $\g=\mathfrak c =\mathbb A^1_E$, $\g_\L=\Tot(\L)$ with trivial $G$-action, and $I_{\L}=J_{\L}=\mathbb G_m\times_E \Tot(\L)$.
    By definition, \eqref{action-primitive} is the morphism $[\L/G] \times_{\L} [\L/G] \to [\L/G]$ given by taking the stackification of the map
    \begin{align*}
        [\L/_{pre}\; G] \times_{\L} [\L/_{pre}\; G] &\to [\L/_{pre}\; G]\\
        (v,v)&\mapsto v &&\text{on objects}\\
        (g_1,g_2)&\mapsto g_1 g_2 &&\text{on morphisms.}
    \end{align*}
    Thus, if we identify $[\L/G]$ as the stack taking a scheme $X$ to the groupoid of pairs $(\mathcal F, v)$, where $\mathcal F$ is a line bundle on $X$ and $v\in H^0(X,\L)$ is a global section of the pullback of $\L$ to $X$, then \eqref{action-primitive} is the morphism 
    \begin{align*}
        [\L/G] \times_{\L} [\L/G] &\to [\L/G]\\
        ((\mathcal F_1,v), (\mathcal F_2,v)) &\mapsto (\mathcal F_1 \otimes \mathcal F_2,v).
    \end{align*}
\end{example}

Let us return to the general case.
\begin{definition}
    Define $\mathcal P:=\Map(FF,[\mathfrak c_\L/J_\L])$, i.e. the prestack
    \begin{align*}
        \Perf_S^{op} &\to \Grpd\\
        T&\mapsto \set{(a,Q):a\in\c_\L(X_T^{sch}),\; Q\text{ is a } J_\L\times_{\mathfrak c_\L, a}X_T^{sch} \text{-torsors on }X_T^{sch}}.
    \end{align*}
    We also define $J_a := J_\L\times_{\mathfrak c_\L,a}X_T^{sch}$.
\end{definition}

\begin{remark}\label{rmk: P is v stack}
    If $S=\Spa C$ for an algebraically closed perfectoid field $C$ over $\mathbb F_q$, then for $T\to S$, we can by \Cref{equiv_G_bundles} identify $J_a$-bundles on $X_T^{sch}$ with exact $\O_{\mathfrak c_\L}$-linear symmetric monoidal functors $\Rep_{\mathfrak c_\L}(J_\L) \to \mathrm{VB}(X_T^{sch}) \simeq \mathrm{VB}(X_T)$. Since $T\mapsto \mathrm{VB}(X_T)$ is a v-stack, $\P$ is also a v-stack by \Cref{lem: exact 1}.
\end{remark}

Taking $\Map(FF,-)$ of \eqref{action-primitive}, we get an action
\begin{equation}\label{action-P on HIggs}
    \mathcal P\times_{\mathcal A_G} \Higgs \to \Higgs
\end{equation}
of $\mathcal P$ on $\Higgs$ over $\mathcal A_G$.
Let us denote the image of $(Q,x)$ by $Q*x$.

\begin{lemma}\label{lem:J isom}
    Let $(\E,\phi)\in\Higgs(T)$, which corresponds to a map $X_T^{sch} \to [\g_\L/G]$. Let $a$ be the composite map $X_T^{sch} \to [\g_\L/G] \to \c$.
    Let $Z\to X_T^{sch}$ be a scheme over $X_T^{sch}$. If the image of $Z\to X_T^{sch} \to [\g_\L/G]$ lies in $[\g^{reg}_\L/G]$, then $J_a|_Z \isomto \underline{Aut}(\E,\phi)|_Z$.
\end{lemma}

\begin{proof}
    The base change of \eqref{eq: final} along $[\g^{reg}_\L/G]\to [\g_\L/G]$, i.e. the map
    \begin{equation}\label{eq:eq a}
        J_\L \times_{\mathfrak c_\L} [\g^{reg}_\L/G] \iso  \mathcal I_{[\g^{reg}_\L/G]},
    \end{equation}
    is an isomorphism, because it is induced from isomorphism \eqref{eq: isom J, I} (in the same way that \eqref{eq: final} is induced from \eqref{J to I}).
    Base changing \eqref{eq:eq a} along the map $Z\to X_T^{sch} \to [\g^{reg}_\L/G]$ gives an isomorphism
    \[J_a|_Z \isomto Z\times_{[\g^{reg}_\L/G]} \mathcal I_{[\g^{reg}_\L/G]}\cong \underline{Aut}(\E,\phi)|_Z.\]
\end{proof}

We have the following analogue of \cite[Proposition 2.2.1]{Ngo_2010}.
\begin{proposition}\label{Higgs^reg is a gerbe}
    Let $\Higgs^{reg}:=\Map(FF,[\g^{reg}_\L/G])$.
    Then $$\Higgs^{reg}\cong \mathcal P$$ as prestacks. This isomorphism is canonical if we choose a Kostant section.
\end{proposition}

\begin{proof}
    We first prove that $[\g^{reg}_\L/G]$ is a neutral gerbe over $\c_\L$  as stacks on $(\Sch_{X_S^{sch}})_{fppf}$.
    Fix a Kostant section $\Kos:\c\to \g^{reg}$. Then the map $\c_\L\to \g^{reg}_\L \to [\g^{reg}_\L/G]$ provides an object in $x=[\g^{reg}_\L/G](\c_\L)$.
    Let $U\in\Sch_{X_S^{sch}}$.
    Suppose $(\E_1,\phi_1),(\E_2,\phi_2)\in [\g^{reg}_\L/G](U)$ have the same image in $\c_\L(U)$.
    We would like to show that locally in $U$, $(\E_1,\phi_1)$ and $(\E_2,\phi_2)$ are isomorphic.
    By passing to a fppf cover of $U$, we may assume $\E_1,\E_2$ are trivial $G$-bundles and $\L\cong \O$. Then $\phi_1,\phi_2\in \g(U)$ have the same image in $\c(U)$. Since $\phi_1,\phi_2$ are regular, they are $G(U)$-conjugate after passing to a fppf cover $V$ of $U$ by \Cref{v are conjugate}. Thus, $(\E_1,\phi_1)|_V\cong (\E_2,\phi_2)|_V$. Hence, $[\g^{reg}_\L/G]$ is a neutral gerbe over $\c_\L$ by \cite[Lemma 06P1]{stacks-project}.

    Let us calculate $\underline{\Aut}(x)$. Let $U\to \c_\L$, which corresponds to $c\in \c_\L(U)$.
    Then we can identify $x|_U$ with $(\E_0,\Kos(c))\in [\g^{reg}_\L/G](U)$, where $\E_0$ is the trivial $G$-bundle on $U$ and $\Kos(c)$ is viewed as a global section of $\Ad(\E_0)_\L$.
    Thus, $\underline{\Aut}(x)(U)=\set{g\in G(U):\Ad_g(\Kos(c))=\Kos(c)}$.
    By \eqref{eq:J_L}, $\underline{\Aut}(x)=J_\L$.

    By a general property of neutral gerbe, we have an isomorphism $[\c_\L/J_{\L}] \isomto [\g^{reg}_\L/G]$ given by the stackification of the map 
    \begin{align}\label{c to g}
        \c_\L(U)/J_\L(U) &\to [\g^{reg}_\L/G](U)\\
        c&\mapsto (\E_0,\Kos(c))\\
        j&\mapsto j.
    \end{align}

    Applying $\Map(FF,-)$ to $[\c_\L/J_\L] \isomto [\g^{reg}_\L/G]$, we obtain $\mathcal P \isomto \Higgs^{reg}$.
\end{proof}

\begin{remark}
    By \Cref{rmk: P is v stack} and \Cref{Higgs^reg is a gerbe}, we know that if $S=\Spa C$, then $\Higgs^{reg}$ is a v-stack.
\end{remark}

We have two actions of $\mathcal P$ on $\Higgs^{reg}$ over $\mathcal A_G$:
\begin{enumerate}
    \item the one in \eqref{action-P on HIggs} restricted\footnote{The action preserves the regularity condition, because the map $J\times_{\mathfrak c}\g \to I$ restricts to a map $J\times_{\mathfrak c}\g^{reg} \to I|_{\g^{reg}}$.} to $\Higgs^{reg}$
    \item the one obtained by identifying $\Higgs^{reg}\cong \mathcal P$ and using the contracted product on $\mathcal P$. 
\end{enumerate}
These two actions are equivalent.

\begin{lemma}\label{2 actions on Higgs}
    The two actions above are equivalent.
    More precisely, the two $1$-morphisms $\mathcal P \times_{\mathcal A} \Higgs^{reg} \to \Higgs^{reg}$ are isomorphic in the $2$-category of prestacks on $\Perf_S$.
\end{lemma}

\begin{proof}
    It suffices to show that the two actions of $[\c_\L/J_\L]$ on $[\g_\L^{reg}/G]$ over $\c_\L$ are isomorphic. 
    In other words, we want to show the two $1$-morphisms $[\c_\L/J_\L]\times_{\c_\L} [\g_\L^{reg}/G] \to [\g_\L^{reg}/G]$ are isomorphic.
    Equivalently, by identifying $[\c_\L/J]$ with $[\g_\L^{reg}/G]$ via the map \eqref{c to g}, we need to show 
    the two $1$-morphisms $$[\c_\L/J_\L]\times_{\c_\L} [\c_\L/J] \isomto [\c_\L/J_\L]\times_{\c_\L} [\g_\L^{reg}/G] \rightrightarrows [\g_\L^{reg}/G]$$ are isomorphic.
    By property of stackification \cite[Lemma 04W9]{stacks-project}, it suffices to see the two $1$-morphisms $[\c_\L/_{pre}\;J] \times_{\c_\L} [\c_\L/_{pre}\;J_\L] \to [\g_\L^{reg}/G]$ are isomorphic.
    This can be checked by chasing through the definitions: on objects, they both send $(c,c)$ to $(\E_0,\Kos(c))$; on morphisms, they both send $(j_1,j_2)$ to $j_1j_2$.
\end{proof}

\section{Affine Springer fiber}
In this section, we define and study the $B_{dR}$-affine Springer fibres. One can view them as the local analogues of Hitchin fibers.

Recall that \cite[section VI.1]{FS_geometrisation} if $S$ is a perfectoid space over $\mathbb F_q$, then $Div^1_X(S)$ is the set of closed Cartier divisors of degree $1$ of $X_S$. Locally in the analytic topology of $S$, any such Cartier divisor arises from an untilt of $S$ over $E$. If $S$ is affinoid and $S\to Div^1_X$ is a map whose corresponding Cartier divisor $D_S$ arises from an untilt of $S$ over $E$, then we define\footnote{Note that $B_{\Div^1_X}^+(S)$ depends on the map $S\to Div^1_X$ and not just $S$.}
\[B_{\Div^1_X}^+(S)\]
to be the global sections of the completion of $\O_{X_S}$ along $\mathcal I_S$, where $I_S\subset \O_{X_S}$ is the ideal sheaf corresponding to $D_S$, and
\[B_{\Div^1_X}(S)\]
to be $B_{\Div^1_X}^+(S)[\frac{1}{I_S}]$
where $I_S$ is the ideal of $B_{\Div^1_X}^+(S)$ determined by $\mathcal I_S$.
If we pick an untilt $S^\sharp=\Spa(R, R^+)$ of $S$ which gives rise to $D_S$, then we can identify $B_{\Div^1_X}^+(S)$ with $B_{dR}^+(R)$ and $B_{\Div^1_X}(S)$ with $B_{dR}(R)$.

In the remainder of this paper, let $\mathbb D_T$ denote $\Spec B^+_{Div^1}(T)$ and $\mathbb D_T^*$ denote $\Spec B_{Div^1}(T)$. Note that this depends on the map $T\to Div^1_X$, not just on $T$. Also let $\E_0$ denote the trivial $G$-bundle.

Motivated by \cite[Definition 3.2.2]{Ngo_2010}, we make the following definition.
\begin{definition}\label{defn: asf 1}
    Let $S$ be a perfectoid space over $\F_q$, $D_S\in Div^1_{X}(S)$, $\L$ be a line bundle on $\mathbb D_S=\Spec B_{Div^1}(S)$, and $\gamma\in \g_\L(\mathbb D_T)$.
    Define the \emph{($B_{dR}$-)affine Springer fibre} $\Sp^\gamma$ as the v-sheaf $\Perf_{S}^{op}\to \mathrm{Set}$ given by sending (an affinoid perfectoid $T\to S$ for which the pullback $D_T$ of $D_S$ arises from an untilt of $T$ over $E$)\footnote{Recall that $\Div^1_X=\Spd E/\varphi^\Z$, where the quotient is taken in the category of sheaves on $\Perf_{\F_q}$ for the topology of open covers \cite[after definition II.1.19]{FS_geometrisation}.
    It follows that \{affinoid perfectoid $T\to S$ for which the pullback $D_T$ of $D_S$ arises from an untilt of $T$ over $E$\} contains every totally disconnected perfectoid space over $S$.} to the isomorphism classes of triples
    $$(\E,\phi,\iota),$$
    where 
    \begin{itemize}
        \item $\E$ is a $G$-bundle on $\mathbb D_T$,
        \item $\phi\in H^0(\mathbb D_T, \Ad(\E)_\L)$,
        \item $\iota:\E|_{\mathbb D_T^*}\isomto \E_0|_{\mathbb D_T^*}$ is an isomorphism whose induced map $\Ad(\E)_\L|_{\mathbb D_T^*}\isomto \Ad(\E_0)_\L|_{\mathbb D_T^*}$ sends $\phi|_{\mathbb D_T^*}$ to $\gamma$.
    \end{itemize}
\end{definition}

\begin{proposition}\label{Sp is a v-sheaf}
    $\Sp^\gamma$ is indeed a v-sheaf (on the basis \{affinoid perfectoid $T\to S$ for which the pullback $D_T$ of $D_S$ arises from an untilt of $T$ over $E$\} of $\Perf_S$, and hence extends uniquely to a v-sheaf on the entire $\Perf_S$).
\end{proposition}

\begin{proof}
    The lemma follows from the fact that the $B^+_{dR}$-affine Grassmannian $\Gr_G$ is a v-sheaf and the following:
    If $T'\to T$ is a v-cover of affinoid perfectoid spaces over $S$ for which $D_T$ arises from an untilt and $\mathcal V$ is a vector bundle on $\Spec B_{\Div^1_X}^+(T)$, then 
    \begin{equation*}
        H^0(\Spec B^+_{\Div^1_X}(T),\mathcal V) \to H^0(\Spec B^+_{\Div^1_X}(T'),\mathcal V) \rightrightarrows H^0(\Spec B^+_{\Div^1_X}(T''),\mathcal V)
    \end{equation*}
    is an equaliser diagram, where $T''=T'\times_T T'$.
    This is proved in a similar way to \cite[end of the proof of proposition 2.1]{Youcis_Cesnavicius}.
    Let us give more detail for their proof.
    Let $Q=\Spec\Sym \mathcal V^{\vee}$, so $\mathcal V=\Hom_{\Spec B_{\Div^1_X}^+(T)}(-,Q)$.
    We need to show
    \begin{equation*}
        \Hom_{B_{\Div^1_X}^+(T)}(\O_Q(Q),B^+_{\Div^1_X}(T))    \to  \Hom_{B_{\Div^1_X}^+(T)}(\O_Q(Q),B^+_{\Div^1_X}(T'))     \rightrightarrows \Hom_{B_{\Div^1_X}^+(T)}(\O_Q(Q),B^+_{\Div^1_X}(T''))    
    \end{equation*}
    is an equaliser diagram. By the definition of limit, it suffices to show 
    \begin{equation}
        B^+_{\Div^1_X}(T)   \to  B^+_{\Div^1_X}(T')     \rightrightarrows B^+_{\Div^1_X}(T'')   
    \end{equation}
    is an equaliser diagram in the category of $B^+_{\Div^1_X}(T)$-algebras.
    Fix an untilt $T^\sharp=\Spa(R,R^+)$ that gives rise to $D_T$. This induces untilts $T'^\sharp=\Spa(R',R'^+)$ and $T''^\sharp=\Spa(R'',R''^+)$ of $T'$ and $T''$.
    We need to see 
    \begin{equation*}
        B^+_{dR}(R)   \to  B^+_{dR}(R')     \rightrightarrows B^+_{dR}(R'')   
    \end{equation*}
    is an equaliser diagram.
    The kernel of $W_{\O_E}(R^{\circ,\flat})\twoheadrightarrow R^\circ$ is generated by a primitive element $\xi$ of degree 1. The images of $\xi$ in $B^+_{dR}(R')$ and $B^+_{dR}(R'')$ are also primitive elements of degree 1.
    By \Cref{exact_criterion} (with $g$ taken to be the difference of the two maps $B^+_{dR}(R')     \rightrightarrows B^+_{dR}(R'') $), it suffices to show 
    \[R\to R' \rightrightarrows R''\]
    is an equaliser diagram.
    This is true because $\Spa(R',R'^{+})\to \Spa(R,R^{+})$
    is a v-cover and $\O$ is a v-sheaf.
\end{proof}

\begin{lemma}\label{lem: injective}
    Let $R$ be a perfectoid $E$-algebra and $\mathcal V$ be a vector bundle on $\Spec B^+_{dR}(R)$.
    Then the restriction map $H^0(\Spec B^+_{dR}(R),\mathcal V) \to H^0(\Spec B_{dR}(R),\mathcal V)$ is injective.
\end{lemma}

\begin{proof}
    Let $W=\Spec(\Sym \mathcal V^\vee)$ be the total space of $\mathcal V$.
    We need to show that 
    \begin{equation*}
        \Hom_{B_{dR}^+(R)}(\O_W(W),B^+_{dR}(R))    \to  \Hom_{B_{dR}^+(R)}(\O_W(W),B_{dR}(R)) 
    \end{equation*}
    is injective.
    This is true because $B^+_{dR}(R) \to B_{dR}(R)$ is injective.
\end{proof}

\begin{proposition}\label{Sp is small}
    The natural map \[\Sp^\gamma\to \Gr_{G,S/\Div^1}\]
    is a closed immersion. Also, $\Sp^\gamma$ is a small v-sheaf.
\end{proposition}

\begin{proof}
    By \Cref{lem: injective}, the natural map 
\[\Sp^\gamma\to \Gr_{G,S/\Div^1}\]
is an injection of v-sheaves. 
Since $\Gr_{G,S/\Div^1}$ is a small v-sheaf, the same is true for $\Sp^\gamma$.

The proof that $\Sp^\gamma\to \Gr_{G,S/\Div^1}$ is a closed immersion is similar to that of \cite[Lemma 19.1.4]{Berkeley}. Let $Z=\Spa(A,A^+)$ be a strictly totally disconnected perfectoid space. Let $Z\to \Gr_{G,S/\Div^1}$ be a map, which corresponds to a pair $(\E,\iota:\E|_{\mathbb D_Z^*}\to G|_{\mathbb D_Z^*})\in \Gr_{G,S/\Div^1}(Z)$.
Let $\mathcal F=\Sp^\gamma\times_{\Gr_{G,S/\Div^1}}Z$, so we have a Cartesian diagram
\[\begin{tikzcd}
	{\mathcal F} & Z \\
	{\Sp^\gamma} & {\Gr_{G,S/\Div^1}.}
	\arrow[from=1-1, to=1-2]
	\arrow[from=1-1, to=2-1]
	\arrow[from=1-2, to=2-2]
	\arrow[from=2-1, to=2-2]
\end{tikzcd}\]
We need to show that $\mathcal F\to Z$ is representable by a closed immersion of perfectoid spaces.

Since $Z$ is strictly totally disconnected, every \'etale cover of it splits. By the same proof as \Cref{J-bundle etale locally trivial}, $\E$ and $\L$ are both trivial. Without loss of generality, $\E=G$ and $\L=\O$, so $(\E,\iota) = (G,\iota:G|_{\mathbb D_Z^*}\xrightarrow{g\cdot} G|_{\mathbb D_Z^*})$ for some $g\in G(\mathbb D_Z)$.
Let $T=\Spa(R,R^+)$ be an affinoid perfectoid space.
Then $$\mathcal F(T)=\set{f:T\to Z| f^*(\Ad_{g^{-1}}(\gamma)) \in Lie(G)\otimes_E B_{dR}^+(R^\sharp) }.$$
Upon choosing basis for $Lie(G)$, $Lie(G)\otimes_E B_{dR}^+(R^\sharp) \cong B_{dR}^+(R^\sharp)^m$ functorial in $R^\sharp$. Without loss of generality, $m=1$.
We have thus reduced to showing that for any $\delta\in B_{dR}(A^\sharp)$, the subfunctor
$$T\mapsto \set{f:T\to Z| f^*(\delta) \in B_{dR}^+(R^\sharp) }$$
of $Z$ is representable by a closed immersion.

We have $\delta \in \xi^{-i} B^+_{dR}(A^\sharp)$ for some $i\ge 1$. 
Recall that a closed subspace of a strictly totally disconnected perfectoid space is also strictly totally disconnected.
It suffices to show that the subfunctor
$$T\mapsto \set{f:T\to Z| f^*(\delta) \in \xi^{-i+1} B_{dR}^+(R^\sharp) }$$
of $Z$ is representable by a closed immersion $Z_i=\Spa(A_i,A_i^+)\subset Z$, because once we have shown this, we can replace $Z$ by $Z_i$, pullback $\delta$ to $B_{dR}(A_i^\sharp)$, and repeat the process with $i$ replaced by $i-1$ until we reach $i=1$.
Recall that $\xi^{-i}B_{dR}^+(R^\sharp)/\xi^{-i+1}B_{dR}^+(R^\sharp)\cong R^\sharp$.
We have finally reduced to showing that for $a\in A^\sharp$, the subfunctor
$$T=\Spa(R,R^+) \mapsto \set{f:T\to Z| f^*(a)=0 \in R^\sharp }$$
of $Z$ is representable by a closed immersion.
This can be shown as in the proof of \cite[Lemma 19.1.4]{Berkeley}.
\end{proof}

\begin{proposition}\label{sheafification description}
    Consider the presheaf $\Perf_{S}^{op}\to \mathrm{Set}$ given by sending (an affinoid perfectoid $T\to S$ for which the pullback $D_T$ of $D_S$ arises from an untilt of $T$ over $E$) to 
    \[\{g\in G(B_{\Div^1_X}(T))/G(B^+_{\Div^1_X}(T)): \Ad_{g^{-1}}(\gamma)\in \g_\L(B^+_{\Div^1_X}(T))\}. \]
    Its sheafification with respect to the analytic topology on $\Perf_S$ is $\Sp^\gamma$.
\end{proposition}

\begin{proof}
    Let us denote the presheaf in the statement as $\mathcal F$.
    We have a map of presheaves $\mathcal F \to \Sp^\gamma$ given by 
    \begin{align*}
        \mathcal F(T)&\to \Sp^\gamma(T)\\
        g&\mapsto (\E_0, \Ad_{g^{-1}}(\gamma), \E_0|_{\mathbb D_T^*} \xrightarrow{g\cdot}  \E_0|_{\mathbb D_T^*}  ).
    \end{align*}
    This map is injective, so its sheafification is also injective.
    The surjectivity of the sheafified map follows from \cite[Theorem 3.1]{Youcis_Cesnavicius}.
\end{proof}

\begin{remark}
    Suppose that $C$ is an algebraically closed perfectoid field. Then every line bundle on $\mathbb D_C$ is trivial by the same arguement as the proof of \Cref{J-bundle etale locally trivial}. By \Cref{sheafification description}, $$\Sp^\gamma(\Spa C) \cong \set{g\in G(C(\!(\xi )\!))/G(C\powerseries{\xi}): \Ad_{g^{-1}}(\gamma) \in \g(C\powerseries{\xi})},$$ as in the case for the classical affine Springer fiber. This expression also suggests that $\Sp^\gamma$ may be related to orbital integrals.
\end{remark}

\section{Product formula}
In this section, we shall show that modulo the actions of certain Picard stacks, there is a natural injective map of \'etale-stacks from the product of affine Springer fibers to the Hitchin fibre that induces an equivalence of categories on every geometric point. This is a geometric manifestation of the fact that global orbital integral factorises as the product of local orbital integrals. This also partly justifies the claim that affine Springer fibers are local analogues of Hitchin fibers.

Let $E$ be a non-Archimedean local field with residue field $\mathbb F_q$, where $q$ is a power of $p$. 
Let $G$ be a split connected reductive group over $E$, $B$ be a Borel subgroup over $E$, $T$ be a split maximal torus in $B$. 
Assume $char E\nmid |W|$, where $W$ is the Weyl group of $G$.
Fix an algebraically closed perfectoid field $C$ over $\mathbb F_q$. Let $\L$ be a line bundle on $X_C^{sch}$.
Fix a Kostant section $\Kos:\c \to \g$.

\begin{lemma}\label{equiv_G_bundles}\cite[Lemma A.25, Example A.28, Proposition A.30]{equiv-G-bundles}
    Let $S$ be a separated regular noetherian scheme of dimension $\le 1$ and $G\to S$ be an affine\footnote{We mean the morphism $G\to S$ is affine. The schemes $G$ and $S$ need not be affine.} flat group scheme of finite type. Then for any prestack $\mathcal X$ over $S$, the groupoids of cohomological $G$-bundles over $\mathcal X$ and of Tannakian $G$-bundles over $\mathcal X$ are equivalent.

    Here,
    \begin{itemize}
        \item A cohomological $G$-bundle on $\mathcal X$ is an fpqc-sheaf on $\mathrm{Aff}/\mathcal X$ endowed with a $G\times_S\mathcal X$-action over $\mathcal X$ which is locally for the fpqc-topology on $\mathcal X$ isomorphic to $G_\mathcal X$.
        \item A Tannakian $G$-bundle on $\mathcal X$ is an exact $\mathcal O_S$-linear symmetric monidal functor $\Rep_S(G) \to \mathrm{VB}(\mathcal X )$, where $\Rep_S(G)$ is the category of vector bundles $\mathcal V$ on $T$ with a morphism of group schemes $G\to \underline{Aut}_{\O_S}(\mathcal V)$ and $\mathrm{VB}(\mathcal X)$ is the category of vector bundles on $\mathcal X$.
    \end{itemize}
\end{lemma}

\begin{definition}
    Let $a\in \mathcal A_G(\Spa C)$. Let $v$ be a closed point of $X_C^{sch}$ and $a_v:= a|_{\Spec B_{\Div^1_X}^+(\Spa C)}$.
    \begin{enumerate}
        \item Define $\Gr_{G,a,v}$ to be $\Gr_{G,\Spa C/\Div^1}^{\Kos(a_v)}$. Here, the map $\Spa C\to \Div^1_X$ is the map corresponding to $v$, $a_v$ is the restriction of $a:X_C^{sch}\to \mathfrak c_\L$ to $\Spec B_{Div^1}^+(\Spa C)$.
        In other words, $\Gr_{G,a,v}$ is the v-sheaf $\Perf_{\Spa C}^{op}\to \Set$ given by sending (an affinoid perfectoid $T\to S$ for which the pullback $D_T$ of $D_S$ arises from an untilt of $T$ over $E$) to the isomorphism classes of triples
        $$(\E,\phi,\iota),$$
        where 
        \begin{itemize}
            \item $\E$ is a $G$-bundle on $\mathbb D_T$,
            \item $\phi\in H^0(\mathbb D_T, \Ad(\E)_\L)$,
            \item $\iota:\E|_{\mathbb D_T^*}\iso \E_0|_{\mathbb D_T^*}$ is an isomorphism whose induced map $\Ad(\E)_\L|_{\mathbb D_T^*}\isomto \Ad(\E_0)_\L|_{\mathbb D_T^*}$ sends $\phi|_{\mathbb D_T^*}$ to $\Kos(a_v)$.
        \end{itemize}
        \item Let $\Gr_{G,a,v}^{reg}$ be the sub-presheaf of $\Gr_{G,a,v}$ consisting of those $(\E,\phi,\iota)$ with $\phi\in H^0(\mathbb D_T, \E\times^G \g^{reg}_\L)$.
        \item Define $\mathcal P_{a,v}^{loc} =\Gr_{J_{\L}\times_{\Div^1,a}\Spa C}$ to be the presheaf over $\Perf_{\Spa C}$ sending $T=\Spa(R,R^+)$ to \\$$\set{J_{\L}\times_{\c_\L, a_v}\mathbb D_T \text{-torsors $Q$ on }\mathbb D_T \text{ with a trivialisation of }Q|_{\mathbb D_T^*}}/\cong.$$ We also define $J_{a_v}:= J_{\L}\times_{\c_\L, a_v}\mathbb D_T$.
    \end{enumerate}
\end{definition}
Recall that $\mathbb D_T$ denotes $\Spec B^+_{Div^1}(T)$ and $\mathbb D_T^*$ denotes $\Spec B_{Div^1}(T)$, and this depends on the closed point $v$.
In case of ambiguity, we will instead denote them by $$\mathbb D_{T,v}$$ and $$\mathbb D^*_{T,v}.$$

If we do not take isomorphism classes in the definition of $\mathcal P_{a,v}^{loc}$, we will get a prestack that is equivalent to $\mathcal P_{a,v}^{loc}$ by an argument similar to \Cref{lem: injective}.
We can by \Cref{equiv_G_bundles} identify $J_{a_v}$-bundles on $\mathbb D_T$ with exact $\O_{\mathfrak c_\L}$-linear symmetric monoidal functors $\Rep_{\mathfrak c_\L}(J_\L) \to \mathrm{VB}(\mathbb D_T)$. Since $T\mapsto \mathrm{VB}(\mathbb D_T)$ is a v-stack by \cite[Corollary 17.1.9]{Berkeley}, $\mathcal P_{a,v}^{loc}$ is a v-sheaf by \Cref{lem: exactness on B_dR}.

Let $\phi_0:=\Kos(a_v)$.

    Note that for all $(\E,\phi,\iota)\in \Gr_{G,v,a}(T)$, the image of $(\E,\phi)$ under the Hitchin fibration is $a_v$, because by \Cref{lem: injective}, the image is determined by its restriction to $\mathbb D_T^*$ and we have an isomorphism $\iota:\E|_{\mathbb D_T^*} \isomto \E_0|_{\mathbb D_T^*}$ which sends $\phi|_{\mathbb D_T^*}$ to $\Kos(a_v)|_{\mathbb D_T^*}$.

Similar to \Cref{sec:action of Picard}, we can define an action 
\begin{equation}\label{P^loc acts on Gr}
    \P_{a_v}^{loc} \times_{\Spa C} \Gr_{G,a,v} \to \Gr_{G,a,v}
\end{equation}
of $\P_{a,v}^{loc}$ on $\Gr_{G,a,v}$ over $\Spa C$.
Indeed, if $(Q,t:Q|_{\mathbb D_T^*} \isomto J_{a_v}|_{\mathbb D_T^*} ) \in \P_{a_v}^{loc}(T)$ and $(\E,\phi,\iota)\in \Gr_{G,a,v}(T)$, then as in \Cref{sec:action of Picard}, we get $Q*(\E,\phi)$. The trivialisations $t$ and $\iota$ induce a trivialisation $Q*(\E,\phi)|_{\mathbb D_T^*}\isomto J_{a_v}*(\E_0,\phi_0)|_{\mathbb D_T^*} \cong (\E_0,\phi_0).$

\begin{lemma}\label{J-bundle etale locally trivial}
    Let $Q$ be a $J_{a_v}$-bundle on $\mathbb D_T$, where $T=\Spa(R,R^+)$ is an affinoid perfectoid space over $C$. Then there exists an \'etale cover $T'\to T$ such that $Q|_{\mathbb D_{T'}}$ is a trivial $J_{a_v}|_{\mathbb D_{T'}}$-bundle.
\end{lemma}

\begin{proof}
    The usual proof for affine Grassmannian works here. Let us record it here to show that it works in this setup.

    We have a map $B_{dR}^+(R^\sharp)\to R^{\sharp}$, i.e. a map $f:\Spec R^{\sharp}\to \mathbb D_T$.
    The base change $f^*Q$ is a $J_{a_v}$-bundle on $\Spec R^{\sharp}$.
    As $J_{a_v}\to \mathbb D_T$ is smooth, by passing to an \'etale cover of $\Spa(R,R^+)$, we may assume $f^*Q$ is the trivial $J_{a_v}$-bundle.
    We claim that $Q$ is also trivial.

    Since $J_{a_v}$ is affine, $Q$ is representable by an affine scheme $\mathfrak Q$ over $\mathbb D_T$ by descent of affine morphisms.
    As $f^*Q$ is trivial, there is an element $s_1\in \Mor_{\mathbb D_T}(\Spec R^{\sharp},\mathfrak Q)$.
    As $\mathfrak Q\to \mathbb D_T$ is smooth (because smoothness can be checked fpqc locally), there is a morphism $s_2$ such that the following commutes:
    \[\begin{tikzcd}
        {\Spec R^\sharp} & {\mathfrak Q} \\
        {\Spec B^+_{dR}(R^\sharp)/\xi^2} & {\mathbb D_T}
        \arrow["{s_1}", from=1-1, to=1-2]
        \arrow[from=1-1, to=2-1]
        \arrow[from=1-2, to=2-2]
        \arrow["{s_2}"{description}, dotted, from=2-1, to=1-2]
        \arrow[from=2-1, to=2-2]
    \end{tikzcd}\]
    Similarly, there is a morphism $s_3$ such that the following commutes:
    \[\begin{tikzcd}
        {\Spec B^+_{dR}(R^\sharp)/\xi^2} & {\mathfrak Q} \\
        {\Spec B^+_{dR}(R^\sharp)/\xi^3} & {\mathbb D_T}
        \arrow["{s_2}", from=1-1, to=1-2]
        \arrow[from=1-1, to=2-1]
        \arrow[from=1-2, to=2-2]
        \arrow["{s_3}"{description}, dotted, from=2-1, to=1-2]
        \arrow[from=2-1, to=2-2]
    \end{tikzcd}\]
    We can continue this process and get compatible sequence of maps $s_1, s_2,\dots$
    i.e. compatible sequence of maps $H^0(\mathfrak Q,\O_{\mathfrak Q})\to B^+_{dR}(R^\sharp)/\xi^n$ over $B^+_{dR}(R^\sharp)$.
    This gives a map $H^0(\mathfrak Q,\O_{\mathfrak Q})\to \lim_n B^+_{dR}(R^\sharp)/\xi^n= B^+_{dR}(R^\sharp)$ over $B^+_{dR}(R^\sharp)$, i.e. 
    an element of $\Mor_{\mathbb D_T}(\mathbb D_T,\mathfrak Q)$, so $Q$ is trivial.
\end{proof}

\begin{lemma}\label{elts of Gr locally}
    Let $(\E,\phi,\iota)\in \Gr_{G,v,a}^{reg}(T)$. Then there exists an \'etale cover $T'\to T$ such that\footnote{Recall that elements in $\Gr_{G,v,a}^{reg}(T')$ are \textit{isomorphism classes}, so equality means equality of isomorphism classes, not of the representatives.}
     $$(\E,\phi,\iota) = (\E_0,\phi_0, \E_0|_{\mathbb D_{T'}^*}\xrightarrow{j\cdot}\E_0|_{\mathbb D_{T'}^*})$$ in $\Gr_{G,v,a}^{reg}(T')$ for some $j\in J_{a_v}(\mathbb D_{T'}^*)$. Moreover, $j$ is unique up to right multiplication by elements of $J_{a_v}(\mathbb D_{T'})$.
\end{lemma}

\begin{proof}
    Let $Q_{\E,\phi}:= \Isom((\E_0,\phi_0), (\E,\phi))$. 
    Let us show that $Q_{\E,\phi}$ is a $J_{a_v}$-bundle on $\mathbb D_T$.
    There is a right $J_{a_v}$-action on $Q_{\E,\phi}$ given by $f*j:=f(j\cdot\;)$.
    There is a fppf cover $U$ of $\mathbb D_T$ such that $\E|_U\cong\E_0|_U$.
    Let us fix such an isomorphism and identify $\E|_U\cong\E_0|_U$.
    Then for all $T\to U$, we can identify $Q_{\E,\phi}(T)$ with $$\{g\in G(T): \Ad_g(\phi_0)=\phi\}.$$
    We know $\phi|_U,\phi_0|_U$ have the same image under the map $\g^{reg}_\L(U) \to \c_\L(U)$.
    By regularity and \Cref{v are conjugate}, there is a fppf cover $V$ of $U$ and $g\in G(V)$ such that $\Ad_g(\phi_0|_V) = \phi|_V$.
    Then we have a $J_{a_v}$-equivariant isomorphism $J_{a_v}|_V\to Q_{\E,\phi}|_V$ given by $j\mapsto gj$.
    Thus, $Q_{\E,\phi}$ is a $J_{a_v}$-bundle on $\mathbb D_T$.

    By \Cref{J-bundle etale locally trivial}, there exists an \'etale cover $T'\to T$ such that $Q_{\E,\phi}|_{\mathbb D_T'}$ is a trivial $J_{a_v}$-bundle.
    In particular, it has a global section, i.e. $(\E_0,\phi_0)\cong (\E,\phi)$. The lemma is now clear. 
\end{proof}

\begin{proposition}\label{LJ/L+J}
    Let $LJ_{a_v}/L^+J_{a_v}:\Perf_{\Spa C}^{op}\to \Set$ be the v-sheafification of the functor sending an affinoid perfectoid $T$ to $J_{a_v}(\mathbb D_T)/J_{a_v}(\mathbb D_T^*)$.
    We have an isomorphism
    \begin{align*}
        LJ_{a_v}/L^+J_{a_v} &\isomto \P_{a,v}^{loc}\\
        j\in J_{a_v}(\mathbb D_T)&\mapsto (J_{a_v}|_{\mathbb D_T}, J_{a_v}|_{\mathbb D_{T}^*}\xrightarrow{j\cdot}J_{a_v}|_{\mathbb D_{T}^*}).
    \end{align*}
    We also have an isomorphism
    \begin{align}
        LJ_{a_v}/L^+J_{a_v} &\isomto \Gr_{G,a,v}^{reg}\label{LJ to Gr}\\
        j\in J_{a_v}(\mathbb D_T)&\mapsto (\E_0|_{\mathbb D_T},\phi_0|_{\mathbb D_T}, \E_0|_{\mathbb D_{T}^*}\xrightarrow{j\cdot}\E_0|_{\mathbb D_{T}^*}).\nonumber
    \end{align}
    In particular, $\Gr_{G,a,v}^{reg}$ is a v-sheaf.
\end{proposition}

\begin{proof}
    The first isomorphism follows from \Cref{J-bundle etale locally trivial}.
    The injectivity of \eqref{LJ to Gr} is clear. By \Cref{elts of Gr locally}, this map is \'etale locally surjective as a map of presheaves. Since $\Gr_{G,a,v}^{reg}$ is a sub-presheaf of $\Gr_{G,a,v}$, we know that $\Gr_{G,a,v}^{reg}$ is a separated presheaf, i.e. for all v-cover $T'\to T$, the map $\Gr_{G,a,v}^{reg}(T) \to \Gr_{G,a,v}^{reg}(T')$ is injective.
    Combining these information, we can deduce that \eqref{LJ to Gr} is surjective as a map of presheaves. Hence it is an isomorphism.
\end{proof}

\begin{remark}
    By \Cref{J-bundle etale locally trivial}, we also have $LJ_{a_v}/_{\et}\;L^+J_{a_v} \isomto \P_{a,v}^{loc}$, where $LJ_{a_v}/_{\et}\;L^+J_{a_v}$ is the \'etale sheafification of $T\mapsto J_{a_v}(\mathbb D_T)/J_{a_v}(\mathbb D_T^*)$. It follows that $LJ_{a_v}/L^+J_{a_v} = LJ_{a_v}/_{\et}\;L^+J_{a_v}$.
\end{remark}

It follows from \Cref{LJ/L+J} that $\Gr_{G,a,v}^{reg} \cong LJ_{a_v}/L^+J_{a_v} \cong \mathcal P_{a,v}^{loc}$. 
We can give a more direct description of this isomorphism.

\begin{proposition}\label{Gr^reg equiv to P^loc}
    We have an isomorphism 
        $$F:\Gr_{G,a,v}^{reg}\isomto \mathcal P_{a_v}^{loc}$$
    which on $T$-valued points (for affinoid $T\in\Perf_{\Spa C}$) are given by
    \begin{equation*}
        (\E,\phi,\iota) \mapsto (Q_{\E,\phi},\mathrm{triv})
    \end{equation*}
    where $Q_{\E,\phi}:= \Isom((\E_0,\phi_0), (\E,\phi))$ and $\mathrm{triv}$ is the trivialisation on $Q_{\E,\phi}|_{\mathbb D_T^*}$ given by 
    \[\Isom((\E_0,\phi_0), (\E,\phi))|_{\mathbb D_T^*} \xrightarrow{\iota_*} \Isom((\E_0,\phi_0), (\E_0,\phi_0))|_{\mathbb D_T^*}\cong J_{a_v}|_{\mathbb D_T^*}. \]
    Moreover, $F$ agrees with the isomorphism induced by \Cref{LJ/L+J}, i.e. the following commutes
    \[\begin{tikzcd}
        & {LJ_{a_v}/L^+J_{a_v}} \\
        {\Gr_{G,a,v}^{reg}} && {\P_{a,v}^{loc}}
        \arrow[from=1-2, to=2-1]
        \arrow[from=1-2, to=2-3]
        \arrow["F"', from=2-1, to=2-3]
    \end{tikzcd}\]
\end{proposition}


\begin{proof}
    Note that $Q_{\E,\phi}$ is indeed a $J_{a_v}$-bundle on $\mathbb D_T$ by the proof of \Cref{elts of Gr locally}.
    It is straightforward to verify the commutativity of the diagram. We deduce that $F$ is an isomorphism from the fact that the other two morphisms of the diagram are isomorphisms.
\end{proof}

We have three actions of $\mathcal P_{a_v}^{loc}$ on $\Gr_{G,a,v}^{reg}$ over $\Spa C$:
\begin{enumerate}
    \item the one in \eqref{P^loc acts on Gr} restricted\footnote{The action preserves the regularity condition, because the map $J\times_{\mathfrak c}\g \to I$ restricts to a map $J\times_{\mathfrak c}\g^{reg} \to I|_{\g^{reg}}$.} to $\Gr_{G,a,v}^{reg}$
    \item the one obtained by identifying $\Gr_{G,a,v}^{reg}\cong \mathcal P_{a_v}^{loc}$ and using the contracted product on $\mathcal P_{a_v}^{loc}$.
    \item the one obtained by identifying $\Gr_{G,a,v}^{reg}\cong LJ_{a_v}/L^+J_{a_v}$, $\mathcal P_{a_v}^{loc} \cong LJ_{a_v}/L^+J_{a_v}$ and using the multiplication map $LJ_{a_v}/L^+J_{a_v}$ given by $(j_1,j_2)\mapsto j_1j_2$.
\end{enumerate}
These three actions are equivalent.

\begin{lemma}\label{3 actions}
    The three actions above are equivalent over $\Spa C$.
    More precisely, the three $1$-morphisms $\mathcal P_{a_v}^{loc} \times_{\Spa C} \Gr_{G,a,v}^{reg} \to \Gr_{G,a,v}^{reg}$ are isomorphic in the $2$-category of v-stacks over $\Spa C$.
\end{lemma}

\begin{proof}
    It is straightforward to verify that the first two actions both agree with the third one, so they all agree.
\end{proof}

\begin{lemma}\label{LJ is small}
    Let $$LJ_{a_v}:\Perf_{\Spa C}^{op}\to \Set$$ be the functor sending an affinoid perfectoid $T$ (for which the pullback $D_T$ of $D_{\Spa C}$ arises from an untilt of $T$ over $E$) to $J_{a_v}(\mathbb D_T^*)$.
    Let $$L^+J_{a_v}:\Perf_{\Spa C}^{op}\to \Set$$ be the functor sending an affinoid perfectoid $T$ (for which the pullback $D_T$ of $D_{\Spa C}$ arises from an untilt of $T$ over $E$) to $J_{a_v}(\mathbb D_T)$.
    Then both $LJ_{a_v}$ and $L^+J_{a_v}$ define small v-sheaves on $\Perf_{\Spa C}$.
\end{lemma}

\begin{proof}
    That $LJ_{a_v}$ and $L^+J_{a_v}$ are both v-sheaves can be proved by the same argument as \Cref{Sp is a v-sheaf}. By the proof of \cite[Proposition III.1.3]{FS_geometrisation}, to show that $LJ_{a_v}$ and $L^+J_{a_v}$ are small, it is enough to show that if $T_i=\Spa(R_i,R_i^+)$, $i\in I$ is an $\omega_1$-cofiltered inverse system of affinoid perfectoid spaces with limit $T=\Spa(R,R^+)$, then $$\colim_i (LJ_{a_v})(T_i) \isomto (LJ_{a_v})(T),\;  \colim_i (L^+J_{a_v})(T_i) \isomto (L^+J_{a_v})(T).$$

    By \cite[Proposition 6.5]{EC_diamonds}, 
    \[\colim_i R_i\isomto R, \; \colim_i R_i^+\isomto R^+.\]
    By the $\omega_1$-cofilterness of $I$ and the construction of the Witt ring by Witt polynomials, 
    \[\colim_i W_{\O_E}(R_i^+) \isomto W_{\O_E}(R^+).\]
    Fix $i_0\in I$, $\varpi$ be a pseudouniformizer of $R_0$ and $\xi\in \ker(W_{\O_E}(R_0^+) \to R_0^{+\sharp})$ a primitive element of degree $1$.
    Then $(\xi)=\ker(W_{\O_E}(R^+) \to R^{+\sharp})$.
    Then 
    \begin{align*}
        B^+_{dR}(R^\sharp)&=\lim_{m\ge 1} W_{\O_E}(R^+)\left[\frac{1}{[\varpi]}\right]/(\xi^m)\\
        &= \lim_{m\ge 1}\colim_i \left(W_{\O_E}(R_i^+)\left[\frac{1}{[\varpi]}\right]/(\xi^m)\right)\\
        &= \colim_i \lim_{m\ge 1}\left(W_{\O_E}(R_i^+)\left[\frac{1}{[\varpi]}\right]/(\xi^m)\right)\\
        &= \colim_i B^+_{dR}(R_i^\sharp),
    \end{align*}
    where the second to last equality holds because in the category of rings, $\omega_1$-filtered colimits commute with sequential limits.
    Inverting $\xi$ shows that we also have $B_{dR}(R^\sharp)=\colim_i B_{dR}(R_i^\sharp).$

    Note that $J_\L \times_{\c_\L} \mathbb D_{\Spa C}$ is an affine scheme of finite presentation over $\mathbb D_{\Spa C}$.
    Let $A$ be the ring of global section of $J_\L \times_{\c_\L} \mathbb D_{\Spa C}$. Then 
    \begin{align*}
        (L^+J_{a_v})(T) &= J_{a_v}(\mathbb D_T) \\
        &= \Hom_{D_{\Spa C}}(\mathbb D_T,J_\L \times_{\c_\L} \mathbb D_{\Spa C})\\
        &= \Hom_{B^+_{dR}(C^\sharp)}(A, B^+_{dR}(R^\sharp) )\\
        &= \colim_i \Hom_{B^+_{dR}(C^\sharp)}(A, B^+_{dR}(R_i^\sharp) )\\
        &= \colim_i (L^+J_{a_v})(T_i)
    \end{align*}
    where the second to last equality holds because $A$ is of finite presentation over $B^+_{dR}(C^\sharp)$.
    Similarly, $(LJ_{a_v})(T) = \colim_i (LJ_{a_v})(T_i)$.
\end{proof}

\begin{lemma}\label{lem: v sheaf}
    If $\mathcal F$ is a small v-sheaf with a group action of a small group v-sheaf\footnote{I.e. a group object in the category of v-sheaves.} $\mathcal G$, then $[\mathcal F/\mathcal G]$ is a small v-stack.
\end{lemma}

\begin{proof}
    As $\mathcal F$ is small, there is a perfectoid space $X$ with a v-surjection $\alpha:X\twoheadrightarrow \mathcal F$.
    We have the following commutative diagram, where the two squares are Cartesian.
\[\begin{tikzcd}
	&& X \\
	{\mathcal G\times X} & {\mathcal G\times \mathcal F} & {\mathcal F} \\
	X & {\mathcal F} & {[\mathcal F/\mathcal G]}
	\arrow["\alpha", from=1-3, to=2-3]
	\arrow["{\id\times\alpha}", from=2-1, to=2-2]
	\arrow["{pr_2}"', from=2-1, to=3-1]
	\arrow["act", from=2-2, to=2-3]
	\arrow["{pr_2}"', from=2-2, to=3-2]
	\arrow[from=2-3, to=3-3]
	\arrow["\alpha"', from=3-1, to=3-2]
	\arrow[from=3-2, to=3-3]
\end{tikzcd}\]
It follows that $X\times_{[\mathcal F/\mathcal G]} X = (\mathcal G \times X)\times_{\mathcal F}X$.
We would like to show that this is small.

We have an injection of v-sheaves $(\mathcal G \times X)\times_{\mathcal F}X\hookrightarrow \mathcal G\times X\times X.$
As $\mathcal G\times X\times X$ is small, there is a perfectoid space $Y$ with a v-surjection $Y\twoheadrightarrow \mathcal G\times X\times X.$
We form the fibre product diagram
\[\begin{tikzcd}
	{\mathcal H} & Y \\
	{(\mathcal G \times X)\times_{\mathcal F}X} & {\mathcal G\times X\times X.}
	\arrow[hook, from=1-1, to=1-2]
	\arrow[two heads, from=1-1, to=2-1]
	\arrow[two heads, from=1-2, to=2-2]
	\arrow[hook, from=2-1, to=2-2]
\end{tikzcd}\]
As $\mathcal H$ is a sub-v-sheaf of $Y$ and $Y$ is a diamond, $\mathcal H$ is also a diamond by \cite[Proposition 11.10]{Berkeley}. In particular, $\mathcal H$ admits a v-surjection from a perfectoid space. Thus, the same holds for $(\mathcal G \times X)\times_{\mathcal F}X$.
\end{proof}

\begin{corollary}
    The v-stack $[\Gr_{G,a,v}/\mathcal P_{a,v}^{loc}]$ is small.
\end{corollary}

\begin{proof}
    By \Cref{Sp is small} and \Cref{LJ is small}, $\Gr_{G,a,v}$ and $\mathcal P_{a,v}^{loc}=LJ_{a_v}/L^+J_{a_v}$ are both small v-sheaves. The result now follows from \Cref{lem: v sheaf}.
\end{proof}

Note that each root $\alpha:T\to \mathbb G_m$ of $G$ induces a map $d\alpha:\mathfrak t\to \mathbb A^1_E$ on Lie algebras.
Their product gives a map $\prod_{\alpha\in \Phi} d\alpha:\mathfrak t\to \mathbb A^1_E$, which is invariant under the Weyl group $W$, where $\Phi$ is the set of roots of $G$.
Hence, it induces a morphism
$$disc:=\prod_{\alpha\in \Phi} d\alpha:\mathfrak c\to \mathbb A^1_E.$$
Define $\mathfrak c^{rs}$ to be the affine scheme given by the locus where $disc$ does not vanish.
Let $$\mathfrak c^{rs}_\L=\tilde \L \times^{\mathbb G_m} \mathfrak c^{rs}_{X_C^{sch}}.$$
The preimage of $\mathfrak c^{rs}$ in $\g\to \c$ is contained in $\g^{reg}$.
Thus, the preimage of $\mathfrak c^{rs}_{\L}$ in $\g_\L\to \c_\L$ is contained in $\g^{reg}_\L$.

For $a\in \mathcal A_G(\Spa C)$, we define $$\mathrm{Higgs}_a := \Higgs \times_{\mathcal A_G} \Spa C,$$ where the map $\Spa C \to \mathcal A_G$ is the one corresponding to $a$. We call it the \emph{Hitchin fibre}.
We also define $$\mathcal P_a := \mathcal P\times_{\mathcal A_G}\Spa C.$$

In analogy with \cite[Definition 4.4]{Ngo_2006}, we make the following definition.
\begin{definition}
    We say that $a\in \mathcal A_G(\Spa C)$ is \emph{generically regular semisimple} if $a$, as a map $X_C^{sch}\to \mathfrak c_\L$, sends the generic point $\eta$ of $X_C^{sch}$ to $\mathfrak c^{rs}_\L$. 
\end{definition}

\begin{lemma}\label{J|rs is torus}
    The group scheme $J_\L|_{\c_\L^{rs}}\to \c_\L^{rs}$ is a torus.
\end{lemma}

\begin{proof}
    By \cite[Proposition 2.4.2]{Ngo_2010} and \cite[Proposition 2.5.1]{Chen_Zhu}, $J|_{\c^{rs}}\to \c^{rs}$ is a torus.
    Thus, the twist $J_\L|_{\c_\L^{rs}}\to \c_\L^{rs}$ is also a torus.
\end{proof}

\begin{lemma}\label{lem:torsor on punctured disc}
    Suppose that $a\in \mathcal A_G(\Spa C)$ is generically regular semisimple.
    \begin{enumerate}
        \item For every closed point $v\in |X_C^{sch}|$, $J_{a_v}|_{\mathbb D_{C,v}^*} =\mathbb D_{C,v}^*\times_{\c_\L}J_\L$ is a torus over $\mathbb D_{C,v}^*$.
        \item Let $T=\Spa(K,K^+)\in \Perf_{\Spa C}$, where $K$ is algebraically closed. Then every $J_{a_v}$-torsor on $\mathbb D_{T,v}^*$ is trivial.
    \end{enumerate}
\end{lemma}

\begin{proof}
    The element $a\in \mathcal A_G(\Spa C)$ corresponds to a map $f:X_C^{sch}\to \mathfrak c_\L$.
    Since $a$ is generically regular semisimple, $f(X_C^{sch})\cap \mathfrak c_\L^{rs} \neq \varnothing.$
    As $\mathfrak c_\L^{rs}$ is open in $\mathfrak c_\L$, $f^{-1}(\mathfrak c_\L^{rs})$ is a non-empty open subset in $X_C^{sch}$. Since $X_C^{sch}$ is Noetherian, $1$-dimensional, and irreducible, the complement of $f^{-1}(\mathfrak c_\L^{rs})$ in $X_C^{sch}$ is a finite set of closed points.

    For every closed point $v\in |X_C^{sch}|$,
    the map 
    \begin{equation}\label{D* to X}
        \mathbb D_{C,v}^*\to X_C^{sch}
    \end{equation}
    factors through the open subscheme $f^{-1}(\c_\L^{rs})\subset X_C^{sch}.$
    To see this, recall that $\mathbb D_{C,v} = \Spec \hat{\O_{X_C^{sch},v}}$.
    The map
    \begin{equation}\label{D to X}
        \mathbb D_{C,v} \to X_C^{sch}
    \end{equation}
    is given by $\Spec \hat{\O_{X_C^{sch},v}} \to \Spec \O_{X_C^{sch},v} \to X.$
    Pick an open neighborhood $U=\Spec B$ of $v\in X_C^{sch}$ such that $U\setminus\set{v} \subset f^{-1}(\c_\L^{rs})$ and an element $\xi\in B$ such that $\set{v}$ is the vanishing locus of $\xi$.
    Then \eqref{D to X} factors through $U$. The map $\mathbb D_{C,v} \to U$ corresponds to the ring homomorphism $B \to \hat{\O_{X_C^{sch},v}}$.
    Inverting $\xi$ gives $B[\frac{1}{\xi}] \to \hat{\O_{X_C^{sch},v}}[\frac{1}{\xi}]$.
    Taking $\Spec$ gives $\mathbb D^*_{C,v} \to U\setminus \set{v}.$
    Thus, \eqref{D* to X} factors as $\mathbb D_{C,v}^* \to U\setminus \set{v} \hookrightarrow  f^{-1}(\mathfrak c_\L^{rs}) \to X_C^{sch}$.
    
    It follows from this and \Cref{J|rs is torus} that $J_{a_v}$ is a torus over $\mathbb D_{C,v}$.
    Recall that $$\mathbb D_{T,v}^*=\Spec B_{dR}(K^\sharp)\cong \Spec K^\sharp((\xi)).$$ By \cite[Section 2 Theorem 2]{C1-field}, $K^\sharp((T))$ is a $C_1$ field.
    Since $K^\sharp((\xi))$ is a $C_1$-field and $J_{a_v}|_{\mathbb D_{T,v}^*}$ is a connected reductive group over $\Spec K^\sharp((T))$, $$H^1_{\et}(\mathbb D_{T,v}^*,J_{a_v}) = H^1(K^\sharp((\xi)),J_{a_v})=0$$ \footnote{The latter cohomology is Galois cohomology.}by \cite[Section 8.6]{Borel-Springer}. This is the desired result.
\end{proof}

\begin{remark}\label{abelianisation for Gr}
    Suppose that $a\in \mathcal A_G(\Spa C)$ is generically regular semisimple.
    The element $a$ corresponds to a map $f:X_C^{sch} \to \mathfrak c_\L$.
    For each closed point $v\in |X_C^{sch}|$ that lies in $f^{-1}(\c_\L^{rs})$, $J_{a_v} =\mathbb D_{C,v}\times_{\c_\L}J_\L$ is a torus over $\mathbb D_{C,v}$ by \Cref{J|rs is torus} and the proof of \Cref{lem:torsor on punctured disc}.
    Also, $\Gr_{G,a,v}^{reg} = \Gr_{G,a,v}$, because if $(\E,\phi,\iota)\in \Gr_{G,a,v}(T)$, then the image of $(\E,\phi)$ under the Hitchin fibration is $a_v$, which lies in $\mathfrak c_\L^{rs}(\mathbb D_{T,v})$, implying that $\phi\in H^0(\mathbb D_T,\E\times^G \g^{reg}_\L)$.
    Thus, by \Cref{Gr^reg equiv to P^loc}, 
    \begin{equation}\label{Gr isom to P}
        \Gr_{G,a,v} \cong \mathcal P_{a_v}^{loc}.
    \end{equation}
    In other words, in this case (which is the case for all but finitely many $v$), the affine Springer fiber is the affine Grassmannian of a torus.
    We can view this as an instance of abelianisation.
    Let us also note that \eqref{Gr isom to P} respects the $P_{a_v}^{loc}$-actions on both sides by \Cref{3 actions}. In particular, the quotient $[\Gr_{G,a,v}/\mathcal P_{a,v}^{loc}]$ is trivial.
\end{remark}

\begin{lemma}\label{elt in ring vanishes}
    Let $T$ be a perfectoid space and $t\in H^0(T,\O_T)$. Suppose that $f^*(t)=0$ for every algebraically closed perfectoid field $K$ and every map $f:\Spa(K,K^\circ)\to T$.
    Then $t=0$.
\end{lemma}

\begin{proof}
    Without loss of generality, we assume that $T=\Spa(A,A^+)$ is affinoid perfectoid.
    Let $x\in T$.
    Let $K(x)$ be the completed residue field at $x$. As $T$ is analytic, the absolute value on $K(x)$ is non-trivial.
    Since $K(x)$ is a complete valued field, its absolute value extends uniquely to its algebraic closure $\overline{K(x)}$ and then to the completion $\hat{\overline{K(x)}}=:C(x).$ By standard argument, $C(x)$ is algebraically closed. It follows that $C(x)$ is perfectoid.
    Since $A$ is uniform, the map
    \[A\to \prod_{x\in T} C(x)\]
    is injective by a result of Berkovich \cite[Theorem 5.2.1]{Berkeley}.
    In other words, the map
    \[H^0(T,\O_T)\to \prod_{x\in T} H^0(\Spa(C(x),C(x)^\circ),\O)\]
    is injective.
    The desired result follows.
\end{proof}

\begin{lemma}\label{elt in vb vanishes}
    Let $S$ be a perfectoid space over $\mathbb F_q$, $\E$ be a vector bundle on $X_S$.
    \begin{enumerate}
        \item\label{section vanishes} Let $t\in H^0(X_S,\E)$. Suppose that for every algebraically closed perfectoid field $K$ with open and integrally closed subring $K^+\subset K^\circ$ and every map $\Spa(K,K^+)\to S$, we have $t|_{X_{\Spa(K,K^+)}}=0$. Then $t=0$.
        \item\label{mor is identity} Let $\alpha:\E\to \E$ be a morphism. Suppose that for every algebraically closed perfectoid field $K$ with open and integrally closed subring $K^+\subset K^\circ$ and every map $\Spa(K,K^+)\to S$, we have $\alpha|_{X_{\Spa(K,K^+)}}=\id$. Then $\alpha=\id$.
    \end{enumerate}
\end{lemma}
By \Cref{thm:GAGA}, the same also holds with $X_S$ replaced by $X_S^{sch}$.

\begin{proof}
    We first prove (\ref{section vanishes}).
    Without loss of generality, assume that $S$ is affinoid perfectoid.
    Since $S$ is affinoid, there exist integers $n,m\ge 1$ such that $\O_{X_S}^m\twoheadrightarrow \E^\vee(n)$ by \cite[Theorem II.2.6]{FS_geometrisation}. Thus, $\E \injects \O_{X_S}^m(n)$. Hence, we may without loss of generality assume $\E=\O_{X_S}(n)$. Let $S^\sharp$ be an untilt of $S$ over $E$. By \cite[Proposition II.2.3]{FS_geometrisation}, for every integer $k$, there is a short exact sequence 
    \[0\to \O_{X_S}(k-1)\to \O_{X_S}(k) \to \O_{S^\sharp} \to 0.\]
    Taking global section gives us an exact sequence
    \[0 \to H^0(X_S,\O_{X_S}(k-1)) \xrightarrow{f_k} H^0(X_S,\O_{X_S}(k)) \xrightarrow{g_k} H^0(S^\sharp,\O_{S^\sharp})\]
    which is compatible with change of $S$.
    By \Cref{elt in ring vanishes} and the assumption on $t$, $g_n(t)=0$.
    Thus, $t$ lies in $H^0(X_S,\O_{X_S}(n-1))$.
    Repeating this process, we reduce to the case where $t\in H^0(X_S,\O_{X_S})$.

    By \cite[Proposition II.2.5(ii)]{FS_geometrisation}, $\mathcal H^0(\O)=\underline{E}$.
    Thus, $t$ corresponds to a continuous map $T:|S|\to E$ whose composite with every map $|\Spa(K,K^+)|\to |S|$ is the zero map. Since the images of the maps $|\Spa(K,K^+)|\to |S|$ cover $|S|$ as $(K,K^+)$ varies, $T$ must be the zero map.
    Therefore $t=0$.

    For (\ref{mor is identity}), note that $\alpha$ can be identified with a global section of the vector bundle $\underline{\Hom}_{\O}(\E,\E)$, so the result follows from (\ref{section vanishes}).
\end{proof}

We have the following analogue of \cite[Lemme 4.5]{Ngo_2006}.
\begin{lemma}
    Suppose $a\in \mathcal A_G(\Spa C)$ is generically regular semisimple. Let $S$ be an affinoid perfectoid space over $\Spa C$.
    Then the $2$-quotient $[\mathrm{Higgs}_a(S)/\mathcal P_a(S)]$ is equivalent to a $1$-groupoid.
\end{lemma}

\begin{proof}
    By \Cref{2-quot equiv to 1-cat}, we need to show that for all $(\E,\phi)\in \mathrm{Higgs}_a(S)$, the natural map 
    \begin{equation*}\label{eq: act}
        F:\Aut(1_{J_a})\to \Aut((\E,\phi))
    \end{equation*}
    is injective, where $1_{J_a}$ is the trivial $J_a$-torsor on $X_S^{sch}$.

    We first prove it for $S=\Spa C$. We employ the same strategy as \cite[Lemme 4.5]{Ngo_2006} for this case. We shall focus on the omitted detail there.
    We have a commutative diagram
    \[\begin{tikzcd}
        {\Aut(1_{J_a})} & {\Aut((\E,\phi))} \\
        {\Aut(1_{J_a}|_\eta)} & {\Aut((\E,\phi)|_\eta)}
        \arrow[from=1-1, to=1-2]
        \arrow[from=1-1, to=2-1]
        \arrow[from=1-2, to=2-2]
        \arrow[from=2-1, to=2-2].
    \end{tikzcd}\]
    Note that $\Aut(1_{J_a})=H^0(X_S^{sch},J_a)$ and $\Aut(1_{J_a}|_\eta)=H^0(\eta,J_a|_\eta)$. By \Cref{lem:J isom} and the assumption that $a$ is generically regular semisimple, the bottom map is an isomorphism. The desired result will follow if we show that the left vertical map is injective.
    Since $I_\L\to \g_\L$ is an affine morphism, the same is true for its base change $\pi:J_a \to X_C^{sch}$. 
    We would like to show that the restriction map
    \begin{equation*}
        \Mor_{X_C^{sch}}(X_C^{sch},J_a) \to \Mor_{X_C^{sch}}(\eta,J_a)
    \end{equation*}
    is injective.
    Let $f,g\in \Mor_{X_C^{sch}}(X_C^{sch},J_a)$ with $f|_\eta=g|_\eta$.
    Let $U=\Spec B$ be a non-empty open subscheme of $X_C^{sch}$.
    We can view the restrictions of $f,g$ to $U$ as maps 
    $f|_U,g|_U:U\to \pi^{-1}(U)$. We know these two maps agree after restriction along $\Spec \Frac(B) \to \Spec B.$
    As $U$ and $f^{-1}(U)$ are both affine schemes, we know $f|_U=g|_U$.
    As $U$ is arbitrary, $f=g$.

    Let us return to the case of general affinoid perfectoid $S$ over $\Spa C$.
    Let $f,g\in \Aut(1_{J_a})$ with $F(f)=F(g)$. Then $F(f)|_{X_K^{sch}}=F(g)|_{X_K^{sch}}$ for all $\Spa(K,K^+)\to S$ with $K$ algebraically closed perfectoid field $K$ and open integrally closed subring $K^+\subset K^\circ.$ The same proof as above proves that $f|_{X_K^{sch}}=g|_{X_K^{sch}}$.

    By \Cref{equiv_G_bundles}, we can identify $J_a$-bundles on $X_S$ with exact $\O_{\mathfrak c_\L}$-linear symmetric monoidal functors $\Rep_{\mathfrak c_\L}(J_\L) \to \mathrm{VB}(X_S^{sch})$. Then the result follows from part (\ref{mor is identity}) of \Cref{elt in vb vanishes}.
\end{proof}

Due to this lemma, we will henceforth regard $[\mathrm{Higgs}_a(S)/\mathcal P_a(S)]$ as a 1-groupoid.

\begin{definition}
    Define $[\mathrm{Higgs}_a/\mathcal P_a]_\et$ to be the \'etale-stackification of $T\mapsto [\mathrm{Higgs}_a(T)/\mathcal P_a(T)]$.
    Define $[\Gr_{G,a,v}/\mathcal P_{a,v}^{loc}]_\et$ to be the \'etale-stackification of $S\mapsto [\Gr_{G,a,v}(T)/\mathcal P_{a,v}^{loc}(T)].$
\end{definition}

\begin{remark}\label{rmk: morphisms}
    Let $T\to \Spa C$ be an affinoid perfectoid.
    Using the fact that $\mathcal P_a$ and $\mathrm{Higgs}_a$ are both v-stacks, one can easily verify that for any $x,y \in \mathrm{Higgs}_a(T)$, the presheaf of isomorphisms from $x$ to $y$ in $[\mathrm{Higgs}_a(T)/\mathcal P_a(T)]$ is a v-sheaf. It follows from \cite[Lemma 02ZN]{stacks-project} that this is also the sheaf of isomorphisms from $x$ to $y$ in $[\mathrm{Higgs}_a/\mathcal P_a](T)$. In particular, the isomorphisms from $x$ to $y$ in $[\mathrm{Higgs}_a(T)/\mathcal P_a(T)]$ and in $[\mathrm{Higgs}_a/\mathcal P_a]_{\et}(T)$ agree.
\end{remark}

\begin{remark}
    We take \'etale stackification to ensure that for every geometric point\footnote{A geometric point is an adic space of the form $\Spa(K,K^+)$, where $K$ is an algebraically closed perfectoid field and $K^+$ is an open and bounded valuation subring of $K$.} $\Spa(K,K^+)$, the value at $\Spa(K,K^+)$ of a presheaf $\mathcal F$ agrees with the value at $\Spa(K,K^+)$ of the stackification of $\Spa(K,K^+)$. This is used in the proof below.
\end{remark}

We have the following analogue of \cite[Theoreme 4.6]{Ngo_2006}, which roughly says that modulo the actions of the Picard stacks, the Hitchin fiber factorises as a product of affine Springer fibers.
\begin{theorem}\label{product-formula}
    Suppose that $a\in \mathcal A_G(\Spa C)$ is generically regular semisimple. Then there is a canonical morphism
    \begin{equation*}
        F:\prod_{\text{closed }v\in |X_C^{sch}|} [\Gr_{G,a,v}/\mathcal P_{a,v}^{loc}]_\et \to [\mathrm{Higgs}_a/\mathcal P_a]_\et
    \end{equation*}
    of \'etale-stacks over $\Spa C$
    and all but finitely many terms on the left are trivial (i.e. isomorphic to $\Spa C$).
    This is injective, i.e. the map on the $\Spa(R,R^+)$-valued points is fully faithful for every $\Spa(R,R^+)\in \Perf_{\Spa C}.$
    Moreover, $F$ is an equivalence on the $\Spa(K,K^+)$-valued points for every geometric point $\Spa(K,K^+)\in \Perf_{\Spa C}.$
\end{theorem}

\begin{proof}
    The element $a\in \mathcal A_G(\Spa C)$ corresponds to a map $f:X_C^{sch}\to \mathfrak c_\L$.
    We have already seen in the proof of \Cref{lem:torsor on punctured disc} that $U:=f^{-1}(\mathfrak c_\L^{rs})$ is a non-empty open subset in $X_C^{sch}$ whose complement in $X_C^{sch}$ is a finite set of closed points.
    For each closed point $v\in |X_C^{sch}|$ that lies in $U$, the quotient $[\Gr_{G,a,v}/\mathcal P_{a,v}^{loc}]$ is trivial by \Cref{abelianisation for Gr}.

    Using \Cref{Higgs^reg is a gerbe,Gr^reg equiv to P^loc} and \Cref{lem:torsor on punctured disc}, we can prove the remaining assertion in the same way as \cite[Theoreme 4.6]{Ngo_2006}. Let us sketch the argument.

    Let $T\to \Spa C$ be an affinoid perfectoid (for which the pullback $D_T$ of $D_{\Spa C}$ arises from an untilt of $T$ over $E$).
    Let $(\E_v,\phi_v,\iota_v)_v \in \prod_{\text{closed }v\in |X_C^{sch}|} \Gr_{G,a,v}(T)$.
    Let $U'$ be the pullback of $U$ along $X_T^{sch} \to X_C^{sch}$.
    By the Beauville-Laszlo lemma, we can glue $(\E_v)_{v\in |X_C^{sch}|\setminus |U|}$ with the trivial $G$-bundle on $U'$ via the trivialisations $(\iota_v)$ to get a $G$-bundle $\E$ on $X_T^{sch}$.
    Note that global sections of $\Ad(\E)_\L$ correspond to morphisms $\O \to \Ad(\E)_\L$, so by the fully faithfulness part of the Beauville-Laszlo lemma, we can glue $(\phi_v)_{v\in |X_C^{sch}|\setminus |U|}$ with $\phi_0|_{U'} \in H^0(U',\Ad(\E)_\L)$ to get a section $\phi \in H^0(X_T^{sch},\Ad(\E)_\L)$.
    Then $(\E,\phi)\in \mathrm{Higgs}_a(T)$.
    This defines a map $\prod_{\text{closed }v\in |X_C^{sch}|} \Gr_{G,a,v} \to \mathrm{Higgs}_a$, which is $\mathcal P_{a,v}^{loc}$ and $\mathcal P_a$ equivariant, so it induces a map \begin{equation*}
        F_0:\prod_{v\in |X_C^{sch}|\setminus |U|} [\Gr_{G,a,v}/_{pre}\;\mathcal P_{a,v}^{loc}] \to [\mathrm{Higgs}_a/_{pre}\;\mathcal P_a].
    \end{equation*}
    Taking \'etale stackification gives
    \begin{equation*}
        F:\prod_{v\in |X_C^{sch}|\setminus |U|} [\Gr_{G,a,v}/\mathcal P_{a,v}^{loc}]_\et \to [\mathrm{Higgs}_a/\mathcal P_a]_\et.
    \end{equation*}

    Since $F$ is the stackification of $F_0$ and $[\mathrm{Higgs}_a(T)/\mathcal P_a(T)] \to [\mathrm{Higgs}_a/\mathcal P_a]_\et(T)$ is fully faithful for every affinoid $T\in\Perf_{\Spa C}$ by \Cref{rmk: morphisms},
    to show that $F$ is injective, it is enough to prove that $F_0(T)$ is fully faithful for every $T$ in a basis of $\Perf_C$.
    Let $T\to \Spa C$ be an affinoid perfectoid (for which the pullback $D_T$ of $D_{\Spa C}$ arises from an untilt of $T$ over $E$).
    Let $m=(m_v),n=(n_v)  \in \prod_{v\in |X_C^{sch}|\setminus |U|} \Gr_{G,a,v}(T)$.
    Then a morphism from $F(m)$ to $F(n)$ in $[\mathrm{Higgs}_a(T)/\mathcal P_a(T)]$ is given by a pair $(P,\theta)$, where $P\in \mathcal P_a(T)$ and $\theta: P*F(m)\isomto F(n)$ is an isomorphism.
    We know that $F(m)|_{U'}=(\E_0|_U,\phi_0|_U)$ and $F(n)|_{U'}=(\E_0|_U,\phi_0|_U)$. 
    By \Cref{Higgs^reg is a gerbe} and \Cref{2 actions on Higgs}, we necessarily have $P|_{U'} \cong J_{a}|_{U'}$.
    We also know that for each $v\in |X_C^{sch}|\setminus |U|$, $F(m)|_{\mathbb D_{T,v}} = m_v$ and $F(n)|_{\mathbb D_{T,v}} = n_v$.
    By the Beauville-Laszlo lemma, the data of such a pair $(P,\theta)$ is thus the same as $(P_v)\in \prod_{v\in |X_C^{sch}|\setminus |U|} \mathcal P_{a,v}^{loc}(T)$. 

    Finally, let $T=\Spa(K,K^+)\in \Perf_{\Spa C}$ be a geometric point.
    To show that 
    \begin{equation*}
        F(T):\prod_{v\in |X_C^{sch}|\setminus |U|} [\Gr_{G,a,v}(T)/\mathcal P_{a,v}^{loc}(T)] \to [\mathrm{Higgs}_a(T)/\mathcal P_a(T)]
    \end{equation*}
    is an equivalence of categories, it remains to show that this is essentially surjective.
    Let $(\E,\phi) \in \mathrm{Higgs}_a(T)$.
    By \Cref{Higgs^reg is a gerbe} and \Cref{2 actions on Higgs}, there is a $J_a$-torsor $P'$ on $U'$ and an isomorphism $P'*(\E,\phi)|_{U'} \xrightarrow{\iota} (\E_0,\phi_0)|_{U'}$.
    By \Cref{lem:torsor on punctured disc}, for every $v\in |X_C^{sch}|\setminus |U|$,  $P'|_{\mathbb D_{T,v}^*}$ is a trivial $J_{a_v}$-torsor on $\mathbb D_{T,v}^*$.
    In particular, it can be extended to a $J_{a_v}$-torsor on $\mathbb D_{T,v}$.
    By the Beauville-Laszlo lemma, we can thus extend $P'$ to a $J_a$-torsor on $X_T^{sch}$, which we still denote as $P'$.
    By replacing $(\E,\phi)$ by $P'*(\E,\phi)$, we can assume we have an isomorphism $(\E,\phi)|_{U'} \xrightarrow{\iota} (\E_0,\phi_0)|_{U'}$.
    Then $((\E,\phi)|_{\mathbb D_{T,v}}, \iota|_{{\mathbb D^*_{T,v}}})_v \in \prod_{v\in |X_C^{sch}|\setminus |U|} \Gr_{G,a,v}(T)$ and its image under $F$ is isomorphic to $(\E,\phi)$.
\end{proof}

\section{Connections with number-theoretic objects}
\subsection{Adelic description of $\Higgs(C)$}
Fix $C$ an algebraically closed perfectoid field.
For simplicity, we take $\L=\O_{X_C^{sch}}$ (in the definition of Higgs bundles).
Throughout the section, we let $\eta$ be the generic point of $X^{sch}_C=X_{\Spa C}^{sch}$ and $$F=\mathcal O_{X_{C}^{sch},\eta}$$ be the function field of $X_{C}^{sch}$.
If $E=\Q_p$, we have $F=\Frac(B_{crys}^{\phi=1})$, where $B_{crys}$ is the crystalline period ring \cite[Theorem 13.5.3]{Berkeley}.
For each closed point $v$ of $X_{C}^{sch}$, we let $\O_v:=\hat{\O_{X_C^{sch},v}} = B^+_{\Div^1_X}(\Spa C)$ and $\mathbb D_v:=\Spec \O_v$ ($=\mathbb D_{C,v}$ in previous sections), where the implicit map $\Spa C\to \Div^1_X$ corresponds to the untilt corresponding to $v$.
Also, let $F_v:=\Frac(\O_v)=B_{\Div^1_X}(\Spa C)$ and $\mathbb D_v^*:=\Spec F_v$($=\mathbb D^*_{C,v}$ in previous sections).
Note that $F\subset F_v$.
Following \cite[Section 1]{Ngo_2006}, we shall give in this section an adelic description of $\Higgs(C)$.

\begin{proposition}\label{prop:extend}
    Every $G$-torsor on $\Spec F$ can be extended to $X_{\Spa C}^{sch}$.
    In particular, we have a surjection $B(G) \twoheadrightarrow H^1(F,G)$, where $B(G)$ is the Kottwitz set of $G$ and the cohomology group on the right is the Galois cohomology.
\end{proposition}

\begin{proof}
    For every untilt $C^\sharp$ of $C$, $B_{dR}(C^\sharp) \cong C^{\sharp}((t))$ is a $C_1$-field by \cite[Section 2 Theorem 2]{C1-field}, so every $G$-bundle on it is trivial by \cite[Section 8.6]{Borel-Springer}.
    The proof in \cite[Lemma 1.1]{Ngo_2006} then gives the result.
\end{proof}

\begin{remark}
    In the case where $E$ is a finite extension of $\Q_p$, Fargues conjectured in \cite[Conjecture 3.10]{Fargues-G-torsors} that $F$ is a $C_1$-field. If this is true, then $H^1(F,G)=1$.
\end{remark}

For each $c\in H^1(F,G)$, we can by \Cref{prop:extend} pick a $G$-torsor $\E_c$ on $X_C^{sch}$ such that the cohomology class corresponding to $(\E_c)_\eta$ is $c$.
Let $$G_c:=\underline{\Aut}(\E_c)$$ be the sheaf on $(Sch/X_{C}^{sch})_{\fppf}$ given by $(U\to X_C^{sch})\mapsto \Aut_G(\E_c|_U)$.
Let $$\g_c$$ be the sheaf on $(Sch/X_{C}^{sch})_{\fppf}$ given by $\Spec R\mapsto \ker(G_c(R[\epsilon]) \to G_c(R))$, where $\epsilon^2=0$.

\begin{lemma}\label{g extend}
    If $\gamma\in \g_c(F)$, then there is an open subscheme $U\subset X_C^{sch}$ such that $\gamma$ extends to an element of $\g_c(U)$, and this extension is unique for any fixed $U$. The same holds for $\gamma'\in G_c(F)$.
\end{lemma}

The idea is very simple: objects finitely presented over $\Frac(B)$ are defined over $B[\frac{1}{b}]$ for some $b$.

\begin{proof}
    Let us prove the lemma for $\g_c(F)$. The proof for $G_c(F)$ is similar.

    Pick a non-empty affine open subscheme $U'=\Spec B\subset X_C^{sch}$. Then $\E_c|_{U'}$ is representable by a flat, finitely presented affine scheme $U=\Spec A$ over $U'$.
    Then $\gamma:\E_c|_{\Spec F[\epsilon]} \to \E_c|_{\Spec F[\epsilon]}$ corresponds to a $\Frac(B)[\epsilon]$-algebra homomorphism $A\otimes_B \Frac(B)[\epsilon] \to A\otimes_B \Frac(B)[\epsilon]$.
    As $A$ is finitely presented over $B$, there exists $b\in B\setminus\set{0}$ such that the ring homomorphism arises from a $B[\frac{1}{b}][\epsilon]$-algebra homomorphism 
    \begin{equation}\label{eq:A otimes B}
        A\otimes_B B[\frac{1}{b}][\epsilon] \to A\otimes_B B[\frac{1}{b}][\epsilon].
    \end{equation}
    This is $G$-equivariant.
    By the $B$-flatness of $A$, $A\otimes_B B[\frac{1}{b}] \hookrightarrow A\otimes_B \Frac{B}$ is injective.
    It follows that \eqref{eq:A otimes B} is $G$-equivariant,
    is in $\g_c(\Spec B[\frac{1}{b}])=\ker(\Aut_G(\E_c|_{B[\frac{1}{b}][\epsilon]}) \to \Aut_G(\E_c|_{B[\frac{1}{b}]})),$
    and is the unique extension of $\gamma$ to $\Spec B[\frac{1}{b}]$.
    The uniqueness of extension to general $U$ (such that $\gamma$ extends to an element of $\g_c(U)$) can be deduced by applying the above to an affine open $U'\subset U$ and take $b=1$.
\end{proof}

\begin{lemma}\label{ad of E_c is Lie}
    We have $\Ad(\E_c) \cong \g_c$.
\end{lemma}

\begin{proof}
    Let $U$ be an affine scheme over $X_C^{sch}$.
    Pick an affine fppf cover $V\to U$ and an isomorphism $\E_c|_{V}\xrightarrow{\alpha} G|_V$.
    Let $\pi_1,\pi_2:V\times_U V \to V$ be the projections to the left and right factors respectively. Let $t= \pi_2^*(\alpha) \circ \pi_1^*(\alpha)$.
    To give an element of $\Ad(\E_c)(U)$ is the same as giving an element of $\mathrm{eq}(\Ad(\E_c)(V) \rightrightarrows \Ad(\E_c)(V\times_U V))$, i.e. an element $v\in \g(V)$ such that $\Ad_{t} (\pi_1^* v) =\pi_2^*v$. The same description holds for $\g_c(U)$. This gives an isomorphism 
    \begin{equation}\label{desired isom}
        \Ad(\E_c)(U) \cong \g_c(U)
    \end{equation}
    By tracing through the definitions, we see that this is given by descending
    \begin{align*}
        \Ad(\E_c)(V)=(\E_c \times^G \g)(V) &\isomto \g_c(V)\\
        (e,v) &\mapsto \alpha^{-1} \circ \Ad_{\alpha(e)} (v) \circ \alpha,
    \end{align*}
    where $\g_c(V)$ is viewed as an element of $\Aut_G(\E_c|_{\O(V)[\epsilon]})$ and $\Ad_{\alpha(e)} (v)$ is viewed as an element of $\Aut_G(G|_{\O(V)[\epsilon]})$.
    It is clear from this expression that \eqref{desired isom} is independent of the choice of $\alpha$ and is compatible with passing to affine fppf cover $V'\to V$ of $V$, so \eqref{desired isom} is also independent of the choice of $V$. As $U$ varies, we get an isomorphism of sheaves.
\end{proof}

For each closed point $v$ of $X_C^{sch}$, we pick a trivialisation $$\iota_v:\E_c|_{\mathbb D_v}\isomto G|_{\mathbb D_v}$$ which exists by the proof of \Cref{J-bundle etale locally trivial}.
Then $\iota_v$ induces isomorphisms $G_c|_{\mathbb D_v}\isomto G|_{\mathbb D_v}$ and $\g_c|_{\mathbb D_v}\isomto \g|_{\mathbb D_v}$. 
We shall use these isomorphisms to view elements of $G_c(F_v)$ as elements of $G(F_v)$ and to view elements of $\g_c(F_v)$ as elements of $\g(F_v)$.

Note that $\g_c(F)\subset \g_c(F_v)$.
To see this, note that by definition $\g_c(F) \subset \Aut(\E_c|_{\Spec F[\epsilon]})$.
We know that $\E_c|_{\Spec F[\epsilon]}$ is representable by an affine flat $\Spec F[{\epsilon}]$-scheme $\Spec A$.
In particular, $A\hookrightarrow A\otimes_{F[\epsilon]} F_v[\epsilon]$, so $\Aut(A)\hookrightarrow \Aut(A\otimes_{F[\epsilon]} F_v[\epsilon])$. This implies $\g_c(F)\subset \g_c(F_v)$.
In particular, elements of $\g_c(F)$ can also be viewed as elements of $\g(F_v)$.

We have the following analogue of \cite[Corollaire 1.3]{Ngo_2006}, whose proof is omitted there. The relevant \cite[Proposition 1.2]{Ngo_2006} is proved there however.
\begin{proposition}
    We have an equivalence of categories between $\Higgs(C)$ and the groupoid of triples
    \[(c,\gamma,(g_v)_{\text{closed }v\in |X_C^{sch}|})\]
    where
    \begin{itemize}
        \item $c\in H^1(F,G)$
        \item $\gamma\in \g_c(F)$
        \item $g_v\in G(F_v)/G(\O_v)$ such that\footnotemark $\Ad_{g_v^{-1}} (\gamma) \in \g(\O_v)$ for every $v$ and $g_v$ is trivial for all but finitely many $v$.
    \end{itemize}
    \footnotetext{More precisely, $\Ad_{g_v^{-1}} (\iota_v \gamma \iota_v^{-1}) \in \g(\O_v)$.}
    The morphisms are given by
    \[\Mor((c,\gamma,(g_v)),(c',\gamma',(g'_v)))=
    \begin{cases}
        \set{h\in G_c(F): hg_v\in g'_v G(\O_v)  \text{ for\footnotemark all }v, \gamma'=\Ad_{h}(\gamma)} & \text{if } c=c' \\
        \varnothing & \text{if } c\neq c'.
      \end{cases}
    \]
    \footnotetext{More precisely, $\iota_v h\iota_v^{-1}g_v\in g'_v G(\O_v)$.}
\end{proposition}

\begin{proof}
    The proof is similar to how one expresses $\Bun_G(\F_q)$ as a double adelic quotient.
    Let us denote the groupoid of triples in the proposition as $\mathcal T$.

    We first construct a functor $f:\mathcal T \to \Higgs(C)$.
    Let $(c,\gamma,(g_v)_{\text{closed }v\in |X_C^{sch}|}) \in \mathcal T$.
    By \Cref{g extend}, we can pick a non-empty open subscheme $U$ of $X_C^{sch}$ such that $g_v$ is trivial for all $v\in U$ and $\gamma$ extends to an element of $\g_c(U)$, which we still denote as $\gamma$.
    By the Beauville-Laszlo lemma, we can glue $\E_c|_U$ with the trivial $G$-bundles $G|_{\mathbb D_v}$ for $v\in |X_C^{sch}|\setminus |U|$ via the isomorphism $$(\E_c|_U)|_{\mathbb D_v^*}\xrightarrow{g_v^{-1}\iota_v} (G|_{\mathbb D_v})|_{\mathbb D_v^*}$$ to get a $G$-bundle $\E$ on $X_C^{sch}$. This is well-defined because changing the representatives of $g_v\in G(F_v)/G(\O_v)$ gives another $G$-bundle on $X_C^{sch}$ with a canonical isomorphism to $\E$. (Alternatively, one can pick a representative for each object of $\mathcal T$.)
    By the Beauville-Laszlo lemma, the sections $\gamma\in \g_c(U)\cong \Ad(\E_c)(U)$ and $\Ad_{g_v^{-1}} (\iota_v(\gamma)) \in \g(\mathbb D_v)=\Ad(G|_{\mathbb D_v})(\mathbb D_v)$ glue to give a section $\phi\in H^0(X_C^{sch},\Ad(\E))$.
    Up to canonical isomorphism, $(\E,\phi)$ is independent of the choice of $U$. 
    This defines $f$ on objects.

    Let $h\in\Mor((c,\gamma,(g_v)),(c',\gamma',(g'_v)))$.
    Then $c=c'$. Fix representative for $g_v, g'_v$ in $G(\O_v)$, which we still denote as $g_v,g'_v$.
    By \Cref{g extend}, we can pick a non-empty open subscheme $W$ of $X_C^{sch}$ such that $g_v, g'_v$ are trivial for all $v\in W$ and $\gamma,\gamma'$ extend to elements of $\g_c(W)$ and $h$ extends to an element in $G_c(W).$
    By the Beauville-Laszlo lemma, the isomorphisms $h:(\E_c|_W,\gamma) \to (\E_c|_W,\gamma')$ and $g_v'^{-1}hg_v:(G|_{\mathbb D_v}, \Ad_{g_v^{-1}} (\gamma))\to (G|_{\mathbb D_v}, \Ad_{g_v^{'-1}} (\gamma'))$ glues to an isomorphism $(\E,\phi) \to (\E',\phi')$.
    This defines the functor $F$.

    It is easy to verify using the Beauville-Laszlo lemma that $F$ is fully faithful.
    It remains to show that $F$ is essentially surjective.
    Let $(\E,\phi)\in \Higgs(C)$.
    Pick an isomorphism $t_\eta:\E_\eta \isomto (\E_c)_{\eta}$ for some $c\in H^1(F,G)$.
    This extends to $t_\eta:\E|_U \isomto \E_c|_U$ for some non-empty open subscheme $U\subset X_C^{sch}$.
    For each $v\in |X_C^{sch}|\setminus |U|$, pick an isomorphism $t_v: \E|_{\mathbb D_v} \isomto G|_{\mathbb D_v}$.
    Then
    \begin{equation*}
        G|_{\mathbb D_v^*} \xrightarrow{t_v^{-1}} \E|_{\mathbb D_v^*} \xrightarrow{t_\eta} (\E_c)|_{\mathbb D_v^*} \xrightarrow{\iota_v} G|_{\mathbb D_v^*}
    \end{equation*}
    corresponds to an element $g_v\in G(F_v)$.
    For $v\in U$, set $g_v=1$.
    Via $t_\eta$ and \Cref{ad of E_c is Lie}, we get an isomorphism $\Ad(\E)|_U \isomto \Ad(\E_c)|_U \cong \g_c|_U$. 
    The image of $\phi|_U$ gives an element in $\g_c(U)$, which can be restricted to $\eta$ to give $\gamma\in \g_c(F)$.
    Then $F(c,\gamma,(g_v)) \cong (\E,\phi).$
\end{proof}

\begin{remark}
    In analogy with \cite[after Corollaire 1.3]{Ngo_2006}, this proposition roughly suggests that the cardinality of $\Higgs(C)$ as a groupoid is given by the orbital integrals
    \[\sum_{c\in H^1(F,G)} \sum_{\gamma\in \g_c(F)/\text{conj.}} \int_{G_\gamma(F)\backslash G(\mathbb A_F)} 1_{\prod_v \g(\O_v)} (\Ad_{g^{-1}} \gamma) dg   \]
    where
    \begin{itemize}
        \item $\mathbb A_F =\prod'_v(F_v,\O_v) \cong \prod'_v (C_v^\sharp((t)),C_v^\sharp\powerseries{t})$ is a restricted direct product and $C_v^\sharp$ is the untilt of $C$ corresponding to $v$
        \item $G_\gamma$ is the centraliser of $\gamma$ in $G$
        \item $1_{\prod_v \g(\O_v)}$ is the indicator function of $\prod_v \g(\O_v)$.
    \end{itemize}
    However, as stated, this equation does not make much sense, since no Haar measures on the groups have been specified; even after fixing such measures, both the cardinality of $\Higgs(C)$ and the sums involved should be infinite. We do not attempt here to make the connection with orbital integrals fully rigorous.
\end{remark}

\appendix
\section{\'Etale topology on sousperfectoid spaces}\label{sec:etale top}
In \Cref{rmk: map FF}, we showed that $\Higgs = \Map(FF,[\g_\L/G])$, where $[\g_\L/G]$ is viewed as an fppf stack on the category of schemes over $X_S^{sch}$. We also use various other fppf stacks in the main text.
In the appendix, we show that, we can instead use adic spaces. More precisely, we use stacks on the big \'etale site of sousperfectoid spaces.
The reason for this choice of site is that, we would like to consider $G$-bundles on the adic Fargues-Fontaine curve $X_S$.
We thus want to have a site that includes $X_S$ as an object. To get the correct notion of $G$-bundles on adic spaces, which are by definition \'etale locally trivial, we use the \'etale topology. We restrict to sousperfectoid spaces, rather than using all adic spaces, because the notion of $G$-bundle is particularly well-behaved on them \cite[Theorem 19.5.2]{Berkeley} and the fibre product of sousperfectoid spaces $X\times_Z Y$ is still a sousperfectoid space if $X\to Z$ is \'etale by \Cref{lem:prod}.

One advantage of using this formalism is that, we do not have to restrict to affinoid $S$ in defining v-stack whose $S$-valued points involve the Fargues-Fontaine curve.

Let $\sPerfd$ be the category of sousperfectoid spaces, which is introduced by Hansen-Kedlaya, cf. \cite[section 6.3]{Berkeley}, \cite[section 7]{HK-sousperfectoid}.
In this section, we will study some basic properties of the \'etale topology on $\sPerfd$. 

Recall that \'etale maps of (pre)adic spaces are defined in \cite[Definition 8.2.16]{KL_1}.

\begin{lemma}\label{lem:prod}
    Let $f: X\to Z, g: Y\to Z$ be two maps of sousperfectoid spaces with $f$ \'etale. Then the fibre product $X\times_Z Y$ exists in the category of adic spaces and is a sousperfectoid space. Moreover, we get a surjection of topological spaces
    \[|X\times_Z Y| \twoheadrightarrow |X|\times_{|Z|} |Y|.\]
\end{lemma}

\begin{proof}
    By the usual argument, this reduces to the case where $X, Y, Z$ are affinoid sousperfectoid and $f=f_1\circ f_2\circ f_3$, where $f_1,f_3$ are open immersions and $f_2$ is finite \'etale. Then the claim that $X\times_Z Y$ exists and is sousperfectoid follows from \cite[Proposition 6.3.3]{Berkeley}.

    For the last claim, one can use \cite[Lemma 9.1.3]{KL_1}. Alternatively, one can easily reduce the claim to the case where $X,Y,Z$ are affinoids and then apply \cite{valuations}.
\end{proof}

By this lemma and \cite[Lemma 8.2.17(c)]{KL_1}, we get the \'etale site on $\sPerfd$:
\begin{definition}We denote by
    \begin{itemize}
        \item $\sPerfd_{\et}$ the  big \'etale site, whose objects are sousperfectoid spaces and coverings are \'etale covers;
        \item $\sPerfd_{X,\et}$ the  big \'etale site over a sousperfectoid space $X$, whose objects are sousperfectoid spaces over $X$ and coverings are \'etale covers;
        \item $X_\et$ the small \'etale site over a sousperfectoid space $X$, whose objects are sousperfectoid spaces \'etale over $X$ and coverings are \'etale covers.
    \end{itemize}
    By an \'etale sheaf on $X$, we mean a sheaf on $X_\et$.
\end{definition}

We shall establish some basic properties of these sites, following the strategy of \cite[sections 8, 9]{EC_diamonds}.

\begin{lemma}\label{lem:sheaf}
    The presheaves $\O,\O^+:\sPerfd_\et^{op}\to \mathrm{Set}$ which send $X$ to $\O_X(X), \O_X^+(X)$ are sheaves.
\end{lemma}

\begin{proof}
    They are sheaves for the analytic topology, so it is enough to check the sheaf axioms for an \'etale cover of the form $\set{Y\to X}$ with $X$ affinoid sousperfectoid. This follows from \cite[Theorem 8.2.22(c)]{KL_1}.
\end{proof}

\begin{proposition}\label{subcanonical}
    The site $\sPerfd_\et$ is subcanonical, i.e. the presheaf represented by any $X\in\sPerfd$ is a sheaf.
\end{proposition}

\begin{proof}
    Let $X\in \sPerfd$ and $\mathcal F$ be the presheaf represented by $X$. Let $q:Z\twoheadrightarrow Y$ be an \'etale cover in $\sPerfd$.
    We want to show that 
    \[\mathcal F(Y)\iso eq(\mathcal F(Z) \rightrightarrows \mathcal F(Z\times_Y Z)).\]

    By \cite[Lemma 8.2.17(b)]{KL_1}, every \'etale morphism is open.\footnote{The lemma there directly implies that every \'etale morphism whose domain is quasi-compact is an open mapping. We can deduce the same result for a general \'etale morphism by writing the domain as a union of quasi-compact open subsets.}
    Being surjective and open, the map $|q|:|Z|\to |Y|$ on the underlying topological spaces is a quotient map.

    Let $f\in eq(\mathcal F(Z) \rightrightarrows \mathcal F(Z\times_Y Z))$.
    By surjectivity of $|Z\times_Y Z| \to |Z|\times_{|Y|} |Z|$, the two composite maps $|Z|\times_{|Y|} |Z|\rightrightarrows |Z| \xrightarrow{|f|} |X|$ agree.
    As $|q|$ is a quotient map, $|f|$ factors as 
\[\begin{tikzcd}
	{|Z|} && {|X|} \\
	& {|Y|}
	\arrow["{|f|}", from=1-1, to=1-3]
	\arrow["{|q|}"', from=1-1, to=2-2]
	\arrow[from=2-2, to=1-3]
\end{tikzcd}.\]
The problem can then be considered locally on $X$, so we can assume $X$ is affinoid. The corollary now follows from the fact that $\O$, $\O^+$ are \'etale sheaves by \Cref{lem:sheaf}.
\end{proof}

\begin{remark}\label{rmk: bundle}
    In \cite[Appendix to Lecture XIX]{Berkeley}, a cohomological $G$-torsor $\E$ on a sousperfectoid space $X$ is viewed as a sheaf on the small \'etale site of $X$. By \cite[Theorem 19.5.2]{Berkeley}, $\E$ is the \'etale sheaf of sections of a geometric $G$-torsor $\mathcal P$, which is sousperfectoid. By \Cref{subcanonical}, $\Hom_X(\;,\mathcal P)$ is a sheaf on the big \'etale site of $X$, whose restriction to $X_\et$ is $\E$. It follows that we can equivalently regard a cohomological $G$-torsor as a sheaf on the big \'etale site.
\end{remark} 

We have \'etale descent of vector bundles on sousperfectoid spaces:
\begin{lemma}\label{vb-stack}
    The prestack sending $X\in\sPerfd$ to the groupoid of vector bundles on $X$ is an \'etale stack. 
\end{lemma}

\begin{proof}
    This is a stack for the analytic topology, so it is enough to prove the axiom of stack for \'etale covering between affinoid sousperfectoid spaces.
    The lemma thus follows from the fact that for $X=\Spa(A,A^+)$ affinoid sousperfectoid, we have $\set{\text{finite locally free $\O_{X}$-modules}} \simeq \set{\text{finite projective }A\text{-modules}} \simeq \set{\text{finite locally free $\O_{X_{\et}}$-modules}}$ by \cite[Theorem 8.22]{KL_1}.
\end{proof}

\begin{lemma}
    Let $X\in\sPerfd$. We have equivalences of symmetric monoidal categories
    \[\{\text{vector bundles on }\sPerfd_{X,\et}\} \iso \{\text{vector bundles on }X_{\et}\} \iso \{\text{vector bundles on }X\}.\]
    By a vector bundle on $\sPerfd_{X,\et}$, we mean a sheaf $\mathcal F$ of $\O_{\sPerfd_{X,\et}}$-module on $\sPerfd_{X,\et}$ which is locally finite free, i.e. there is an \'etale cover $\{Y_i\to X\}$ of $X$ by sousperfectoid $Y_i$ such that $\mathcal F|_{\sPerfd_{,Y_i\et}} \cong \O_{\sPerfd_{,Y_i\et}}^{n_i}$ for some $n_i\in \Z_{\ge 0}$.
\end{lemma}

\begin{proof}
    Given a vector bundle on $\sPerfd_{X,\et}$, we can restrict it to get a vector bundle on $X_\et$. This, together with the obvious maps on morphisms, define a functor which we call $a$.

    Given a vector bundle on $X_\et$, we can restrict it to get a sheaf on $X$, which is a vector bundle by \Cref{vb-stack}.  We can also restrict morphisms. We call this functor $b$.

    Given a vector bundle $\mathcal V$ on $X$, we can form a sousperfectoid geometric vector bundle $\Tot(\mathcal V)\to X$ such that $\Hom_X(\;,\Tot(\mathcal V))=\mathcal V$ as sheaves on $X$. Then $\Hom_X(\;,\Tot(\mathcal V))$, regarded as a functor on $\sPerfd_{X,\et}^{op}$, is a sheaf on $\sPerfd_{X,\et}$ as $\sPerfd_{X,\et}$ is subcanonical by \Cref{subcanonical}. To endow it with an $\O_{\sPerfd_{X,\et}}$-module structure, we can work locally and assume $\mathcal V\cong \O_X^n$, so $\Tot(\mathcal V)=\mathbb A^n_X$ and $\Hom_X(\;,\Tot(\mathcal V))=\O_{\sPerfd_{X,\et}}^n$, which has a canonical $\O_{\sPerfd_{X,\et}}$-module structure. It is easy to see this is a vector bundle on $\sPerfd_{X,\et}$. This assignment, together with the obvious maps on morphisms, define a functor which we call $c$.

    It is not hard to show that $c\circ b\circ a, a\circ c\circ b, b\circ a\circ c$ are all equivalent to the identities.
    To enhance $a$ (and similarly for $b$) to a symmetric monoidal functor, we can use the universal property of sheafification to define for vector bundles $\E_1,\E_2$ on $\sPerfd_{X,\et}$ a map $a(\E_1)\otimes a(\E_2) \to a(\E_1\otimes E_2)$, which can be easily shown to be symmetric monoidal by working locally. 
    As $c$ is a quasi-inverse of $b\circ a$, it can also be enhanced to a symmetric monoidal functor.
\end{proof}

Let us state the following easy lemma, which relates two ways to twist a sheaf by a line bundle.
\begin{lemma}\label{twist_of_line_bundle}
    Let $X$ be a sousperfectoid space and $\mathcal L$ be a line bundle. Let $\tilde{\mathcal L}=\underline{Isom}_{\O_X}(\O_X,\L)$ be the associated $\mathbb G_m$-bundle, where $\mathbb G_m$ acts via $f*a:=af$. Suppose $\mathcal F$ is an $\O_X$-module. Then we have a canonical isomorphism of $\O_X$-modules
    \[\tilde\L\times^{\mathbb G_m}\mathcal F  \iso \L \otimes_{\O_X}\mathcal F\]
    where the $\O_X$-module structure on $\tilde\L \times^{\mathbb G_m} \mathcal F$ is induced\footnote{The addition on $\tilde\L \times^{\mathbb G_m} \mathcal F$ is obtained by sheafifying $(a,x)+(a,y)=(a,x+y)$.} from $\mathcal F$.
\end{lemma}

\begin{proof}
    Define the map as the sheafification of $(h,v)\mapsto h(1)\otimes v$.
    To check that this is an isomorphism, we can work locally on $X$, so we may without loss of generality assume $\L=\O_X$.
    Then $\tilde\L \times^{\mathbb G_m} \mathcal F =\mathcal F\iso \mathcal F=\L \otimes_{\O_X}\mathcal F.$
\end{proof}

As an application, we show that we can replace some uses of fppf stacks by \'etale stacks on sousperfectoid spaces. There are some clash of notations with the main text. We hope this is not too confusing.

We just give one example.
We can view $G$ as an \'etale sheaf on $\sPerfd_{X_S}$, whose $\Spa(R,R^+)$-valued points are $G(R)$.
To see that this is indeed an \'etale sheaf, note that $G$ is smooth over $\Spec E$, so its analytification\footnote{See \cite[Section 6]{foundational_adic_geometry} for the notion of analytification, which is a process that turns schemes into adic spaces.} over $\Spa E$ exists and is a sousperfectoid adic space. Thus, the presheaf it represents is an \'etale sheaf on $\sPerfd_{\Spa E}$ (and hence also on $\sPerfd_{X_S}$) by \Cref{subcanonical}.

Similarly, we can view $\g$ as an \'etale sheaf on $\sPerfd_{X_S}$, whose $\Spa(R,R^+)$-valued points are $\g(R)$.
We can then form $\g_\L:= \g\otimes_{\O_{X_S}}\L$, which is a vector bundle on $X_S$.
We can define the \'etale stack $[\g_\L/G]$ on $\sPerfd_{X_S}$.
Then $\Higgs = \Map(FF,[\g_\L/G])$, where $\Map(FF,[\mathfrak g_\L/G])$ denotes the prestack on $\Perf_{S}$ sending $T$ to the groupoid $[\mathfrak g_\L/G](X_T)$. Note that by 2-Yoneda lemma, this is also the prestack sending $T$ to the groupoid of morphisms $X_T\to[\g_\L/G]$ of stacks on $\sPerfd_{X_S}$.

\bibliographystyle{amsalpha}
\bibliography{ref}

\end{document}